\documentclass{article}
\usepackage[margin=1.0in]{geometry}
\usepackage{todonotes}

\usepackage{amsmath, amsthm, thmtools, amssymb, enumerate, multicol, mathrsfs}

\makeatletter
\newcommand{\hangpara}[2]{\hangindent#1\hangafter#2\noindent}
\newenvironment{hangparas}[2]{\setlength{\parindent}{\z@}
  \everypar={\hangpara{#1}{#2}}}{\par}
\makeatother

\usepackage{tikz}
\usetikzlibrary{calc}
\newcommand{\vertex}{
coordinate[draw, shape=circle, inner sep=2]
}

\newcommand{\graphnode}[3]{
\ifthenelse{\equal{#3}{left}}
{\path[ultra thick] #1 coordinate[draw, shape=circle, inner sep=2] (#2) node[#3, xshift=-3] {$#2$};}
{	
	\ifthenelse{\equal{#3}{right}}
	{\path[ultra thick] #1 coordinate[draw, shape=circle, inner sep=2] (#2) node[#3, xshift=3] {$#2$};}
	{     	
		\ifthenelse{\equal{#3}{above}}
		{\path[ultra thick] #1 coordinate[draw, shape=circle, inner sep=2] (#2) node[#3, yshift=3] {$#2$};}
		{  	
			\ifthenelse{\equal{below}{below}}
			{\path[ultra thick] #1 coordinate[draw, shape=circle, inner sep=2] (#2) node[#3, yshift=-3] {$#2$};}
			{{\path[ultra thick] #1 coordinate[draw, shape=circle, inner sep=2] (#2) node[#3] {$#2$};}}
		}
	}	
}
}

\definecolor{mixed}{HTML}{E4E4C9}

\theoremstyle{plain}
\newtheorem{Thm}{Theorem}[section]
\newtheorem{Cor}[Thm]{Corollary}
\newtheorem{Prop}[Thm]{Proposition}

\theoremstyle{definition}
\newtheorem{Exam}[Thm]{Example}

\newcommand{\thmref}[1]{Theorem \ref{#1}}
\newcommand{\corref}[1]{Corollary \ref{#1}}
\newcommand{\propref}[1]{Proposition \ref{#1}}
\newcommand{\examref}[1]{Example \ref{#1}}

\renewcommand{\S}{\mathcal{S}}

\newcommand{\T}{\mathcal{T}}
\newcommand{\Q}{\mathbb{Q}}
\newcommand{\Z}{\mathbb{Z}}
\newcommand{\N}{\mathbb{N}}
\newcommand{\QG}{\Q[G]}
\newcommand{\QS}{\Q[\S]}
\newcommand{\QT}{\Q[\T]}
\newcommand{\I}{\mathscr{I}} 
\newcommand{\B}{\mathscr{B}}
\renewcommand{\H}{\mathscr{H}} 
\renewcommand{\O}{\mathrm{O}}
\newcommand{\K}{\mathrm{K}}
\newcommand{\LO}{\mathrm{LO}}
\newcommand{\RO}{\mathrm{RO}}
\newcommand{\ORO}{\mathrm{ORO}}
\newcommand{\OLO}{\mathrm{OLO}}
\newcommand{\OROP}{\mathrm{OROP}}
\newcommand{\LOZ}{\mathrm{LOZ}}
\newcommand{\LOO}{\mathrm{LOO}}
\newcommand{\LORO}{\mathrm{LORO}}
\newcommand{\CH}{\mathrm{CH}}
\newcommand{\z}{\mathrm{Z}}
\newcommand{\OO}{\mathrm{OO}}
\newcommand{\OCH}{\mathrm{OCH}}
\newcommand{\OZ}{\mathrm{OZ}}
\newcommand{\OG}{\mathrm{OG}}
\newcommand{\s}{\ \mathfrak{s}\ }
\newcommand{\shat}{\ \hat{\mathfrak{s}}_\varphi\ }
\renewcommand{\t}{\ \mathfrak{t}\ }
\renewcommand{\u}{\ \mathfrak{u}\ }
\renewcommand{\r}[1]{\ \mathfrak{r}_{#1}\ }
\newcommand{\rhat}[1]{\ \widehat{\mathfrak{r}_{#1}}\ }
\DeclareMathOperator*{\Aut}{Aut}
\DeclareMathOperator*{\Span}{Span}

\title{On Schur Rings Over Semigroups}
\author{Joseph E. Marrow\footnote{Joseph E. Marrow, Southern Utah University, josephmarrow@suu.edu},  Andrew Misseldine\footnote{Corresponding Author: Andrew Misseldine, Southern Utah University, andrewmisseldine@suu.edu, (telephone) 435-865-8228}}

\date{\today}

\begin{document}
\maketitle

\begin{abstract}
We generalize the idea of a Schur ring of a group to the category of semigroups. Fundamental results of Schur rings over groups are shown to be true for Schur rings over semigroups. Examples where Schur rings differ between the two categories are provided. We prove some results for Schur rings over specific families of semigroups. We consider parallels between  semigroup extensions and their Schur rings. We fully enumerate the Schur rings for all semigroups of orders 0 -- 7, and some statistical analysis is performed. 
\end{abstract}


\\

Schur rings are partitions of groups that themselves behave like subgroups. More specifically, Schur rings are combinatorial subrings of the group algebra afforded by partitions of the group. With this perspective on Schur rings, this paper considers the question of whether one can consider a Schur ring over a semigroup, that is, to investigate those partitions of a semigroup that behave like subsemigroups. 

Let $G$ be a set, and let $\S$ be a partition over $G$. We call the elements of $\S$ the \emph{primitive sets} of $\S$ or the \emph{$\S$-classes}. All partitions in this paper are required to have finite primitive sets, even if $G$ is infinite.\footnote{In order for Schur rings over infinite groups to be embedded into the group algebra, as discussed in \cite{InfiniteI}, primitive sets must be finite as to avoid infinite sums.} For $x\in G$, let $X$ denote the unique primitive set in $\S$ containing $x$, that is, we will generally denote elements of $G$ using lower-case Roman letters, such as $x$, and the primitive set in $\S$ containing it by its corresponding upper-case Roman letter, such as $X$. We follow a similar scheme for Greek letters, e.g., the primitive set containing $\theta$ is $\Theta$. For any subset $C\subseteq G$, we say that $C$ is an \emph{$\S$-subset} if $C$ is the union of $\S$-classes.  Let $H, K \subseteq G$, and let $\S$ and $\T$ be partitions of $H$ and $K$, respectively. If all the primitive sets of $\T$ can be written as unions of primitive sets of $\S$ we call $\T$ a \emph{coarsening} of $\S$ and $\S$ a \emph{refinement} of $\T$. Let $\S\wedge \T$ be the \emph{common coarsening} of $\S$ and $\T$, which is the largest partition over $H\cap K$ which is a coarsening of both $\S$ and $\T$. Similarly, let $\S\vee \T$ be the \emph{common refinement} of $\S$ and $\T$, which is the smallest partition over $H\cup K$ which is a refinement of both $\S$ and $\T$. 

Consider now a group $G$ and a partition $\S$ over $G$. Let $X\in \S$. Then $X^* = \{x^{-1}\mid x\in X\}$. Let $\QG$ denote the group ring over $G$ with rational coefficients. For each finite subset $X\subseteq G$, we may identify $X$ with an element of the group ring $\QG$, namely the \emph{simple quantity} $\sum_{x\in X} x\in \QG$. By abuse of notation, we will denote this simple quantity in $\QG$ as $X$ itself. As such, we will use $x$ and $\{x\}$ interchangeably for singletons. When we discuss products of subsets, say $XY = \{xy\mid x\in X,y\in Y\}$, this is always understood as a multiset, where specific elements have a multiplicity within $XY$, so that the multiset $XY$ coincides with the product of simple quantities within $\QG$. In the few instances where we do consider only a product of sets, we will indicate that multiplicities are being ignored.

Let $\QS$ denote the subspace of $\QG$ spanned by the simple quantities in $\S$, that is, $\QS = \Span_\Q(\S) = \Span_\Q\{X\mid X\in \S\}$.  We say $\S$ is a \emph{Schur ring} over the group $G$ if:
\begin{enumerate}
    \item\label{multiply} $\forall X,Y\in \S$, $XY\in \QS$;
    \item\label{identity} $1\in \QS$;
    \item\label{inverse} $\forall X\in \S$, $X^*\in \QS$
\end{enumerate}
The term Schur ring (or S-ring) was coined by Helmut Wielandt \cite{Wielandt64} and originally developed by Issai Schur \cite{Schur33} and Wielandt \cite{Wielandt49}. Schur rings were originally developed as an alternative to characters within representation theory but have become useful tools within algebraic combinatorics, see \cite{KlinPoschel}. 

By definition, $\QS$ is a subring of $\QG$. A quick examination of the axioms of a Schur ring shows parallels between the axioms of a subgroup, that is, the partition associated to a Schur ring is closed under multiplication, identity, and inverses. This observation leads to a simple, useful, but less precise notion of a Schur ring, that is, Schur rings are those partitions of $G$ which behave like subgroups. Under such a perspective, \emph{does it necessarily need to be a group?}\footnote{In some ways, this question has already been considered. Mathematicians have previously studied partitions $\S$ over a group $G$ satisfying axioms \ref{multiply} and \ref{inverse}, from above, which we will refer to as a \emph{pre-Schur ring} \cite{MePhD}. It is well known that the primitive set that contains the identity in a pre-Schur ring is a subgroup of $G$. Furthermore, every primitive set is a union of double cosets of this same subgroup. In other words, each pre-Schur ring is a coarsening of the double coset partition. Refining the partition associated to a pre-Schur ring by separating off the identity as a singleton always forms a Schur ring over $G$ (see Proposition \ref{prop:monoidschur}). Hence, pre-Schur rings generalize the definition of a Schur ring by relaxing axiom \ref{identity}, but, on the other hand, all the pre-Schur rings can essentially be found already within the set of Schur rings. Using the language of this paper, a pre-Schur ring is a Schur ring that is closed under multiplication and inversion, that is, a Schur ring over an inverse semigroup. This correspondence shows that no truly new Schur rings are obtained from a group when its structure is relaxed to the ambient inverse semigroup structure of the group.}

 As a motivating example, consider \label{sym}$S_3$, the symmetric group on three letters. A Schur ring over $S_3$, when viewed as a group, must be closed under multiplication, identity, and inversion. There are ten such Schur rings, listed below:
\begin{multicols}{2}
\begin{enumerate}
    \item $\{1, (123)+(132) +(12)+ (13)+(23)\}$,
    \item $\{1, (12), (13)+(23)+(123)+(132)\}$,
    \item $\{1, (13), (12)+(23)+(123)+(132)\}$,
    \item $\{1, (23), (12)+(13)+(123)+(132)\}$,
    \item $\{1, (123)+(132), (12)+(13)+(23)\}$,
    \item $\{1, (123), (132), (12)+(13)+(23)\}$,
    \item $\{1, (12), (123)+(132), (13)+(23)\}$,
    \item $\{1, (13), (123)+(132), (12)+(23)\}$,
    \item $\{1, (23), (123)+(132), (12)+(13)\}$,
    \item $\{1, (123), (132), (12), (13), (23)\}$.
\end{enumerate}
\end{multicols}

In this paper, we explore the notion of a Schur ring over a semigroup,\footnote{In \cite{Tamaschke70}, Tamaschke introduces the notion of a Schur semigroup, which was an early attempt to extend Schur rings toward infinite objects, as Wielandt's original definition was restricted to finite groups. Tamaschke's Schur semigroups are not the semigroup Schur rings that we discuss here. Speaking of infinite groups, this definition of semigroup Schur ring, following the definition of infinite group Schur ring from \cite{InfiniteI}, allows the semigroup to be finite or infinite.} a natural generalization of groups. Let $G$ be a semigroup. Then we may generalize the above terminology and notation related to Schur rings. Let $\S$ be a partition over $G$. We say that $\QS$ is a \emph{Schur ring} over $G$ if $\QS$ is a subring of the semigroup ring $\QG$. Note that for a subset of $G$ to be a subsemigroup, we only require that it be closed under multiplication, as notions of identity or inverses might be absent over $G$. As such, a Schur ring over the semigroup $G$ is a relaxation of the notion of a Schur ring over a group for which we only retain the first axiom.

Now, we see that there exist additional Schur rings over $S_3$ when viewed as a semigroup. Two such examples are given as:
\[\{1+(12)+(13)+(23)+(123)+(132)\},\quad \{(12)+(123), 1+(132)+(13)+(23)\}. \] In total there are $45$ semigroup Schur rings over $S_3$. This includes the $10$ inherited from the group structure, and shows that there are $35$ additional Schur rings over $S_3$ when viewed as a semigroup.

This generalization to a category with less algebraic structure allows potentially more partitions of a semigroup to yield Schur rings than comparative groups. For example, some semigroups have sufficiently high symmetry to admit all partitions as Schur rings, such as in \thmref{thm:leftnullschur}. The \emph{Bell number} $\B(n)$ denotes the number of partitions over a set of $n$ objects. 
For any semigroup $G$ of order $n$, $\B(n)$ is an upper bound to the number of Schur rings over $G$. As alluded here, there do exist semigroups that obtain this bound, although $\z_2$, the cyclic group of order $2$, is the only group that obtains this bound; see \thmref{leftnullconverse}. On the other hand, other semigroups will have rigid structures which heavily regulate possible Schur rings, such as requiring certain elements only appear in singletons, as in \thmref{thm:nullschur}.  Semigroups which have structures allowing every partition to be admissible, or every partition of a subset to be admissible, cause the Bell numbers to appear frequently when counting the number of semigroup Schur rings. As shown in Table \ref{fig:sgsr4}, only $22$ of the $126$ semigroups of order $4$ have a number of semigroup Schur rings which is not a Bell number. 

We can define analogous Schur rings for other group generalizations, namely a monoid Schur ring is a partition closed under multiplication and identity, and an inverse (semigroup) Schur ring is a partition closed under multiplication and inverses. Clearly, every group Schur ring is both a monoid and inverse Schur ring and all are semigroup Schur rings. In Table \ref{tab:groupsemigroup} we provide the enumeration of Schur rings for each group of order at most $15$ when considered as a group, monoid, and semigroup. We take the values of the numbers of Schur rings attained as a group from \cite{MePhD,Ziv14} and \cite[A384156]{semicount}. We point out that for a fixed group
 \[\text{Number of group Schur rings} \leq \text{Number of monoid Schur rings} \leq \text{Number of semigroup Schur rings}\]
 where the second inequality is strict if the order of the group is not $1$.\label{kl}\label{Cg}

\begin{table}[!ht]
\begin{center}
\begin{tabular}{cccc|cccc|cccc|cccc}
 & G & M & S &  & G & M & S &  & G & M & S & & G & M & S\\
\hline
$\z_1$ & 1 & 1 & 1  & $S_3$ & 10 & 22 & 45    & $\z_9$ & 7 & 7 & 10 & $A_4$ & 52 & 266 & 482\\
$\z_2$ & 1 & 1 & 2  & $\z_7$ & 4 & 4 & 5      & $E_9$ & 40 & 40 & 49 & $D_6$ & 120 & 324 & 532\\
$\z_3$ & 2 & 2 & 3  & $\z_8$ & 10 & 10 & 15   & $\z_{10}$ & 10 & 10 & 15 & $\text{Dic}_{12}$ & 54 & 90 & 138\\
$\z_4$ & 3 & 3 & 5  & $E_8$ & 100 & 100 & 143 & $D_5$ & 25 & 95 & 197 & $\z_{13}$ & 6 & 6 & 7\\
$V_4$ & 5 & 5 & 9 & $\z_4\times \z_2$ & 28 & 28 & 43 & $\z_{11}$ & 4 & 4 & 5 & $\z_{14}$ & 13 & 13 & 19\\
$\z_5$ & 3 & 3 & 4  & $D_4$ & 34 & 66 & 107   & $\z_{12}$ & 32 & 32 & 46 & $D_7$ & 55 & 237 & 505\\
$\z_6$ & 7 & 7 & 11 & $Q_8$ & 25 & 26 & 35    & $\z_6\times\z_2$ & 76 & 76 & 108 & $\z_{15}$ & 21 & 21 & 27
\end{tabular}
\caption{Enumeration of Schur rings for groups of order at most $15$, considered as groups, monoids, and semigroups.}\label{tab:groupsemigroup}
\end{center}
\end{table}

In Section \ref{Schurring}, we explore fundamental concepts of semigroup Schur rings, generalizing many properties from the category of groups, and classify Schur rings over a few families of semigroups. In Section \ref{sec:trad} we discuss the four traditional Schur rings and how they adapt to the category of semigroups. In Section \ref{indecomposable}, we explore how indecomposable elements of a semigroup influence the structure of Schur rings. In Section \ref{sec:roster}, we explore one element extensions of semigroups and how these extensions likewise influence the Schur rings over these semigroups. In Section \ref{gap}, using GAP, we enumerate all Schur rings over semigroups of order $0-7$. Lastly, in Section \ref{sec:questions} we postulate questions of future research.

\section{Semigroup Schur Rings}\label{Schurring}
In this section, we consider and generalize fundamental notions of Schur rings over groups to the category of semigroups, as well as consider novel constructions of Schur rings over semigroups which have no analog for groups. We also classify the set of Schur rings over certain fundamental semigroup families. 

Let $\S$ be a Schur ring over a semigroup $G$.  A subsemigroup $H\le G$ is called an \emph{$\S$-subsemigroup} if $H$ is also an $\S$-subset. Likewise, an ideal $I\trianglelefteq G$ is called an \emph{$\S$-ideal} if $I$ is also an $\S$-subset. Analogous definitions hold for \emph{left $\S$-ideals} and \emph{right $\S$-ideals}.

The following properties of group Schur rings carry over to semigroup Schur rings immediately.

\begin{Prop} Let $G$ be a semigroup and $\S$ is a Schur ring over $G$. Let $X$ be an $\S$-subset. Then the semigroup generated by $X$, namely $\langle X\rangle$, is an $\S$-subsemigroup. Likewise, the ideal generated by $X$, namely $(X)$, is an $\S$-ideal.
\end{Prop}
\begin{proof}
    As $X$ is an $\S$-subset and $\S$ is closed under multiplication, it follows by induction that $X^n$ is an $\S$-subset for each $n$. Since $\langle X\rangle = \bigcup_{n=1}^\infty X^n$, the subsemigroup generated by $X$ is likewise an $\S$-subset. Likewise, note that $(X) = GXG = \bigcup_{C,D\in \S} CXD$, where each $CXD$ is an $\S$-subset.
\end{proof}

\begin{Prop}\label{prop:subring} Let $G$ be a semigroup and $\S$ is a Schur ring over $G$. Let $H$ be an $\S$-subsemigroup. Then the set of primitive $\S$-sets contained within $H$ forms a Schur ring over $H$, called a Schur subring of $\S$, which we denote as $\S|_H$.
\end{Prop}
\begin{proof}
    As $H$ is an $\S$-subsemigroup, the restriction of $\S$ to only those primitive sets inside $H$, namely $\S|_H$, forms a partition of $H$ itself. Let $X,Y\in \S|_H$. Clearly, $XY\subseteq H$, but also $XY = \sum_{C\in \S} \lambda_CC$. Those primitive sets with nonzero multiplicities appearing in $XY$ are necessarily contained within $H$. This implies that $\S|_H$ is closed under multiplication, that is, $\S|_H$ is a Schur ring over $H$.
\end{proof}

\begin{Prop}\label{prop:singleton} Let $G$ be a semigroup and $\S$ is a Schur ring over $G$. Let $H$ be the subset of $G$ consisting of all singletons in $\S$. Then $H$ is a subsemigroup of $G$.    
\end{Prop}
\begin{proof}
    Let $x,y\in H$. Then $x$, $y\in \S$. As $\S$ is closed under multiplication, the product $xy$ is a linear combination of primitive $\S$-sets, but in this combination the only nonzero multiplicity is the singleton $xy$ itself. Hence, the product $xy$ is a singleton in $\S$, which proves the claim.
\end{proof}

If $G$ is a semigroup, then let \label{Pageop}$G^\text{op}$ denote the \emph{opposite semigroup} of $G$, that is, the semigroup formed by reversing the order of multiplication. In the category of groups, it holds that $G\cong G^{\text{op}}$. In the category of semigroups, a semigroup and is opposite are generally not isomorphic. Of course, there is a natural anti-isomorphism between a semigroup and its opposite. We say two semigroups are \emph{equivalent} if they are isomorphic or anti-isomorphic.

\begin{Prop} Let $G$ be a semigroup and $\S$ is a Schur ring over $G$. Then $\S$ is a Schur ring over $G^\text{op}$.
\end{Prop}
\begin{proof}
    Let juxtaposition denote the multiplication in $G$ and $\cdot$ denote the multiplication in $G^\text{op}$. Hence, $x\cdot y = yx$. Let $X,Y\in \S$ be primitive sets. Then $X\cdot Y = YX = \sum_{C\in \S}\lambda_CC$, where the second equality holds since $\S$ is a Schur ring over $G$. As $X\cdot Y$ is a linear combination of primitive $\S$-sets, $\S$ is closed under the multiplication of $G^\text{op}$. Therefore, $\S$ is a Schur ring over $G^\text{op}$.
\end{proof}

As the Schur rings over $G$ and $G^\text{op}$ are identical, when studying Schur rings over a semigroup, it suffices to study semigroups up to equivalence.

For a semigroup $G$ with $\theta\in G$, we say $\theta$ is a \emph{left-zero} if $\theta x=\theta$ for all $x\in G$ and $\theta$ is a \emph{right-zero} if $x\theta = \theta$ for all $x\in G$. The semigroup $G$ has a \emph{zero element} $\theta$ if $\theta$ is both a left- and right-zero. A \emph{zero-semigroup} is a semigroup with a zero element $\theta$. Let $e\in G$. We say $e$ is a \emph{left-identity} if $ex=x$ for all $x\in G$ and $e$ is a \emph{right-identity} if $xe = x$ for all $x\in G$. The semigroup $G$ has an \emph{identity} $e$ if $e$ is both a left- and right-identity. A \emph{monoid} is a semigroup with an identity $e$. 

\begin{Prop}\label{prop:monoidschur} Let $G$ be a monoid with identity $e$, and let $\S$ be a Schur ring over $G$. Then the partition $(\S - \{E\}) \cup \{e, E - e\}$ affords a Schur ring over $G$.
\end{Prop}

Essentially, one may turn any semigroup Schur ring over a monoid into a monoid Schur ring by separating the identity.

\begin{proof}
    Let $\T = (\S - \{E\}) \cup \{e, E - e\}$. Since $XY\in \QT$ for all $X,Y\in \S$, we note that every $\S$-subset is necessarily a $\T$-subset. Likewise, every primitive $\S$-set is a primitive $\T$-set, except $E$. In that case, $E-e$ and $e$ are primitive $\T$-sets. We note $(E-e)(E-e) = E^2 - 2E + e\in \QT$. If $C\in \S$ such that $C\neq E$, then $(E-e)C = EC-C\in \QT$. Likewise, $C(E-e)\in \QT$. Therefore, $\QT$ is a subring of $\QG$.
\end{proof}

\begin{Thm}\label{thm:zeroschur} Let $G$ be a zero-semigroup, with zero $\theta$, and let $\S$ be a Schur ring over $G$. Then $\theta\in \S$. 
\end{Thm}

By analog to monoid Schur rings defined above, we can define a zero(-semigroup) Schur ring as a partition closed under multiplication and zero. This result shows that all the semigroup Schur rings over a zero-semigroup are necessarily zero Schur rings.

\begin{proof}
    Let $\Theta = \theta + X$ for some $X\subseteq G$. 
\[\Theta^2 = (\theta + X)(\theta + X) = \theta + \theta X+X\theta + X^2 = (2|X|+1)\theta + X^2.\] 
    Then the multiplicity of $\theta$ in $\Theta^2$ is at least $2|X|+1$. Therefore, every element of $\Theta$ appears in $\Theta^2$ with at least that same multiplicity. This accounts for at least $(2|X|+1)(|X|+1)$ products in $\Theta^2$. Of course, $|\Theta^2| = (|X|+1)^2$. This implies $(2|X|+1)(|X|+1) \le (|X|+1)^2$, which is only possible if $|X|=0$.
\end{proof}

The \label{null}\emph{null semigroup} $\O_n$ is the zero-semigroup of order $n$ such that $xy = \theta$ for all $x,y\in G$.

\begin{Thm}\label{thm:nullschur} Let $\S$ be a partition over the null semigroup $\text{O}_n$, for some positive integer $n$. Then $\S$ is a Schur ring if and only if $\theta\in \S$. The number of Schur rings over $\text{O}_n$ is the Bell number $\B(n-1)$.
\end{Thm}
\begin{proof}
Supposing that $\S$ is  a Schur ring, we take the primitive set $X\in\S$ and note that $X^2 = |X|^2 \theta\in \QS$. Then $\theta$ must be a primitive set. Suppose instead that $\theta\in \S$. If $X,Y\in \S$ are primitive sets, then $XY = |X||Y|\theta$. Hence, $\QS$ is closed under multiplication. As $\theta\in \S$ is the only restriction on the partition of $\O_n$, the number of Schur ring is the number of partitions which can be placed on the remaining elements, namely $\B(n-1)$.
\end{proof}

The \label{lnull}\emph{left-null semigroup} $\LO_n$ is the semigroup of order $n$ such that $xy=x$ for all $x,y\in \LO_n$. Every element of the left-null semigroup is a left-zero and a right-identity. The \label{rnull}\emph{right-null semigroup} $\RO_n$ is defined as $\RO_n = \LO_n^\text{op}$.

\begin{Thm}\label{thm:leftnullschur} Let $\S$ be a partition over the left-null semigroup $\LO_n$, for some positive integer $n$. Then $\S$ is a Schur ring. The number of Schur rings over $\LO_n$ is $\B(n)$.
\end{Thm}
\begin{proof}
Let $X,Y\in \S$, such that $X=x_1+x_2+\ldots+x_k$ for some $k\in \N$. Then
\begin{multline*}
    XY = (x_1+x_2+\ldots+x_k)Y = x_1Y+x_2Y+\ldots+x_kY = |Y|x_1+|Y|x_2+\ldots+|Y|x_k\\ 
    = |Y|(x_1+x_2+\ldots+x_k) = |Y|X\in \QS.
\end{multline*}  Therefore, $\QS$ is a Schur ring.
\end{proof}

The converse of \thmref{thm:leftnullschur} also holds for almost all semigroups. The cyclic group $\z_2$, viewed as a semigroup, has exactly two Schur rings.

\begin{Thm}\label{leftnullconverse} Let $G$ be a finite semigroup of order $n\neq 2$. Then $G$ is equivalent to $\LO_n$ if and only if the number of Schur rings over $G$ is $\B(n)$.
\end{Thm}
\begin{proof}
   Necessity is already handled by \thmref{thm:leftnullschur}. So we assume that $G$ has $\B(n)$ many Schur rings. Based upon the enumeration conducted in Section \ref{gap}, we may assume that $|G|\ge 5$.
        
    As there are exactly $\B(n)$ partitions over $G$, every partition over $G$ is a Schur ring. Let $H\subseteq G$ such that $|H|\le n-2$, that is, $|G - H| \ge 2$. Then $\S = H\cup \{G - H\}$ forms a partition over $G$, more specifically, the partition on $H$ is discrete while its complement is a single primitive set. By assumption, $\S$ is a Schur ring over $G$. By \propref{prop:singleton}, $H$ is an $\S$-subsemigroup, that is, $H\le G$. Hence, all subsets of $G$, except maybe those of cardinality $n-1$, are subsemigroups of $G$.

    Continuing, let $\T$ be a partition of $H$. Let $\S = \T\cup \{G - H\}$. Then $\S$ is a Schur ring over $G$. Likewise, $\T=\S|_H$ is a Schur ring over $H$, by \propref{prop:subring}. Therefore, every partition of $H$ is a Schur ring over $H$ for any subsemigroup of $G$, when $|H|\le n-2$. 

    Let $z,y,z\in G$. Then $H=\{x,y,z\}$ is a subsemigroup of $G$ of order $3$ for which every partition is a Schur ring over $H$. The only order $3$ semigroups with this property are equivalent to $\LO_3$, by Table \ref{fig:sgsr3}. Without the loss of generality, we may suppose that $xy = x$. Then $H\cong \LO_3$, and $yz = yy = y$, and $zy = zz = z$. As $y,z$ are arbitrary, this shows that multiplication in $G$ is projection onto the first factor, that is, $G\cong \LO_n$.
\end{proof}

Given a semigroup $G$, a monoid can be formed by adjoining a new element $e$ to $G$, 
that is, $xe=ex=x$ for all $x\in G$. Let \label{monoid}$G^e=G\cup e$ be this monoid. Note that if $G$ is already a monoid, then this new identity $e$ supersedes the old identity in $G$. Similarly, a zero-semigroup can be formed adjoining a new zero element $\theta$ to $G$, that is, $x\theta  = \theta x = \theta$ for all $x\in G$. We denote this by \label{theta}$G^\theta = G\cup \theta$. Again, if $G$ is already a zero-semigroup, then adjoining $\theta$ adds a new zero to $G$ which supersedes the old zero. These two extensions of semigroups have direct consequences for Schur rings.

\begin{Thm}\label{thm:monoidSchur} Let $G$ be a semigroup and let $\S$ be a Schur ring over $G^e$. Then $e\in \S$. It follows that the Schur rings over $G^e$ are in one-to-one correspondence with the Schur rings over $G$. 
\end{Thm}

Essentially, all the Schur rings over $G$ are the same as the Schur rings over $G^e$.

\begin{proof}
    Since $e\in E^2$, we have $E\subseteq E^2$. Suppose that $x\in E$ such that $x\neq e$. Then $x$ appears in $E^2$ with a multiplicity of at least $2$, since $x = ex=xe$. On the other hand, the multiplicity of $e$ in $E^2$ is exactly $1$ since the only factorization of $e$ is $e=e^2$. Since $e$ and $x$ have different multiplicities in $E^2$, it must be that $x\notin E$, a contradiction. Therefore, $E=e$, that is, $e\in \S$.

    Let $\S$ be a Schur ring over $G$. Define $\S^e = \S\cup \{e\}$, which is a partition over $G^e$. For all $X,Y\in \S$, we have $XY\in \QS$, which implies that $XY\in \Q[\S^e]$. Likewise, $eX=Xe=X\in \QS \le \Q[\S^e]$. Thus, $\Q[\S^e]$ is a Schur ring. Conversely, if $\T$ is a Schur ring over $G^e$, then $e\in \T$. Thus, $G$ is a $\T$-subsemigroup. Hence, $\T|_G$ is a Schur ring over $G$. As these two mappings between Schur rings over $G$ and $G^e$, namely $\S\mapsto \S^e$ and $\T\mapsto \T|_G$, are inverses they establish a one-to-one correspondence. 
\end{proof}

By the same argument used in the proof of Theorem \ref{thm:monoidSchur}, the following corollary is immediate from Theorem \ref{thm:zeroschur}. 

\begin{Cor}\label{cor:thetaschur} Let $G$ be a semigroup. Then the Schur rings over $G^\theta$ are in one-to-one correspondence with the Schur rings of $G$.
\end{Cor}

\begin{multicols}{2}
Let $\K_{m,n}$ denote the \emph{complete bipartite graph} with block sizes $m$ and $n$. The graph $\K_{1,n} = \{\theta\} \cup \{1,2, \ldots, n\}$ can be given a natural associative multiplication where, for all $x,y\in \K_{1,n}$, $x^2=x$ and $xy=yx=\theta$ whenever $x\neq y$. This makes \label{bipart}$\K_{1,n}$ into a semilattice. Despite all elements being idempotent, $\K_{1,n}$ behaves algebraically very similar to $\O_{n+1}$. This is manifested by their Schur rings. \columnbreak

\begin{center}
\begin{tikzpicture}
\def\h{1}
\node (0) at (0, 0) {$\theta$};
\node (2) at (-0.5, -\h) {$2$};
\node (1) at (-1.5, -\h) {$1$};
\node (3) at (0.5, -\h) {$3$};
\node (4) at (1.5, -\h) {$4$};

\draw[-] (1) -- (0) -- (2);
\draw[-] (3) -- (0) -- (4);
\end{tikzpicture}\\[0.5cm]
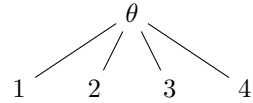
The complete bipartite graph $\K_{1,4}$
\end{center}
\end{multicols}

\begin{Thm}\label{thm:bipartiteschur}
Let $\S$ be a partition over the semilattice $\K_{1,n}$ for positive integer $n$. Then $\S$ is a Schur ring if and only if $\theta\in \S$. The number of Schur rings over $\K_{1,n}$ is the Bell number $\B(n)$.
\end{Thm}
\begin{proof}
    As $\K_{1,n}$ is a zero-semigroup, by \thmref{thm:zeroschur}, $\theta$ is a primitive set of $\S$. Let $X,Y\subseteq \K_{1,n} - \theta$ such that $X\cap Y=\emptyset$. Then $X^2 = X + |X|(|X|-1)\theta$ and $XY = |X||Y|\theta$. Therefore, any partition of $\K_{1,n}$ containing $\theta$ is a Schur ring. 
\end{proof}

Suppose that $G$ is a semigroup with ideal $I$. Then the partition $\{I\} \cup (G-I)$ forms a semigroup, denoted $G/I$ and called the \emph{Rees quotient semigroup}. All products in $G/I$ are identical to products in $G$ except those involving elements of $I$, in which case all elements of $I$ are considered equivalent. Hence, $G/I$ is a zero-semigroup with zero element $I$. We call $G$ an \emph{ideal extension} of $I$ by $G/I$. 

For any semigroup $G$ and positive integer $n$, we define a semigroup \label{og}$\OG_n = \O_n\cup G$ which extends the multiplication of $\O_n$ and $G$,\footnote{When $G$ itself utilizing a subscript in its notation, we place both subscripts after $G$, where the first describes $\O_n$ and the second describes $G$, for example, \label{oz2}$\OZ_{3,2} = \O_3\cup \Z_2$ is the ideal extension of $\O_3$ by $\z_2^\theta$.} such that for $x\in \O_n$ and $g\in G$ we have $xg=\theta$ (the zero element of $\O_n$) and $gx=x$. It is routine to check for associativity. Note that $\O_n$ is an ideal in $\OG_n$ and $\OG_n/\O_n \cong G^\theta$, and $\OG_1=G^\theta$. It is clear to see that if $\S$ and $\T$ are Schur rings over semigroups $\O_n$ and $G$, respectively, then $\S\vee \T$ is a Schur ring over $\OG_n$. 
When \label{oro}$G=\RO_n$, we denote $\OG_m$ as $\ORO_{m,n}$.

\begin{Thm}\label{thm:oro} Let $\S$ be a partition over $\ORO_{m,n}$ for positive integers $m,n$. Then $\S$ is a Schur ring if and only if $\theta\in \S$. The number of Schur rings over $\ORO_{m,n}$ is the Bell number $\B(m+n-1)$.
\end{Thm}
\begin{proof}
    For any subset $X\subseteq \ORO_{m,n}$, define $X_\O = X\cap \O_n$ and $X_\RO=X\cap \RO_m$. Let $\theta$ be the zero of $\O_n$. Hence, $\Theta\in S$ is the primitive set containing $\theta$. Let $X, Y\in \S$ be primitive sets. Then $XY = (X_\O+X_\RO)Y = X_\O Y + X_\RO Y = |X_\O||Y|\theta + |X_\RO|Y$. Thus, if $\theta\in \S$, then $\QS$ is a Schur ring. 
    
    Conversely, if $\QS$ is a Schur ring, then $\Theta^2 = (\Theta_\O+\Theta_\RO)^2 = |\Theta_\O||\Theta|\theta + |\Theta_\RO\Theta$. As $\theta\in \Theta$, the multiplicity of $\theta$ in $\Theta^2$ is $|\Theta_\O||\Theta|+|\Theta_\RO|$, which says that all elements in $\Theta$ have this same multiplicity in $\Theta^2$. As there are only $|\Theta|^2$ products in $\Theta^2$, this multiplicity would require $(|\Theta_\O||\Theta|+|\Theta_\RO|)|\Theta|$ products. The only solution to the equality $|\Theta|^2=(|\Theta_\O||\Theta|+|\Theta_\RO|)|\Theta|$ is $|\Theta|=1$, that is, $\Theta = \theta$.
\end{proof}
 
Note that $\RO_m$ and $\LO_m$ are equivalent, which implies they have the same number of Schur rings, but $\ORO_{m,n}$ and \label{olo}$\OLO_{m,n}$ generally are not equivalent and will have a very different number of Schur rings. 

\begin{Thm}\label{thm:og} Let $G$ be a finite semigroup such that the regular right action of $G$ onto itself is a permutation action. Let $\S$ be a partition over $\OG_n$ for some positive integer $n$. Then $\S$ is a Schur ring if and only if $\S$ is the common refinement of a Schur ring over $\O_n$ and a Schur ring over $G$, or there exists some $X\in \S$ such that $G\subseteq X$ and $\theta\in \S$. If $\label{omega}\Omega(G)$ is the number of semigroup Schur rings over $G$, then \[\Omega(\OG_n) = \Omega(G)\B(n-1) + \sum_{k=1}^{n-1} \dbinom{n-1}{k}\B(n-k-1).\]
\end{Thm}
Note that the condition here on $G$ includes $\LO_m$ and all finite groups.
\begin{proof}
    Let $\S$ be a Schur ring over $\OG_n$. If $G$ is an $\S$-subsemigroup, then $\S|_{G}$ is a Schur ring over $G$. Since $G$ is an $\S$-subset, then $\OG_n-G = \O_n$ is likewise an $\S$-subset. Therefore, $\S = \S|_{\O_n}\vee \S|_{G}$. 

    Suppose that $G$ is not an $\S$-subsemigroup. Let $X\in \S$ such that $X = X_\O+X_G$ where $X_\O = X\cap \O_n \neq \emptyset$ and $X_G = X\cap G\neq \emptyset$. If $G \not\subseteq X$, then there exists some $y\in G$, $x\in X\cap G$, and $z\in G-X$ such that $yx=z$. Let $Y\in \S$ be the primitive set containing $y$. Let $Y = Y_\O+Y_G$ where $Y_\O = Y\cap \O_n$ and $Y_G = Y\cap G\neq \emptyset$. Then $YX = (Y_\O+Y_G)(X_\O+X_G) = |X||Y_\O|\theta + |Y_G|X_\O + X_GY_G$. As $|Y_G|X_\O\subseteq YX$, it must be that $|Y_G|X_G\subseteq YX$. Considering coefficients, $Y_GX_G=|Y_G|X_G \subseteq G$, but $z\in YX-G$, a contradiction. Therefore, $G\subseteq X$. As $X^2 = (|X_\O|+|G|)|X_\O|\theta + |G|X$ and $|X_\O|\ge 1$, $\theta$ and $X$ have differing multiplicities, which implies that $\theta\in\S$. This proves the first direction.

    Let $\S$ and $\T$ be Schur rings over $\O_n$ and $G$, respectively. Let $X\in \S$ and $Y\in \T$. Then $XY = |X||Y|\theta$ and $YX = |Y|X$. Therefore, $\Q[\S\vee \T]$ is closed under multiplication and is a Schur ring over $\OG_n$. 

    Let $\S$ be a partition over $\OG_n$ such that $\theta \in \S$ and there exists some $X\in \S$ such that $G\subseteq X$. Let $X = X_\O + G$. 
    Then $X^2 = (|X_\O|+|G|)|X_\O|\theta + |G|X\in \QS$. If $Y\subseteq \O_n$, then $XY = (X_\O+G)Y = |X_\O||Y|\theta + |G|Y$ and $YX = Y(X_\O+G) = |Y|(|X_\O|+|G|)\theta$. Therefore, $\S$ is a Schur ring over $\OG_n$. This proves the second direction.

    The number of Schur rings over $\O_n$ is $\B(n-1)$ and the number of Schur rings over $G$ is $\Omega(G)$. Therefore, the number of common refinements is $\B(n-1)\Omega(G)$. Each Schur ring of this type has $G$ as a subsemigroup. If a Schur ring instead does not, then the entirety of $G$ is contained within a single primitive set with at least one nonzero element of $\O_n$. Those nonzero elements of $\O_n$ not adjoined to $G$ can be partitioned freely. For $k$ nonzero elements of $\O_n$ adjoined to $G$, there are $\dbinom{n-1}{k}$ choices for this and each choice has $\B(n-k-1)$ many partitions for the remaining $n-k-1$ nonzero elements. This finishes the proof.
\end{proof}

\begin{Cor}\label{cor:olo} For positive integers $m,n$, 
\[\Omega(\OLO_{m,n}) = \B(m-1)\B(n) + \sum_{k=1}^{m-1} \dbinom{m-1}{k}\B(m-k-1).\]
\end{Cor}

\begin{Exam} By these counts, there are $\B(4) = 15$ Schur rings for $\ORO_{3,2}$ and $\B(2)\B(2) + \dbinom{2}{1}\B(1) + \dbinom{2}{2}\B(0) = 7$ Schur rings for $\OLO_{3,2}$. Similarly, $\ORO_{2,3}$ has 15 Schur rings, but $\OLO_{2,3}$ has 6. \end{Exam}

\begin{Exam}
    Consider the semigroup $\OZ_{3,2}$, which has 
has 7 Schur rings, like $\OLO_{3,2}$, since $\Omega(\z_2)=2$. Conversely, $\OZ_{2,3}$ has 4 Schur rings, since $\Omega(\z_3)=3$, whereas $\OLO_{2,3}$ has 6, since $\Omega(\LO_3)=5$. 
\end{Exam}

\begin{Thm}\label{thm:och} Let $G$ be a finite zero-semigroup and let $n$ be a positive integer. Then every Schur ring over $\OG_n$ is a common refinement of Schur rings over $\O_n$ and $G$. If $\Omega(G)$ is the number of Schur rings over $G$, then the number of Schur rings over $\OG_n$ is $\Omega(G)\B(n-1)$.
\end{Thm}
\begin{proof}
    Let $\theta\in \O_n$ and $\phi\in G$ be the zero elements of these semigroups. Let $\S$ be a Schur ring over $\OG_n$. Suppose to the contrary that $G$ is not an $\S$-subset. Then, there exists a primitive set $X\in \S$ and $X=X_\O+X_G$ where $X_\O=X\cap \O_n\neq \emptyset$ and $X_G=X\cap G\neq \emptyset$. Let $\Phi\in \S$ be the primitive set containing $\phi$. Let $\Phi = \Phi_\O + \phi + \Phi_G$, where $\Phi_\O=\Phi\cap \O_n$ and $\Phi_G = \Phi\cap G-\theta$. Then \[\Phi X = (\Phi_\O + \phi + \Phi_G)(X_\O+X_G) = |X||\Phi_\O|\theta + (1+|\Phi_G|)X_\O + |X_G|\phi + \Phi_GX_G.\] Considering coefficients, since $(1+|\Phi_G|)X_\O\subseteq \Phi X$, it must be that $\Phi_GX_G = (1+|\Phi_G|)X_G$, but $|\Phi_GX_G| = |\Phi_G||X_G| < (1+|\Phi_G|)|X_G|$ when $|X_\O|,|X_G|>0$, a contradiction.
\end{proof}

\begin{Cor}\label{oo} For any cardinals $n,m$, the number of Schur rings over $\OO_{n,m}$ is $\B(n-1)\B(m-1)$.
\end{Cor}

\section{Traditional Schur Rings}\label{sec:trad}
As mentioned above, Schur rings were first used by Schur and Wielandt to study permutation actions of groups. A so-called \emph{transitivity module} (see \cite{Wielandt64}) was a partition of a group formed by an ambient permutation group acting on it. One calls a Schur ring over group $H$ \emph{Schurian} if it can be realized as a transitivity module for some permutation group $G$ acting on $H$. The group $H$ is likewise called \emph{Schurian} if all of its Schur rings are Schurian. Some \cite{Ponomarenko2022, Muzychuk2016, Muzychuk09, Ryabov2015, Ryabov2022,  Ryabov2023} have recently made efforts to classify Schurian groups. Regrettably, not all semigroups can be embedded into a permutation group which would act regularly upon it, as is necessary in the case of a transitivity module. Therefore, outside the category of groups, the discovery of Schurian partitions seems problematic.

In \cite{InfiniteI}, the notion of \emph{traditional Schur rings}\footnote{Traditional Schur rings are exactly those Schur rings used in the classification of Schur rings over finite cyclic groups \cite{LeungII, LeungI}.} is introduced. A traditional Schur ring is a partition belonging to one of four families: automorphic, direct product, trivial, and wedge product. A group $G$ is called \emph{traditional} if all Schur rings are traditional. Unlike Schurian Schur rings coming from external permutations groups, traditional Schur rings involve concepts internal to the group itself, which offer more natural generalizations to semigroups. In this section, we attempt to generalize these four families of traditional Schur rings to the category of semigroups.  

\subsection{Automorphic Schur Rings}\label{sec:auto}

Let $G$ be a semigroup, and let $\Aut(G)$ denote the automorphism group of $G$. Let $\H\le \Aut(G)$. Then $\H$ acts on $G$ via this automorphism action. Then the partition of $G$ associated to the $\H$-action is a Schur ring, called an \emph{automorphic Schur ring}. The proof follows the same argument used for automorphic Schur rings over groups (\cite[see Example 2.20]{MePhD}). Of note is the special case when $\H=1$ is trivial. The automorphic Schur ring afforded by this action on $G$ is the partition of singletons, $\S = \{x\mid x\in G\} = G$, which we call the \emph{discrete Schur ring}. Thus, each semigroup $G$ is a Schur ring over itself. Note that $\Aut(\O_n) = S_{n-1}$, $\Aut(\LO_n)=S_n$, and $\Aut(\K_{1,n})=S_n$. Therefore, all Schur rings over $\O_n$, $\LO_n$, and $\K_{1,n}$ are necessarily automorphic.

\subsection{Direct Product Schur Rings}
Let $G$ and $H$ be two semigroups, and let $\S$ and $\T$ be Schur rings over $G$ and $H$, respectively. Let \label{direct}$\S\times \T$, called the \emph{direct product} of $\S$ and $\T$, be the partition of $G\times H$ whose primitive sets are each of the form $X \times Y = \{(x,y)\mid x\in X, y\in Y\}$ for some $X\in \S$ and $Y\in \T$. The proof follows the same argument used for direct product Schur rings over groups (\cite[see Example 2.26]{MePhD}). 

Every rectangular band is equivalent to $\LO_n\times \RO_m$ for some positive integers $n, m$. While not every Schur ring over a rectangular band is necessarily a direct product Schur ring, we can easily see that $\B(n)\B(m)$ Schur rings will be. 

\subsection{Trivial Schur Rings}
The \emph{trivial Schur ring} of a group $G$ is the partition $\S = \{e, G-e\}$. For a general finite semigroup $G$, we say a Schur ring $\S$ is \emph{trivial} if $\S = \{\varepsilon, G-\varepsilon\}$ for some $\varepsilon\in G$. Note that $\QS = \Span\{\varepsilon,G\}$. For $|G|>2$, then $\varepsilon^2 = \varepsilon$, that is, $\varepsilon$ is an idempotent. Hence, for groups, this extended notion of a trivial Schur ring gives rise only to the original one. Considering next $\varepsilon G$ in $\QS$, we either have $\varepsilon G = |G|\varepsilon$, which implies that $\varepsilon$ is a left-zero, or $\varepsilon G = G$, which implies that $\varepsilon$ is a left-identity. Considering also $G\varepsilon$, we see that $\varepsilon$ is a zero, an identity, both a left-zero and a right-identity, or both a right-zero and a left-identity. The first case obviously implies that $G$ is a zero-semigroup, the second implies that $G$ is a monoid, and the last two cases implies that $G$ is isomorphic to $\LO_n$ or $\RO_n$, respectively. Considering \thmref{thm:zeroschur}, we see that the only semigroup, up to equivalence, which can have multiple trivial Schur rings is $\LO_n$. Hence, the trivial Schur ring, if present, will be unique, except for $\LO_n$.

We have now seen a collection of semigroups which have a trivial Schur ring, including $\K_{1,n}$, $\LO_n^e$, $\O_n$, $\OLO_{m,n}$, and $\ORO_{m,n}$. In each of these cases, due to Theorems \ref{thm:zeroschur} and \ref{thm:monoidSchur}, each Schur ring is a refinement of the trivial Schur ring. Hence, the number of Schur rings $\Omega(G)$ of any of these semigroups is at most $\B(|G|-1)$. The following theorem classifies all semigroups with a trivial Schur ring that obtain this bound.

\begin{Thm}\label{bnminus} Let $G$ be a semigroup of order $n\ge 6$. Suppose that $G$ has the trivial Schur ring and all other Schur rings are a refinement of the trivial Schur ring; hence, $\Omega(G) = \B(n-1)$. Then $G$ is equivalent to one\footnote{Note that $\RO_{n-1}^\theta \cong \ORO_{1,n-1}$.} of $\K_{1,n-1}$, $\LO_{n-1}^e$, $\O_n$, or $\ORO_{k,n-k}$ for $k\in\{1, \ldots, n-1\}$.
\end{Thm}
\begin{proof}
    Let $\S=\{\varepsilon, G-\varepsilon\}$ be the trivial Schur ring. Since $\Omega(\LO_n)=\B(n)$, we may assume that $\varepsilon$ is either a zero or identity. By assumption, a partition of $G$ is a Schur ring if and only if it is a refinement of $\S$. 
    Let $y\in G-\varepsilon$. Then $\{\varepsilon, y, G-\varepsilon-y\}$ is a Schur ring over $G$. Thus, $\{\varepsilon,y\}\le G$ for all $y$, by \propref{prop:singleton}. This implies that $y^2=\varepsilon$ or $y^2=y$. Since $y$ is arbitrary, this lets us split $G-\varepsilon$ into two sets. Let $A = \{y\in G-\varepsilon\mid y^2=\varepsilon\}$ and $B = \{y\in G-\varepsilon\mid y^2=y\}$. Of course, $G=\varepsilon\cup A\cup B$. 
    
    Let $x\in G-\varepsilon$ be a fixed point throughout, and let $y,z\in G-\varepsilon$ be two generic elements. Thus, $\{\varepsilon,y,z,G-\varepsilon-y-z\}$ is a Schur ring, and $\{\varepsilon, y,z\} \le G$ for all $y, z$.  Likewise, $\{\varepsilon, y+z, G-\varepsilon-y-z\}$ is a Schur ring over $G$. In particular, the semigroup $\{\varepsilon,y,z\}$ also has the trivial Schur ring. Applying the enumeration of Schur rings of order 3 conducted in Section \ref{gap} (see Table \ref{fig:sgsr3}), these conditions 
     show that $\{\varepsilon, y,z\}$ must be equivalent to one of \begin{equation}\label{eq:exy}\O_3, \ORO_{2,1}, \K_{1,2}, \LO_2^\theta, \LO_2^e.\end{equation} Similarly, $\{\varepsilon, x,y,z,G-\varepsilon-x-y-z\}$ is a Schur ring, which implies that $\{\varepsilon,x,y,z\}\le G$ and $\{\varepsilon, y,z\}\le\{\varepsilon,x,y,z\}$. Again, $\{\varepsilon,G-\varepsilon-x-y-z\}\cup \T$, where $\T$ is any partition on $\{x,y,z\}$, is a Schur ring of $G$, meaning that every refinement of the trivial Schur ring over $\{\varepsilon, x,y,z\}$ is likewise a Schur ring over $\{\varepsilon, x,y,z\}$. Applying the enumeration of Schur rings of order 4 conducted in Section \ref{gap} (see Table \ref{fig:sgsr4}), considering these conditions
    , it must be that $\{\varepsilon, x,y,z\}$ is equivalent to one of \begin{equation}\label{eq:exyz} \O_4, \ORO_{3,1}, \ORO_{2,2}, \K_{1,3}, \LO_3^\theta, \LO_3^e.\end{equation} 
    
    Suppose $\varepsilon$ is an identity. Then each of $\{\varepsilon,x,y\}$, $\{\varepsilon,x,z\}$, $\{\varepsilon,y,z\}$ is one of $\LO_2^e$, $\RO_2^e$ by \eqref{eq:exy}, but $\{\varepsilon, x,y,z\}\cong \LO_3^e$,$\RO_3^e$  by \eqref{eq:exyz}. Hence, $x^2=x$, $y^2=y$, and $z^2=z$. Without the loss of generality, suppose $xy=x$. Then $\{\varepsilon,x,y,z\}\cong \LO_3^e$. Hence $xz=x$, $zx=z$, $yz=y$, and $zy=z$. Relative to the fixed element $x$, each product in $G$ is determined as $y^2=y$ and $yz=y$. Therefore, $G\cong \LO_{|B|}^e$.

Suppose $\varepsilon$ is a zero. Then each of $\{\varepsilon,x,y\}$, $\{\varepsilon,x,z\}$, $\{\varepsilon,y,z\}$ must be equivalent to one of\begin{equation}\label{eq:exyzero} \O_3, \ORO_{2,1},\K_{1,2}, \LO_2^\theta\end{equation} by \eqref{eq:exy}, but $\{\varepsilon, x,y,z\}$ must be equivalent to one of\begin{equation}\label{eq:exyzzero}\O_4, \ORO_{3,1}, \ORO_{2,2}, \K_{1,3}, \LO_3^\theta\end{equation} by \eqref{eq:exyz}. 

 If $B=\emptyset$ then $y$, $z\in A$. Then  $\{\varepsilon, y,z\}\cong \O_3$ by \eqref{eq:exyzero}. Therefore, $y^2=z^2=yz=zy=\varepsilon$, and each product in $G$ is determined. Therefore, $G\cong \O_{|A|+1}$. For the rest of the proof we may then assume that $x\in B$, that is, $x^2=x$.

If $A=\emptyset$ then $x$, $y$, $z\in B$. Then each of $\{\varepsilon, x,y\}$, $\{\varepsilon, x,z\}$, $\{\varepsilon, y,z\}$ must be isomorphic to one of $\K_{1,2}$, $\LO_2^\theta$, $\RO_2^\theta$. Up to equivalence, there are five cases to consider:
    \begin{enumerate}[i.]
        \item Suppose $\{\varepsilon, x,y\}\cong \LO_2^\theta$, $\{\varepsilon,x,z\} \cong \LO_2^\theta$, and $\{\varepsilon,y,z\}\cong \LO_2^\theta$. Then $\{\varepsilon,x,y,z\}\cong \LO_3^\theta$.
        \item Suppose $\{\varepsilon, x,y\}\cong \LO_2^\theta$, $\{\varepsilon,x,z\} \cong \LO_2^\theta$, and $\{\varepsilon,y,z\}\cong \RO_2^\theta$. Then multiplication is non-associative, since $y(xz) = yx = y\neq z = yz = (yx)z.$ 
        \item Suppose $\{\varepsilon, x,y\}\cong \LO_2^\theta$, $\{\varepsilon,x,z\} \cong \LO_2^\theta$ or $\RO_2^\theta$, and $\{\varepsilon,y,z\}\cong \K_{1,2}$. Then multiplication is also non-associative, since $x(yz) = x\varepsilon  = \varepsilon \neq xz = (xy)z.$ 
        \item Suppose $\{\varepsilon, x,y\}\cong \LO_2^\theta$, $\{\varepsilon,x,z\} \cong \K_{1,2}$, and $\{\varepsilon,y,z\}\cong \K_{1,2}$. Then $\{\varepsilon,x,y,z\}\cong \LO_2^\theta\u \CH_2$, which contradicts \eqref{eq:exyzzero}.
        \item Suppose $\{\varepsilon, x,y\}\cong \K_{1,2}$, $\{\varepsilon,x,z\} \cong \K_{1,2}$, and $\{\varepsilon,y,z\}\cong \K_{1,2}$. Then $\{\varepsilon,x,y,z\}\cong \K_{1,3}$.
    \end{enumerate}
 This shows that if $xy=x$, then $\{\varepsilon,x,y,z\}\cong \LO_3^\theta$. Thus $xz=x$, $zx=z$, $yz=y$, and $zy=z$. Relative to the fixed element $x$, each product in $G$ is determined as $y^2=y$ and $yz=y$. Therefore, $G\cong \LO_{|B|}^e$.  Otherwise, if $xy=\varepsilon$, then $\{\varepsilon,x,y,z\}\cong \K_{1,3}$. Thus $xz=zx=yz=zy=\varepsilon$. Relative to this fixed point $x$, each product in $G$ is determined as $y^2=y$ and $yz=\varepsilon$. Therefore, $G\cong \K_{1,|B|}$.

Next, assume both $A$, $B\neq\emptyset$. Then $A+\varepsilon\cong \O_{|A|+1}$ and, without the loss of generality, $B+\varepsilon$ is isomorphic to one of $\K_{1,|B|}$, $\RO_{|B|}^\theta$. If $B=\{x\}$, then $y$, $z\in A$, $\{\varepsilon, y,z\}\cong \O_3$, and $\{\varepsilon, x,y\}, \{\varepsilon,x,z\}\cong \ORO_{2,1}$ by \eqref{eq:exyzero}. Thus $x^2=x$, $xy=y$, $xz=z$, $yx=y^2=yz=zx=zy=z^2=\varepsilon$. As this determines all products in $G$, we see that $G\cong \ORO_{|A|+1,1}$.

Finally, suppose then $|B|>1$, $y\in A$, and $x$, $z\in B$. 
$\{\varepsilon, x,y\}, \{\varepsilon,y,z\}\cong \ORO_{2,1}$. Then both $\{\varepsilon, x,y\}$, $\{\varepsilon,y,z\}\cong \ORO_{2,1}$. If $B+\varepsilon\cong \K_{1,|B|}$, then $\{\varepsilon, x,z\}\cong \K_{1,2}$, but by \eqref{eq:exyzzero} there is no semigroup $\{\varepsilon,x,y,z\}$ that satisfies the requirements on the subsemigroup structure. Therefore, $B+\varepsilon\cong \RO_m^\theta$. Then $\{\varepsilon,x,y,z\}\cong \ORO_{2,2}$ by \eqref{eq:exyzzero}. Hence, $y^2=\varepsilon$, $yz=\varepsilon$, $zy=y$, and $z^2=z$, which determines all products in $G$. Therefore, $G\cong \ORO_{|A|+1,|B|}$. As these cases exhaust all possibilities on the products $\varepsilon^2$, $\varepsilon y$, $y\varepsilon$, $y^2$, $yz$, $zy$, and $z^2$, utilizing a fixed $x$ we have been able to determine uniformity between choices such as $\LO_k$ versus $\RO_k$, the result is now proven.
\end{proof}

There are a few small counterexamples to the above theorem, thus requiring $n\ge 5$, namely, for $\B(1) \colon \CH_2$ (see Theorem \ref{thm:chainschur} for chain semilattices); for $\B(2) \colon \z_2^\theta$, $\z_2^e$; and for $\B(3) \colon \LO_2\times \RO_2$, $\z_4$. Note that $\B(1)=\B(0)$, which explains $\CH_2$. Also note that $\z_2$ is the only semigroup other than $\LO_n$ that obtains the full Bell number $\B(n)$, which explains $\z_2^\theta$ and $\z_2^e$ obtaining $\B(3-1)$. Related to this, we note that it is difficult for the number of Schur rings to fall between $\B(n-1)$ and $\B(n)$. The only known counterexamples are $\z_3$ which has 3 Schur rings, $\LO_2\times \z_2$ which has 7 Schur rings, and the Klein 4-group $\text{V}_4 = \z_2\times \z_2$ which has 9 Schur rings. Note that all these counterexamples are closely related to $\z_2$ being a counterexample to Theorem \ref{leftnullconverse}. Mimicking the proof of Theorem \ref{leftnullconverse}, if there was any other counterexample to a semigroup having a number of Schur rings strictly between $\B(n-1)$ and $\B(n)$, it would have to be a very large monoid and seems unlikely to exist.

A coarser partition than the trivial partition which we consider is the indiscrete partition $\{G\}$. The indiscrete partition is a Schur ring over $G$ if and only if $G^2 = |G|G$ if and only if all elements appear in the Cayley table of $G$ with the same multiplicity. This implies that $G$ must be finite. If $G$ is a finite rectangular group, that is, a direct product between a finite rectangular band and a finite group, this property holds. When the indiscrete partition is a semigroup Schur ring over $G$, then we call this the \emph{indiscrete Schur ring}. 

By Proposition \ref{prop:monoidschur}, the trivial partition $\{e, G-e\}$ over a monoid $G$ is a Schur ring if the indiscrete partition $\{G\}$ is likewise a Schur ring over $G$, but there do exist monoids which have the trivial Schur ring but not the indiscrete one, as demonstrated in the next proposition which follows immediately from Theorem \ref{thm:monoidSchur}.

\begin{Prop} Let $G$ be a semigroup for which the indiscrete partition $\{G\}$ is a Schur ring over $G$. Then $G^e$ has the trivial Schur ring $\{e,G^e-e\}$ but does not have the indiscrete Schur ring $\{G^e\}$.
\end{Prop}

\begin{Prop} If $G$ is a finite semigroup and $\S$ and $\T$ are two of its Schur rings, then $\Q[\S\wedge \T] = \Q[\S]\cap \Q[\T]$ is also a Schur ring over $G$.
\end{Prop}
\begin{proof} See Proposition 2.25 from \cite{MePhD}. As the proof there relies solely on the representation theory of the Hadamard product, and not on the group multiplication, the proof carries over immediately to the semigroup ring $\Q[G]$.
\end{proof}

\begin{Cor}\label{coarsest}
Let $G$ be a finite semigroup. Then there is a coarsest Schur ring of $G$, which is unique.
\end{Cor}

\begin{Cor}\label{cor:omegabound}
Let $G$ be a finite semigroup, with $\S$ its coarsest Schur ring. If $\lambda_1,\lambda_2, \ldots, \lambda_n$ are the sizes of the primitive sets of $\S$, then
$$
\Omega(G) \leq \B(\lambda_1)\B(\lambda_2)\ldots\B(\lambda_n).
$$
\end{Cor}

When the indiscrete partition is a Schur ring, such as it is for finite rectangular groups, it is clearly the coarsest semigroup Schur ring. For $\O_n$, the coarsest Schur ring is the trivial Schur ring. For some semigroups, the coarsest Schur ring might even be more refined than that. For some semigroups, the coarsest Schur ring could be the discrete Schur ring. To demonstrate this claim, let \label{ch} $\CH_n$, called the \emph{chain semigroup}, be the set $\{1, 2, \ldots, n\}$ with the operation $\max(\cdot,\cdot)$. 

\begin{Thm}\label{thm:chainschur}
The only Schur ring over $\CH_n$ is the discrete Schur ring.
\end{Thm}
\begin{proof}
    The proof follows by induction on $n$. The base case is clear. Note that $\CH_n = \CH_{n-1}^\theta$. Thus, the Schur rings over $\CH_n$ correspond with the Schur rings over $\CH_{n-1}$, by \corref{cor:thetaschur}, but $\CH_{n-1}$ only has the discrete Schur ring by the inductive hypothesis.
\end{proof}

The following corollary follows immediately from Theorems \ref{thm:och} and \ref{thm:chainschur}.

\begin{Cor}\label{cor:och} For any positive integers $m,n$, the number of Schur rings over $\OCH_{m,n}$ is $\B(m-1)$.
\end{Cor}\label{och}

We will see in Section \ref{gap} that it is quite common for semigroups to have a single Schur ring, necessarily the discrete Schur ring. Regardless of how refined this coarsest Schur ring is, it serves the same purpose as the trivial Schur ring does over groups. Therefore, any consideration of Schur rings over semigroups should include this coarsest Schur ring, as noted by \corref{cor:omegabound}. 

\subsection{Wedge Products}
The notion of the wedge product was introduced in \cite{LeungI} by Leung and Man.\footnote{Wedge products were also independently introduced by Evdokimov and Ponomarenko \cite{Ponomarenko01, Ponomarenko02}, which they called the \emph{generalized wreath product}. The wedge product was later generalized to association schemes by Muzychuk \cite{MuzychukWedge} and to supercharacter theories by Hendrickson \cite{Hendrickson}.} A \emph{wedge product} of Schur rings is built by combining a Schur ring over a subgroup and a Schur ring over a quotient group, so long as certain compatibility conditions are satisfied.   The wedge product construction depends heavily on the theory of cosets, which is vastly different in the category of semigroups. As such, we do not present here a complete generalization of wedge products for semigroups. Instead, we present some examples of Schur rings generated by Rees quotients \cite{Rees}, which are examples that ought to be considered a wedge product over semigroups.

\begin{Thm}\label{thm:Reesextension} Let $G$ be a semigroup with ideal $I$. Let $G/I$ be the Rees quotient. Let $\S$ be a Schur ring over $G/I$. Then $I\vee \S$ 
 is a Schur ring over $G$, called the ideal extension of $I$ by $\S$.
\end{Thm}
Note that $I\vee \S$ is the Schur ring induced by lifting $\S$ from $G/I$ into $G$ and inserting a discrete partition over $I$. As $\S$ is a partition of the set $\{I\}\cup G$, we see $I\vee \S$ is the refinement of $\S$ and the discrete partition over $I$.
\begin{proof}
    As $I$ is a discrete $(I\vee \S)$-subset, all subsets of $I$ are likewise $(I\vee \S)$-subsets. Hence, $\Q[I]\le \Q[I\vee \S]$. If $X,Y\in \S$, then $XY = \alpha + \sum_{C\in \S} \lambda_CC \in \Q[I\vee \S]$, for some $\alpha\in \Q[I]$. If $a, b\in I$, then each of $aX$, $Xa$, $ab \in \Q[I] \le \Q[I\vee \S]$. Therefore, $I\vee \S$ is a Schur ring over $G$.
\end{proof}

In Example \ref{exam:numerical}, we examine a numerical semigroup with two Schur rings, the discrete Schur ring and an ideal extension using the semigroup's unique maximal ideal and the only non-discrete Schur ring over this Rees quotient.

\begin{Thm}\label{thm:reesquotient} Let $G$ be a semigroup with ideal $I$. Let $\S$ be a Schur ring over $G$ such that $I$ is an $\S$-ideal. Then $\S/I = \S|_{G-I}\vee \{I\}$ is a Schur ring over the Rees quotient $G/I$.
\end{Thm}
\begin{proof}
 Viewing $I$ as the zero element of $G/I$, we have $I^2 = xI= Ix = I$ for all $x\in G/I$. If $X\in \S/I$ is nonzero, then $XI = IX = |X|I$. If $Y\in \S/I$ is likewise nonzero, then, in $\Q[G]$, $XY = \sum_{a\in I} \mu_aa + \sum_{C\in \S-\S|_I} \lambda_CC$. Then, in $\Q[G/I]$, $XY = \sum_{a\in I}\mu_aI + \sum_{C\in \S-\S|_I} \lambda_CC\in \Q[\S/I]$. 
\end{proof}

Theorems \ref{thm:Reesextension} and \ref{thm:reesquotient} establish a Correspondence Theorem with respect to Schur rings and Rees quotients. In particular, there is a one-to-one correspondence between Schur rings over $G/I$ and the Schur rings over $G$ which contain $\Q[I]$ as a discrete Schur subring. Thus, we see that Schur rings $\S$ can be built from the Rees quotients with respect to their $\S$-ideals. We note that while a Schur ring $\S/I$ over $G/I$ is guaranteed to lift to a Schur ring $I\vee (\S/I)$ over $G$ where the partition over $I$ is discrete, it is possible for $G$ to have a Schur ring $\S$ with $\S$-ideal $I$ and the partition over $I$ is non-discrete (see Example \ref{exam:nilpotent}). 

In the next section, we will explore more the ramifications of \thmref{thm:reesquotient}, but before doing it we want to present some examples of semigroup ``wedge'' products. Given two semigroups $G$, $H$, we define a multiplication on the disjoint union $G\cup H$ by extending the respective multiplications in $G$ and $H$ by the rule $gh = hg = g$ whenever $g\in G$ and $h\in H$. This multiplication is formed by stacking $G$ onto $H$, where $G$ dominates over every element of $H$ similar to higher elements in a semilattice. As we will discuss some other ways to define a multiplication on $G\cup H$,  we call this semigroup the \emph{stack} of $G$ on $H$, denoted \label{stack}$G\s H$. Note that $G$ is an ideal in $G\s H$ and $(G\s H)/G \cong H^\theta$. This construction has been used before (see \cite{Distler}) but no well-accepted name or notation is known to the authors.

\begin{Thm}\label{thm:stackschur} If $\S$ and $\T$ are Schur rings over semigroups $G$ and $H$, respectively, then 
$\S\vee \T$ is a Schur ring over the stack $G \s H$. Furthermore, every Schur ring over $G \s H$ has this form. 
\end{Thm}
\begin{proof}
As $\S$ and $\T$ are Schur rings, we need only check the products $XY$ and $YX$ where $X\in \S$ and $Y\in \T$. In this case, we have $XY = YX = |Y|X\in \QS \le \Q[\S\vee \T]$.

Let $\S$ be a Schur ring over $G \s H$. Let $X\in \S$ which contains an element from $G$. Let $X_G=X\cap G$ and $X_H=X\cap H$; hence, $X = X_G+X_H$. Then $X^2 = (X_G+X_H)^2 = X_G^2 + 2|X_H|X_G + X_H^2$. Since $X_G\neq \emptyset$, the elements of $X_G$ appear in $X^2$ with multiplicity at least $2|X_H|$. Since $X$ is a primitive set, the elements of $X_H$ must also appear with multiplicity at least $2|X_H|$. But the only elements of $H$ contained within $X^2$ must be within $X_H^2$. Counting multiplicities accounts for at most $|X_H|^2$ elements from $H$, and there must be at least $2|X_H|^2$ such elements. Therefore, $|X_H|=0$, that is, $X=X_G\subseteq G$. In particular, $G$, and consequentially $H$, must be an $\S$-subsemigroup, which completes the proof.
\end{proof}

Similar to how a  semidirect product is defined using automorphisms to twist a direct product, we can construct a \emph{semi-stack}\footnote{We will consider another version of a semi-stack in Section \ref{sec:idempotent}, denoted $\LORO(\ell,m,n,k)$, which is delayed until we consider indecomposable imposition.} of semigroups when the first semigroup is equivalent to $\LO_n$. As mentioned in Section \ref{sec:auto}, $\Aut(\LO_n) = S_n$. Let $\varphi: G \to S_n$ be a group homomorphism. Then define a multiplication on \label{semistack}$\LO_n \shat G = \LO_n\cup G$ which extends the multiplications on $\LO_n$ and $G$ by the rule: for each $x\in \LO_n$ and $g\in G$, define $xg = x$ and $gx = \varphi_g(x)$, where $\varphi_g :\LO_n\to \LO_n$ is the automorphism associated to $g$. Then it is easy to check that $\LO_n\shat G$ equipped with this multiplication is a semigroup and $\LO_n$ is an ideal. Note that the forgetfulness of multiplication in $\LO_n$ assures that this multiplication will be associative. For example, $x(gy)= x\varphi_g(y) = x = xy = (xg)y$, but if $x,y$ were drafted from an arbitrary semigroup then $xy = x\varphi_g(y)$ for all $g\in G$, which in most settings would be very restrictive. 

By similarly reasoning to the proof of \thmref{thm:stackschur}, all Schur rings over $\LO_n\shat G$ have the form $\S\vee\T$ for Schur rings over $\S$ and $\T$ over $\LO_n$ and $G$, respectively. Unlike the case of stacks though, not all Schur rings over $\LO_n$ and $G$ are compatiable.

\begin{Exam}
    In the special case where $G=\z_n$ and $\varphi : \z_n \to \Aut(\LO_n)$ is the cyclic permutation action, we let \label{lozn}$\LOZ_n = \LO_n\shat \z_n$. By \thmref{thm:leftnullschur}, all partitions over $\LO_n$ are Schur rings, giving $\B(n)$ such Schur rings. Counting Schur rings over cyclic groups is much more subtle. See \cite{CountingII} for a survey of enumerating group Schur rings over finite cyclic groups. Consider the case $\LOZ_3 = \LO_3\cup \z_3 = \{0,1,2\}\cup \{e,z,z^2\}$. Then $\z_3$ has only three semigroup Schur rings, the discrete, trivial, and indiscrete Schur rings: \[\{ e,z,z^2 \},\quad \{ e, z+z^2 \},\quad\text{and}\quad \{ e+z+z^2\}.\] Therefore, the number of Schur rings over $\LO_3\s\z_3$ will be $5(3)=15$, but $\LOZ_3$ only has 12 Schur rings. In particular, the partition
\[\S = \{ 0, 1+2, e,z,z^2 \}\] is not a Schur ring over $\LOZ_3$ since $z0 = 1 \notin \QS$. The other two missing partitions are constructed similarly.
\end{Exam}

Suppose that $G$ and $H$ are both zero-semigroups, with zeros $\theta_G$ and $\theta_H$, respectively. Extending the multiplication in $G$ and $H$ to $G\cup H$, for $g\in G$ and $h\in H$, we define $gh = \theta_G$ and $hg=\theta_H$. This defines an associative multiplication on the union of $G$ and $H$ with a \emph{twist}, which we denote \label{twist}$G\t H$. Unlike $G\s H$, neither $G$ nor $H$ are ideals in $G\t H$, although they are both right-ideals.  

\begin{Thm}\label{thm:veebarschur} If $\S$ and $\T$ are Schur rings over zero-semigroups $G$ and $H$, respectively, then $\S\vee \T$ is a Schur ring over $G \t H$. 
\end{Thm}
\begin{proof}
By \corref{cor:thetaschur}, $\theta_G$ and $\theta_H$ are elements of $\S$ and $\T$, respectively. Let $X\in \S$ and $Y\in \T$. Then $XY = |X||Y|\theta_G$, and $YX = |X||Y|\theta_H\in \Q[\S\vee \T]$, which proves the result.
\end{proof}

\begin{Exam}\label{exam:veebar}
\thmref{thm:veebarschur} does not necessarily give every Schur ring of $G\t H$. For example, let $G=\z_2^\theta=\{\theta_G, e_G, g\}$ and $H=\z_2^\theta=\{\theta_H, e_H, h\}$, but notice that $G\t H$ has the Schur ring $\{\theta_G, \theta_H, e_G+ e_H, g+ h\}$ which is not given by the theorem.
\end{Exam}

The following theorem has a similar proof to Theorem \ref{thm:zeroschur}.

\begin{Thm} Let $G$ be a nontrivial zero-semigroup and let $\z_1 = \{\varepsilon\}$ be the trivial group. Let $\S$ be a Schur ring over $G\t \z_1$. Then $\varepsilon\in \S$. It follows that the Schur rings over $G\t \z_1$ are in one-to-one correspondence with the Schur rings over $G$. 
\end{Thm}

Similar to the semi-stack $\LOZ_n$, we can describe a similar \emph{semi-twist} again using $\LO_n$. More specifically, we can construct a twist utilizing left-zeros, which all elements of $\LO_n$ are. Let \label{semitwist}$\LO_n \t G = \LO_n\cup G$, for some zero-semigroup $G$, be a semigroup which extends the multiplications of $\LO_n$ and $G$ such that for all $g\in G$ and $x\in \LO_n$ we have $gx=\theta$ and $xg=x$. It is straightforward to show this multiplication is associative.\footnote{Note that the twist operation was previously defined for zero-semigroups, which $\LO_n$ is not when $n\ge 2$. But as each element of $\LO_n$ is a left-zero, extending twists to include $\LO_n$ extends the intent of twists, namely when products involve two elements from different semigroups, the left element acts as a zero to the right. In the case that $n=1$, $\LO_1 \cong \z_1$ and $\LO_1 \t G\cong \z_1\t G$. Hence, the two interpretations agree in the single case where the symbol $\mathfrak{t}$ is overloaded. For this reason, it was not necessary to introduce a new symbol, such as $\hat{\mathfrak{t}}$, unlike the case of stacks and semi-stacks.} Note that $\LO_n$ and $G$ are right-ideals and $\LO_n\cup \theta \cong \LO_{n+1}$ is a 2-sided ideal of $\LO_n\t G$. Hence, $\LO_n \t G$ is an ideal extension of $\LO_{n+1}$ by $G$. Note that $\LO_n \t \z_1\cong \LO_{n+1}$.

\begin{Thm} Let $G \not\cong\z_1$ be a zero-semigroup with zero $\theta$, and let $n\in \N$. Then every Schur ring over $\LO_n\t G$ has the form $\S\vee\T$ where $\S$ and $\T$ are Schur rings over $\LO_n$ and $G$, respectively. If $\Omega(G)$ is the number of Schur rings over $G$, then the number of Schur rings over $\LO_n\t G$ is $\B(n)\Omega(G)$.
\end{Thm}
\begin{proof}
    Let $\S$ be a Schur ring over $\LO_n\t G$ and let $X$ be a primitive set, where $X= X_n+X_G$ for $X_n\subseteq \LO_n$ and $X_G\subseteq G$. Note $X^2 = |X|X_n+ |X_n||X_G|\theta + X_G^2$. Since $X_n\subseteq X^2$, it must be that $X\subseteq X^2$. As $X_G\subseteq |X_n||X_G|\theta+X_G^2$, it must be that $|X_n||X_G|\theta + X_G^2 = |X|X_G$ if $\theta\in X_G$ or $X_G^2=|X|X_G$ if not. 
    
    Consider first if $\theta\notin X_G$. Then $X_G^2=|X|X_G = (|X_n|+|X_G|)X_G$, but, counting multiplicities, $X_G^2$ has $|X_G|^2$ products and $|X|X_G$ has $(|X_n|+|X_G|)|X_G|$ products. As $|X_G|^2 \le (|X_n|+|X_G|)|X_G|$, equality is only obtained at $|X_n|=0$ or $|X_G|=0$. 
    
    Consider next if $\theta \in X_G$. Then let $X_G=\theta+X_G'$ and then \[X_G^2 = (|X_n|(|X_G'|+1) +2|X_G'|+1)\theta + X_G'^2 = |X| = |X_n| + |X_G'|+1,\] which implies that $|X_G'|(|X_n|+1)=0$, that is, $|X_G'|=0$. Thus, $X=X_n+\theta$. Since $|G|>1$, $G-\theta$ is an $\S$-subset. Thus, there exists a primitive set $Y\subseteq G-\theta$, but $YX = |X||Y|\theta$, contradicting the primitivity of $X$. Therefore, $\LO_n$ and $G$ are $\S$-subsemigroups. The result then follows.
\end{proof}

Note that $\{\theta_G,\theta_H\}$ is an ideal of $G\t H$. Thus, we can remove the twist\footnote{Similar to how we extended twists to include $\LO_n$, we likewise can extend unites to include $\LO_n$. Hence, \label{semiunite}$\LO_n\u G$ is the quotient semigroup formed by identifying $0\in \LO_n$ with $\theta\in G$, but as this results in $\LO_{n+1}\u G\cong \LO_n\t G$ we gain no new semigroups or Schur rings with this generalization.} of $G\t H$ by \emph{uniting} the zero elements together. Let \label{unite}$G\u H = (G\t H)/\{\theta_G,\theta_H\}$. Then $G\u H$ contains both $G$ and $H$ as subsemigroups, but neither is necessarily an ideal. Regardless, we can build Schur rings over $G\u H$ similarly to \thmref{thm:stackschur}, which is an immediate consequence of \thmref{thm:reesquotient} and \thmref{thm:veebarschur}.

\begin{Cor}\label{thm:thetastackrings}
 If $\S$ and $\T$ are Schur rings over zero-semigroups $G$ and $H$, respectively, then $\S\vee \T$ is a Schur ring over $G \u H$.
\end{Cor}

\begin{Exam}
Similar to \examref{exam:veebar}, we see that not all Schur rings over $G\u H$ arise by \corref{thm:thetastackrings}. Mimicking \examref{exam:veebar}, consider $G = \z_2 = \{e_G, g\}$, $H=\z_2 = \{e_H, h\}$, and observe that $G^\theta \u H^\theta$ has the Schur ring $\{g+ h, e_G+e_H, \theta\}$.

The Schur ring constructed above is generalizable. Let $G$ be any zero-semigroup with Schur ring $\S$. Then $G\u G$ contains two distinct copies of $G$, which we denote as $G_1$ and $G_2$. Likewise, the Schur ring $\S$ can be made into a Schur ring over $G_1$ and $G_2$, denoted $\S_1$ and $\S_2$, respectively. For each primitive set $X\in \S$, let $X_1\in \S_1$ and $X_2\in \S_2$ be the corresponding primitive sets. Then, we define $X^+ = X_1\cup X_2$, and $\S^+ = \{X^+\mid X\in \S\}\cup\theta$, which is a partition of $G\u G$. For each $X^+$, $Y^+\in \S^+$, we have $X^+Y^+ = (X_1+X_2)(Y_1+Y_2) = (|X_1||Y_2|+|X_2||Y_1|)\theta + X_1Y_1+X_2Y_2$. Since $X_1Y_1\in \S_1$ and $X_2Y_2\in \S_2$, it holds that $X_1Y_1+X_2Y_2 = (XY)^+ = \sum_{C\in \S} \lambda_CC^+$ for $\lambda_C\in \N$. Therefore, $\S^+$ is a Schur ring over $G\u G$. Note that both $G_1$ and $G_2$ are $(\S_1\vee\S_2)$-subsemigroups, but neither is an $\S^+$-subsemigroup.

By Correspondence (\thmref{thm:Reesextension}), $\S^+$ lifts to a Schur ring over $G\t G$ with the same primitive sets, except $\{\theta_{G_1}, \theta_{G_2}\}$ is now discrete. The coarsening of this lifted Schur ring, where $\{\theta_{G_1}, \theta_{G_2}\}$ is a primitive set, is likewise a Schur ring over $G\t G$.
\end{Exam}

One element extensions of semigroups are ubiquitous, such as $G^e = G\s \z_1$, $G^\theta = \z_1\s G$, $\ORO_{n,1}$, $G\t \z_1$, $G\u \O_2$, and $G\u\CH_2$, which we have seen above, and cloning as discussed in the next section. By combining these extensions, many semigroups can be recursively created one element at a time, but many of these listed extensions\footnote{In the case of $G\u \O_2$, the new element adjoined is indecomposable, as explained in Section \ref{indecomposable}, and has major restrictions on which other elements of $G$ it might be fused with in a Schur ring. Many zero-semigroups $G$ will not satisfy the conditions and will then require that the new element in $G\u \O_2$ be isolated in every Schur ring. But, unlike $G^e$, $G^\theta$, and $G\t \z_1$, it is possible for this new element to be fused with another element in a Schur ring over $G\u \O_2$. For example, $\O_n\u \O_2 \cong \O_{n+1}$, in which case every possible partition on the nonzero elements affords a Schur ring. Clones are likewise indecomposable and have similar considerations as will be shown.

A similar situation holds for $G\u \CH_2$. For many zero-semigroups $G$ will require the new element to be isolated in all Schur rings over $G\u\CH_2$, but for some $G$ this restriction can be overcome. Since $\K_{1,n}\u \CH_2 \cong \K_{1,n+1}$, every possible partition on the nonzero elements affords a Schur ring. Unlike the new element in $G\u \O_2$, the extra element of $G\u \CH_2$ is idempotent. The presence of idempotents in semigroups greatly affects the possible Schur rings over those semigroups, although the authors do not investigate this relationship here.} require the new adjoined element be a singleton in each Schur ring. This makes cases like \thmref{thm:chainschur} quite common in the enumeration of Schur rings over small semigroups. Combining these observations with results like Theorems \ref{thm:nullschur}, \ref{thm:leftnullschur}, \ref{thm:bipartiteschur}, and \ref{thm:oro}, the number of Schur rings over a small semigroup is often a sum or product of Bell numbers, as we will see in Section \ref{gap}.

\section{Indecomposable Elements}\label{indecomposable}
Let $G$ be a semigroup, and let $x\in G$. We say $x$ is \emph{decomposable} if there exists $a,b\in G$ such that $x=ab$. Otherwise, $x$ is \emph{indecomposable}. Note that the indecomposable elements of $G$ are those elements which never appear in the Cayley table of the semigroup.

\subsection{Cloning}\label{sec:clone}
Suppose $G$ is a semigroup and $x\in G$, then \label{clone}$G[x] = G\cup \{x^\prime\}$, where $x^\prime\cdot y=x\cdot y$,  $y\cdot x^\prime =y\cdot x$, and $\left(x^\prime\right)^2=x^2$ for any element $y\in G$. The element $x^\prime$, called a \textit{clone} of $x$, behaves indistinguishably from $x$. There is some freedom in the square of a clone, that is, it is possible to define $(x')^2$ to a different element than $x^2$, but for $G[x]$ we define $(x^\prime)^2=x^2$.\footnote{If $x$ is idempotent, then we present an alternative for $(x')^2$ in Section \ref{sec:idempotent}.}  The associativity of $G[x]$ or $G[\hat{x}]$ is clear.

The cloning process can be iterated, and the order in which one adjoins new clones to a semigroup is immaterial. Let $X\subseteq G$. Then $G[X] = \bigcup_{x\in X} G[x]$. Note that the product of two clones is also well-defined, i.e., if $x'$ and $y'$ are clones of $x$ and $y$, respectively, then $x'y' = xy' = x'y = xy$. If $x''$ is a clone of $x'$ which is a clone of $x$, then $x''y = x'y = xy$ for all $y\in G$. Hence, we need only consider clones of the original semigroup elements, and we consider the subset $X\subseteq G$ to be a multiset. 

\begin{Thm}\label{thm:clonerings}
Let $G$ be a semigroup and $\S$ a Schur ring over $G$. Let $X$ be an $\S$-subset. If $X^\prime$ is the set of clones of elements of $C$ in $G[X]$, then $\S[X] = \S\vee \{X^\prime\}$ is a Schur ring over $G[X]$.
\end{Thm}
Note that $\S[X]$ is the partition of $G[X]$ formed by cloning a $\S$-subset and adjoining it to the original partition $\S$.
\begin{proof}
Note that $\left(X^\prime\right)^2 = X^2\in \QS$,  $X^\prime Y = XY\in \QS$, and $YX^\prime = YX \in \QS$, for all $Y\in \S$. 
\end{proof}

Given $G[X]$, $G$ is an ideal and $G[X]/G \cong \O_{|X|+1}$. Therefore, $G[X]$ is an ideal extension of $G$ by $\O_n$. Then the Schur ring $\S[X]$ is an ideal extension Schur ring using any Schur ring over the ideal $G$ and the trivial Schur ring over $\O_n$, the opposite allowance given by \thmref{thm:Reesextension}. This further illustrates the need for compatibility criteria for the fusion of two Schur rings over two semigroups in an ideal extension.

\thmref{thm:clonerings} can be iterated, cloning one $\S$-subset after another, until we have adjoined the clone of some coarsening of $\S$. 

\begin{Exam} Let $G = \z_2[\z_2] = \{e, g, e', g'\}$. The subgroup $\z_2$ has two (semigroup) Schur rings: $\{e,g\}$ and $\{e+g\}$. This then gives three cloned Schur rings: $\{e,g,e',g'\}$,  $\{e,g, e'+g'\}$, and $\{e+g, e'+g'\}$. By \thmref{thm:indecomp} below, the only other partition to consider over $G$ would be $\{e+g, e', g'\}$, which is not a Schur ring since $(e')^2 = e$ is not a primitive set.
\end{Exam}


\subsection{Indecomposable Imposition}
The \emph{decomposable elements} of a semigroup $G$ form the ideal $G^2=\{ab \mid a,b \in G\}$. Let \label{decomp}$\I(G) = G- G^2$ denote the set of indecomposable elements of $G$. Note that $G/G^2 \cong \O_{|\I(G)|+1}$. Hence, $G$ is an ideal extension of $G^2$ by a null semigroup.

If $G$ is a monoid, then $\I(G)=\emptyset$ since $x=ex=xe$. Hence, indecomposable elements play no role in group theory, let alone Schur rings over a group. Conversely, when a semigroup does have indecomposable elements, they play a critical role on which partitions of the semigroup afford Schur rings and which do not.

\begin{Thm}\label{thm:indecomp} Let $G$ be a semigroup, and let $\S$ be a Schur ring over $G$. Then $\I(G)$ is an $\S$-subset.
\end{Thm}
\begin{proof}
    Suppose not, that is, let $X$ be a primitive set in $\S$ containing elements $x$ and $y$ such that $x$ is indecomposable and $y$ is decomposable. Hence, there are elements $a,b\in G$ such that $y=ab$. Then $y\in AB$, which implies that $X\subseteq AB$. As $x\in X$, there exists some $c\in A$ and $d\in B$ such that $x=cd$, a contradiction.  
\end{proof}

Any partition of $\I(G)$ can appear in a Schur ring by \thmref{thm:Reesextension}, that is, if $\S$ is a partition of $\I(G)$, then $G^2\vee \S$ is a Schur ring over $G$ containing $\S$.

\begin{Exam}\label{exam:nullindecomp} Note that every nonzero element of $\O_n$ is indecomposable. Every Schur ring over $\O_n$ is of the type described above, that is, $\theta \vee \S$. 
\end{Exam}

\thmref{thm:indecomp} tells us that $G^2$ is an $\S$-ideal for every Schur ring $\S$ over semigroup $G$. Of course, $G^3 =\{abc\mid a,b,c\in G\}$ is likewise an ideal of $G$. 
Since $G^3=\bigcup_{A,B,C\in \S} ABC$, $G^3$ is likewise an $\S$-subset. By induction, we have the following corollary.

\begin{Cor}\label{cor:indecomp} Let $\S$ be a Schur ring over a semigroup $G$. Then $G^k$ is an $\S$-ideal for all $k\in \N$.
\end{Cor}

We say that $x\in G$ is a $k$-\emph{decomposable} element if \label{kdecomp}$x\in G^{k}- G^{k+1}$, that is, an element which can be factored into a product of $k$ factors but not into $k+1$ many factors. Let $\I_k(G)$ denote the $k$-decomposable elements of $G$. Note that the $1$-decomposable elements of $G$ are simply just the indecomposable elements of $G$. \corref{cor:indecomp} gives us that $\I_k(G)$ is an $\S$-subset for all $k\in \N$. Let \label{ginfty}$G^\infty = \bigcap_{k=1}^\infty G^k$. As $\I_k(G)\cap \I_\ell(G)=\emptyset$ for $k\neq \ell$, we obtain the following lower bound for the coarsest Schur ring over a semigroup.

\begin{Thm}[Indecomposable Imposition]\label{Thm:indecompseries} Let $\S$ be a Schur ring over semigroup $G$. Then $\S$ is a refinement of the partition
\[\{G^\infty\} \cup \{\{\I_k(G)\}\mid k\in \N\}.\]
\end{Thm}

This shows that the decomposability of elements in a semigroup imposes severe restrictions on those partitions which can be Schur rings.


\begin{Exam}\label{exam:numerical} A \emph{numerical semigroup} is a subsemigroup of the monoid $(\N,+)$. Let $G = \langle 2,3\rangle$ be the numerical semigroup generated by the primes $2$ and $3$. Note that $G = \{2,3,4,5,6,\ldots\}$. The two generators are indecomposable in $G$. In fact, $\I(G) = \{2,3\}$. More generally, the indecomposable elements of a generic numerical semigroup are exactly the elements of the minimum generating set. There are two possible partitions on $\{2,3\}$. By \thmref{thm:Reesextension}, these afford two distinct Schur rings\footnote{We usually use $+$ to denote classes in a partition, given that the operation of the semigroup is multiplication. As the operation of $\N$ is itself $+$, we use the more clunky notation $\{\ \{\ldots\}, \{\ldots\}, \ldots \}$ to avoid confusion.} over $G$, namely
\[\{\ 2, 3, 4, 5, 6, \ldots\ \} \quad\text{and}\quad
\{\ \{2,3\}, 4, 5, 6, \dots\ \}\] These are actually the only two Schur rings over $G$. As observed, $\I_1(G) = \{2,3\}$, but we also note 
\[\I_2(G) = \{4,5\},\ \I_3(G) = \{6, 7\},\ \I_4(G) = \{8,9\}, \ldots \]
Generally, $\I_k(G) = \{2k, 2k+1\}$. Note that for this $G$, in a Schur ring $\S$, if $2$ and $3$ are singletons, then they generate the entire discrete Schur ring $G$, which must be a subring of $\S$ by \propref{prop:singleton}. Suppose that $\{2,3\}$ is a primitive $\S$-set. Since 
$\{2,3\} + \{2,3\} = \{4,6\} + 2\{5\}$, $4$ and $5$ belong to distinct primitive $\S$-sets. Since $\I_2(G) = \{4,5\}$ is an $\S$-subset, it must hold that $4, 5\in \S$ are primitive sets. Similar arguments show that each set $\I_k(G)$ must likewise be discrete, which proves our claim. 
\end{Exam}

\subsection{Monogenic Semigroups}
\begin{multicols}{2}
A \emph{monogenic} (or \emph{cyclic}\footnote{As mentioned above, Leung and Man classified all Schur rings over finite cyclic groups as traditional. The question of extending this result to the infinite cyclic group is what led Misseldine et al. \cite{InfiniteI} to extend the notion of Schur rings to infinite groups. This same question provides genesis again for extending the notion of Schur rings to semigroups as observed here.}) semigroup is a semigroup generated by a single element. For $m,n\in \N$, let \label{zmn}$\z_{m,n} = \langle z\mid z^m=z^{m+n}\rangle$ present the finite monogenic semigroup with \emph{index} $m$ and \emph{period} $n$. Note that the order of a finite monogenic semigroup is given as $|\z_{m,n}| = m+n-1$. The ideal $( z^m)$ is a cyclic subgroup of order $n$, that is, $(z^m) \cong \z_n$. The unique element $z^d$ where $m\le d < m+n$ and $n\mid d$ is the identity of this subgroup and is the unique idempotent of $\z_{m,n}$. The monogenic semigroup \label{monoG}$\z_{m,n}$ is itself a cyclic group if and only if $m=1$, namely $\z_{1,n} \cong \z_n$. \columnbreak

\begin{center}
\begin{tikzpicture}
    \coordinate (A) at (-2,0);  
    \coordinate (B) at (-1,0);  
    \coordinate (C) at (0,0);  
    \coordinate (D) at (1,1);  
    \coordinate (E) at (2,1);  
    \coordinate (F) at (3,0);  
    \coordinate (G) at (2, -1);
    \coordinate (H) at (1, -1);
    \fill (A) circle (2pt) node[below] {$z^1$};
    \fill (B) circle (2pt) node[below] {$z^2$};
    \fill (C) circle (2pt) node[right] {$z^3=z^9$};
    \fill (D) circle (2pt) node[below] {$z^4$};
    \fill (E) circle (2pt) node[below] {$z^5$};
    \fill (F) circle (2pt) node[right] {$z^6$};
    \fill (G) circle (2pt) node[below] {$z^7$};
    \fill (H) circle (2pt) node[below] {$z^8$};

    \draw[->,shorten >=5pt] (A) to (B);
    \draw[->,shorten >=5pt] (B) to (C);
    \draw[->,shorten >=5pt] (C) to [bend left=40] (D);
    \draw[->,shorten >=5pt] (D) to (E);
    \draw[->,shorten >=5pt] (E) to [bend right=320] (F);
    \draw[->,shorten >=5pt] (F) to [bend right=320] (G);
    \draw[->,shorten >=5pt] (G) to (H);
    \draw[->,shorten >=5pt] (H) to [bend right=320] (C);
\end{tikzpicture}\\
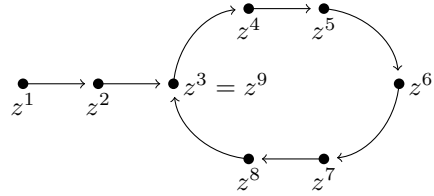
The Cayley graph for $\z_{3,6}$
\end{center}
\end{multicols}

This more diverse family of monogenic semigroups, compared to the subfamily of finite cyclic groups, would presumably contain a more diverse family of Schur rings. Surprisingly, the opposite holds. If $m>1$, then $\I(\z_{m,n}) = z$. Furthermore, $\I_{k}(\z_{m,n}) = z^{k}$ for $k\leq m$, $\I_k(\z_{m,n})=\emptyset$ for $k > m$, and $\z_{m,n}^\infty = (z^m)$.

\begin{Thm}\label{thm:monogenic} If $m>1$, then the only Schur ring over $\z_{m,n}$ is the discrete one.  
\end{Thm}
\begin{proof}
    Let $\S$ be a Schur ring over $\z_{m,n}=\langle z\rangle$. By \thmref{thm:indecomp}, $\I(\z_{m,n}) = z$ is an $\S$-subset. As $z$ generates all of $\z_{m,n}$, the necessarily primitive set $\{z\}$ generates the whole discrete Schur ring $\Z_{m,n}$, which is a subring of $\S$ by \propref{prop:singleton}.
\end{proof}

A similar argument applies to the infinite monogenic semigroup. Let $\Z^+ = \{1, 2,3,4,\ldots\}$ denote the set of positive integers. Then, up to isomorphism, $(\Z^+,+)$ is the only infinite monogenic semigroup.\footnote{This likewise represents the rank 1, free semigroup.} It follows that $\I_{k}(\Z^+) = k$ for all $k\in \Z^+$. By \corref{cor:indecomp}, any Schur ring over $\Z^+$ is necessarily discrete.

\begin{Cor} All Schur rings over monogenic semigroups which are not cyclic groups are discrete. In particular, all Schur rings over a monogenic semigroup are traditional. \end{Cor}

Monogenic monoids can be defined analogously. As every monogenic monoid is isomorphic to $\z_{m,n}^e$ or $\N$, \thmref{thm:monogenic} and its corollary transfer immediately to monogenic monoids. In particular, the Schur ring over a monogenic monoid which is itself not a group is the discrete Schur ring.

\subsection{Nilpotent Semigroups}
A zero-semigroup $G$, with zero $\theta$, is called \emph{nilpotent} if for each $x\in G$, there exists $n\in\N$ such that $x^n=\theta$. If $G$ is finite, this is equivalent to $G^n=\theta$ for some $n$. We say a semigroup is $n$-\emph{nilpotent} if $G^n=\theta$. Note that $G$ is 2-nilpotent if and only if $G\cong\O_n$. \thmref{Thm:indecompseries} provides an important partition to finite nilpotent semigroups since 
\[G > G^2 > G^3 > G^4 >\ldots > \theta\] is a strictly decreasing series of ideals.

\begin{Thm}\label{thm:nil3} Let $G$ be a nilpotent semigroup such that $G^3=\theta$. Then all Schur rings $\S$ over $G$ are constructed by the following manner: first, select any partition $\S|_{\I_1(G)}$ on the indecomposable elements $\I_1(G)$; second, we say two elements of $\I_2(G)$ are equivalent if their multiplicities in $XY$ agree for all primitive sets $X,Y\in \S|_{\I_1(G)}$; third, select any refinement $\S|_{\I_2(G)}$ of this equivalence on  $\I_2(G)$; fourth, $\S = \S|_{\I_1(G)}\vee \S|_{\I_2(G)}\vee \theta$.
\end{Thm}
\begin{proof}
    Let $X,Y\in \S|_{\I_1(G)}$. Note, $\I_1(G)$ is necessarily nonempty, since $G$ is finite. Then $XY$ is, by construction, a linear combination of subsets of equivalent elements of $\I_2(G)$ and $\theta$. Note that all other products in $\S$ equal a multiple of $\theta$. Therefore, $\S$ is closed under multiplication and hence is a Schur ring. By \thmref{Thm:indecompseries}, all Schur rings over $G$ must have this form.
\end{proof}

The vast majority of finite semigroups are nilpotent, and among those the vast majority are 3-nilpotent \cite{Distler}. We describe briefly a simple scheme to generate 3-nilpotent semigroups. For $m,n,k\in \Z$ for which $0\le k<m^{n^2}$, let \label{bigO}$\O(m,n,k) = \{0,1,\ldots, m+n-1\}$ be a semigroup with the following multiplication: let $x,y\in \O(m,n,k)$; if $x<m$ then $xy=yx=0$; otherwise, the product $xy$, where $x,y\ge m$, is determined by an $n\times n$ matrix consisting of entries from $\{0,1,\ldots, m-1\}$. This matrix can be encoded as an integer $k$. If the $n^2$ entries of this matrix are, reading left-to-right and top-to-bottom, $k_{n^2-1},\ldots, k_2, k_1, k_0$, then $k = k_{n^2-1}(m^{n^2-1}) + \ldots + k_2(m^2)+ k_1(m) + k_0$. Note that for all $x,y,z\in \O(m,n,k)$ we have that $x(yz) = (xy)z=0$, which makes $\O(m,n,k)$ a 3-nilpotent semigroup for all $m,n,k$. Note that the choice of integer $k$ is just a random matrix inserted into the Cayley table. Let $M=\{0,\ldots, m-1\}$ and $N=\{m,\ldots, m+n-1\}$. Then $M\cong \O_m$. If $m-1\notin N^2$, then $\O(m,n,k) = \O(m-1,n+1,k')$, where $k'$ is the matrix formed by augmenting $k$ with an extra row and columns of zeros. Thus, we may assume that $NG=M$, ignoring multiplicities.

Schur rings over $G=\O(m,n,k)$ are determined by the restriction imposed by the indecomposable elements $\I_1(G) = N$. Any partition of these elements does afford a Schur ring over $G$, giving at least $\B(n)$ many Schur rings. The possible partitions on $\I_2(G) = M-0$ will be restricted based upon the selected partition on $\I_1(G)$ and the equivalencies imposed by its resulting multiplicities.  Since $M\cong \O_m$, supposing that all possible partitions on $M-0$ are compatible with all the partitions on $N$, we get an upper bound on the number of Schur rings over $G$ at $\B(m-1)\B(n)$. We summarize these observations.

\begin{Thm}\label{thm:bigOschur} Let $\S$ be a Schur ring over $\O(m,n,k)$ for positive integers $m,n,k$. Then $M,N\in \QS$. Furthermore, $\B(n) \le \Omega(\O(m,n,k))\le \B(m-1)\B(n)$.
\end{Thm}

\begin{Exam} \label{exam:nilpotent} 
Let $G = \O(4,2,123)$, whose Cayley table is illustrated to the right. Note that $123 = 1(4^3) + 3(4^2)+2(4) + 3$, as indicated in the red region. First, we observe that $\I_1(G) = \{4,5\}$ and $\I_2(G) = \{1,2,3\}$. There are two possible partitions on $\I_1(G)$. If $\{4,5\}$ is a primitive set of a Schur ring $\S$ over $G$, then $1$ and $2$ are equivalent with no other equivalences on $\I_2(G)$. Hence, 
\[\{ 0, 1+2, 3, 4+5\}\] 

\begin{center}
\begin{tikzpicture}
    \fill[fill=red!70, fill opacity=0.3] (1.75, -2.75) rectangle (2.75, -1.75); 
        \path (0, 0.5) node {$0$} (0.5, 0.5) node {$1$} (1, 0.5) node {$2$} (1.5, 0.5) node {$3$} 
            (2, 0.5) node {$4$} (2.5, 0.5) node {$5$}
        (-0.5, 0) node {$0$} (-0.5, -0.5) node {$1$} (-0.5, -1) node {$2$} (-0.5, -1.5) node {$3$} 
            (-0.5, -2) node  {$4$} (-0.5, -2.5) node {$5$}
        (-1, -0.75) node {$M$} (-1, -2.25) node {$N$}
        (0.75, 1) node {$M$} (2.25, 1) node {$N$};
    \draw (-0.25, 1.25)--(-0.25, -2.75) (-1.25, 0.25)--(2.75, 0.25);
    \draw[dashed]  (1.75, 0.25) -- (1.75, -2.75)
         (-0.25, -1.75) -- (2.75, -1.75);
    \path (0, 0) node {$0$} (0.5, 0) node {$0$} (1, 0) node {$0$} (1.5, 0) node  {$0$} 
            (2, 0) node {$0$} (2.5, 0) node  {$0$}
        (0, -0.5) node {$0$} (0.5, -0.5) node {$0$} (1, -0.5) node {$0$} (1.5, -0.5) node {$0$} 
            (2, -0.5) node {$0$} (2.5, -0.5) node {$0$}
        (0, -1) node {$0$} (0.5, -1) node {$0$} (1, -1) node {$0$} (1.5, -1) node {$0$}
            (2, -1) node {$0$} (2.5, -1) node {$0$}
        (0, -1.5) node {$0$} (0.5, -1.5) node {$0$} (1, -1.5) node {$0$} (1.5, -1.5) node {$0$}
            (2, -1.5) node {$0$} (2.5, -1.5) node {$0$}
        (0, -2) node {$0$} (0.5, -2) node {$0$} (1, -2) node {$0$} (1.5, -2) node {$0$}
            (2, -2) node {$1$} (2.5, -2) node {$3$}
        (0, -2.5) node {$0$} (0.5, -2.5) node {$0$} (1, -2.5) node {$0$} (1.5, -2.5) node {$0$}
            (2, -2.5) node {$2$} (2.5, -2.5) node {$3$};
\end{tikzpicture}
\end{center}

is the coarsest Schur ring over $G$. The partition $\{ 1+2, 3 \}$ on $\I_2(G)$ has only one other refinement: the discrete one; which affords the Schur ring\footnote{This is the same construction as outlined in \thmref{thm:Reesextension}.}
\[\{ 0, 1, 2, 3, 4+5 \}.\]   Lastly, suppose that $4$ and $5$ are singletons in $\S$. Then the partition on $\I_2(G)$ must also be discrete, constructing the discrete Schur ring over $G$. Hence, $G$ has three Schur rings.
\end{Exam}

Note that the construction mentioned in \thmref{thm:nil3} is recursive. After selecting a partition on the indecomposables $\I_1(G)$, we determine an equivalence on $\I_2(G)$ according to the multiplicities generated by all possible products of primitive sets over $\I_1(G)$. We then selected some refinement of that equivalence. 

For general finite nilpotent semigroups, where the sets $\I_k(G)$ may be nonempty, we can generalize this process and define an equivalence on $\I_3(G)$ according to the multiplicities generated by all possible products of primitive sets over $\I_1(G)\cup \I_2(G)$. We then can select any refinement of that equivalence. Inductively, an equivalence is imposed on $\I_k(G)$ according to the multiplicities generated by products of all primitive sets over $\bigcup_{j=1}^{k} \I_{j}(G)$. We may select any refinement of this partition on $\I_k(G)$. Continuing in this manner, we eventually arrive upon $G^n=\theta$, which has only one possible partition. This process recursively constructs a Schur ring over $G$.

Consider now a general semigroup $G$. This same process of recursively imposing restrictions on $\I_k(G)$ based upon the previous $j$-decomposables $\bigcup_{j=1}^{k}\I_j(G)$ is still equally valid, but the coarsest partition imposed on $G^\infty$ by the partitions of $j$-decomposables is not necessarily a Schur ring over $G^\infty$. Following the manner of \thmref{thm:Reesextension}, one compatible partition on $G^\infty$ with the indecomposable imposition is the discrete Schur ring over $G^\infty$.  We call such a Schur ring constructed recursively in this manner a \emph{recursively indecomposable imposition Schur ring} (or a RIISR). 

\thmref{thm:nil3} states that every Schur ring over a $3$-nilpotent semigroup is RIISR. The previous argument generalizes \thmref{thm:nil3} and shows the following:

\begin{Thm}\label{thm:nilriisr} All Schur rings over finite nilpotent semigroups are RIISR.
\end{Thm}

We have seen that nilpotent semigroups of low rank, namely $\O_n$ and $\O(m,n,k)$, can have a relatively high number of Schur rings. As the nilpotency rank increases, the number of Schur rings decreases as the indecomposable imposition of \thmref{thm:nil3} becomes more restrictive. For example, $\z_{n, 1}$ is the unique $n$-nilpotent semigroup of order $n$, which has only one Schur ring.

\begin{Cor} Let $G$ be an $(n-1)$-nilpotent semigroup of order $n$. Then $G$ has exactly two Schur rings.
\end{Cor}
\begin{proof}
    Note that $|\I_1(G)|= 2$ and $|\I_k(G)| = 1$ for $2\le k \le n-1$. By \thmref{thm:nilriisr}, the number of Schur rings over $G$ will be $\B(|\I_1(G)|) = 2$. 
\end{proof}

By similar reasoning to the previous corollary, the number of Schur rings over an $(n-2)$-nilpotent semigroup of order $n$ is less than $\B(3)=5$.

For any semigroup $G$ and $k\in \N$, $G/G^k$ is a nilpotent semigroup. Therefore, studying the Schur rings of nilpotent semigroups will illustrate how this recursive indecomposable imposition affects more general semigroups via ideal extensions.

We recognize that a coarser partition on $G^\infty$ might be compatible with the indecomposable imposition of $G$. If we could wrangle both the recursive imposition on $j$-decomposables and the coarsest Schur ring on $G^\infty$ compatible with the indecomposable imposition, then every Schur ring over a semigroup could be constructed from recursive indecomposable imposition. Of course, without further clarity on the Schur subring over $G^\infty$ or the decomposition series itself, such an observation may be quite trivial. For example, if $G$ is a group, then $G=G^\infty$ and $\I_j(G)=\emptyset$ for all $j$. But when the decomposition series converges to something relatively small, this consideration can be very meaningful, such as is the case for finite nilpotent semigroups. For numerical semigroups, we observe that $G^\infty = \emptyset$. Generalizing the arguments from \examref{exam:numerical}, we see the following:

\begin{Thm} All Schur rings over a numerical semigroup are RIISR.
\end{Thm}

We present next a similar construction to $\O(m,n,k)$ but with a nontrivial $G^\infty$. For $\ell, m,n,k\in \N$ for which $0\le k<\ell^{nm}$, let \label{loo}$\LOO(\ell, m,n,k) = \{0,1,\ldots, \ell+m+n-1\}$ be a semigroup with the following multiplication: let $x,y\in \LOO(\ell, m,n,k)$; if $x<\ell+m$ then $xy=x$; if $x\ge \ell + m$ and if $y<\ell$ or $y\ge \ell + m$ then $xy=0$; otherwise, the product $xy$, where $x\ge \ell +m$ and $\ell \le y <m$, is determined by an integer $k$ which encodes this portion of the Cayley table.\footnote{We admit that for sufficiently large $m$ or $n$, this encoding scheme with an integer $k$ is much more inefficient than simply reporting this portion of the Cayley table, but it works well for the $3$-nilpotent semigroups of order up to $7$ which we wish to discuss.} When $k$ is converted into a base $\ell$ number, it gives the entries of the $n\times m$ submatrix in standard reading order, that is left-to-right, top-to-bottom. Let $L=\{0,\ldots, \ell-1\}$, $M=\{\ell, \ldots, \ell+m-1\}$, and $N=\{\ell+m,\ldots, \ell+m+n-1\}$. Then $L\cong \LO_\ell$, $M\cong \LO_m$, $L\cup M\cong \LO_{\ell+m}$, $N\cup 0 \cong \O_{n+1}$, and $L\cup N\cong \LO_{\ell-1}\t \O_{n+1}$. If $\ell-1\notin NM$, then $\LOO(\ell,m,n,k) = \LOO(\ell-1,m+1,n,k')$ for some potentially new $k'$. Thus, we may assume that $NM=L$, ignoring multiplicities.

Schur rings over $G=\LOO(\ell,m,n,k)$ are determined by the restriction imposed by the indecomposable elements $\I_1(G)=N$, giving $\B(n)$ partitions on $N$. On the other hand, $G^\infty=L\cup M$, but the partition on $N$ places no imposition on $M$, although the choice of $k$ may impose upon $L$. More explicitly, if $\S$ is a Schur ring over $G$, then $L+M$, $N\in \QS$, since $N=\I_1(G)$.  Then $N(L+M) = L$, ignoring multiplicities, which implies $L$, $M\in \QS$. In particular, not every partition on $G^\infty = L\cup M \cong \LO_{\ell+m}$ can be found in a Schur ring. As $M\cong \LO_m$, there $\B(m)$ partitions we can place $M$. Therefore, $\LOO(\ell,m,n,k)$ has at least $\B(m)\B(n)$ many Schur rings, which are counted by assuming the partition on $L$ is discrete. Note that many of these Schur rings are not RIISR, since we have a non-discrete partition on $G^\infty = L\cup M$. To allow a non-discrete partition on $L$, a partition must be selected according to the restrictions imposed by the partition on $N$ but is independent to the partition on $M$. Since $L\cong \LO_\ell$, supposing that all possible partitions on $L-0$ are compatible with all the partitions on $M\cup N$, we get an upper bound on the number of Schur rings over $G$ at $\B(\ell-1)\B(m)\B(n)$. We summarize these observations.

\begin{Thm}\label{thm:looschur} Let $\S$ be a Schur ring over $\LOO(\ell,m,n,k)$ for positive integers $\ell,m,n,k$. Then $L,M,N\in \QS$. Furthermore, $\B(m)\B(n) \le \Omega(\LOO(\ell,m,n,k))\le \B(\ell-1)\B(m)\B(n)$.
\end{Thm}

\begin{Exam}
\begin{multicols}{2}
    We present here the Cayley table for the semigroup $G = \LOO(2,2,2,6)$, illustrated to the right. Note that $\ell=2$ and $k=6 = 0(\ell^3) + 1(\ell^2)+1(\ell) + 0$. The two shaded submatrices indicate the portion encoded by $k$ (red) and the $N\cup 0 = \O_3$ (green). By Theorem \ref{thm:looschur}, there are exactly four Schur rings over $G$. There are two RIISRs over $G$, namely
    \[\{ 0,1,2,3,4+5 \}\quad \text{and}\quad\{ 0,1,2,3,4,5 \}. \] The two other Schur rings over $G$ are  
\begin{center}
    \begin{tikzpicture}
    \fill[fill=red!70, fill opacity=0.3] (0.75, -1.75) rectangle (1.75, -2.75); 
    \fill[fill=green!70, fill opacity=0.3] (1.75, 0.25) rectangle (2.75, -0.25) (-0.25, 0.25) rectangle (0.25, -0.25)
        (1.75, -1.75) rectangle (2.75, -2.75) (-0.25, -1.75) rectangle (0.25, -2.75); 
        \path (0, 0.5) node {$0$} (0.5, 0.5) node {$1$} (1, 0.5) node {$2$} (1.5, 0.5) node {$3$} 
            (2, 0.5) node {$4$} (2.5, 0.5) node {$5$}
        (-0.5, 0) node {$0$} (-0.5, -0.5) node {$1$} (-0.5, -1) node {$2$} (-0.5, -1.5) node {$3$} 
            (-0.5, -2) node  {$4$} (-0.5, -2.5) node {$5$}
        (-1, -0.25) node {$L$} (-1, -1.25) node {$M$} (-1, -2.25) node {$N$}
        (0.25, 1) node {$L$} (1.25, 1) node {$M$} (2.25, 1) node {$N$};
    \draw (-0.25, 1.25)--(-0.25, -2.75) (-1.25, 0.25)--(2.75, 0.25);
    \draw[dashed] (0.75, 0.25) -- (0.75, -2.75) (1.75, 0.25) -- (1.75, -2.75)
        (-0.25, -0.75) -- (2.75, -0.75) (-0.25, -1.75) -- (2.75, -1.75);
    \path (0, 0) node {$0$} (0.5, 0) node {$0$} (1, 0) node {$0$} (1.5, 0) node  {$0$} 
            (2, 0) node {$0$} (2.5, 0) node  {$0$}
        (0, -0.5) node {$1$} (0.5, -0.5) node {$1$} (1, -0.5) node {$1$} (1.5, -0.5) node {$1$} 
            (2, -0.5) node {$1$} (2.5, -0.5) node {$1$}
        (0, -1) node {$2$} (0.5, -1) node {$2$} (1, -1) node {$2$} (1.5, -1) node {$2$}
            (2, -1) node {$2$} (2.5, -1) node {$2$}
        (0, -1.5) node {$3$} (0.5, -1.5) node {$3$} (1, -1.5) node {$3$} (1.5, -1.5) node {$3$}
            (2, -1.5) node {$3$} (2.5, -1.5) node {$3$}
        (0, -2) node {$0$} (0.5, -2) node {$0$} (1, -2) node {$0$} (1.5, -2) node {$1$}
            (2, -2) node {$0$} (2.5, -2) node {$0$}
        (0, -2.5) node {$0$} (0.5, -2.5) node {$0$} (1, -2.5) node {$1$} (1.5, -2.5) node {$0$}
            (2, -2.5) node {$0$} (2.5, -2.5) node {$0$};
    \end{tikzpicture}
\end{center}
\end{multicols}\vspace{-15pt}
\[\{ 0,1,2+3,4+5 \}\quad \text{and}\quad\{ 0,1,2+3,4,5 \}.\] 
\end{Exam}

The above observations and examples establish the following principle: analyzing this series of indecomposable elements is a critical step in studying the Schur rings over semigroups, an analysis completely absent in the category of groups.

\subsection{Indivisible Idempotents}\label{sec:idempotent}
Before closing this section, we want to briefly compare indecomposable elements with \emph{indivisible idempotents}, that is, $\varepsilon^2 = \varepsilon$ is the only factorization of $\varepsilon$.  While they are not indecomposable as defined above, indivisible idempotents do\textemdash in many cases\textemdash behave like indecomposables as they impose restrictions on partitions which are Schur rings. Compare for example $\O_{n+1}$ and $\K_{1,n}$. As outlined in Theorems \ref{thm:nullschur} and \ref{thm:bipartiteschur}, $\O_{n+1}$ and $\K_{1,n}$ have the same number of Schur rings, $\B(n)$, which derive from all possible partitions on the nonzero elements. For $\O_{n+1}$, there are $n$ indecomposable elements, all of whose products are $\theta$. Likewise, for $\K_{1,n}$, there are $n$ indivisible idempotents, but all non-square products are $\theta$. Furthermore, if $I$ is the set of indivisible idempotents, then $G-I$ is an ideal and $G/(G-I)\cong \K_{1,|I|}$. 

As another example, the new identity $e$ adjoined to $G^e$ is an indivisible idempotent. As we saw in \thmref{thm:monoidSchur}, as the only indivisible idempotent, it must be a singleton in each Schur ring.

Similar to indecomposables, there are settings where we can guarantee that indivisible idempotents are isolated from the rest of the semigroup in a Schur ring.

\begin{Thm}\label{thm:indivisibleidempotent} Let $G$ be a commutative semigroup and let $\S$ be a Schur ring over $G$. Then the set of indivisible idempotents of $G$ is an $\S$-subset.
\end{Thm}
\begin{proof}
    Let $X\in \S$ such that $x,\alpha\in X$ where $\alpha$ is an indivisible idempotent and $x$ is not. Consider $X^2$ which contains $\alpha$ with a multiplicity of $1$. Thus, $X\subseteq X^2$. So, there $y,z\in X$ such that $x=yz$. If $y\neq z$, then $x=zy$ and the multiplicity of $x$ in $X^2$ is at least two, a contradiction. Hence, $y=z$ and $x=y^2$.
    If $x\neq y$, then by similar reasoning, there exists some $z\in X$ such that $y=z^2$ and $y\neq z$. Then the multiplicity of $x$ in $X^3$ is at least $3$ since $yzz=zyz=zzy= x$, but the multiplicity of $\alpha$ in $X^3$ is still $1$. Hence, $x=y$, that is, $x^2=x$. As $x$ is divisible, there exists $y,z\in G$ such that $x=yz$ and $z\neq x$. Let $Y,Z\in \S$ be the primitive sets for $y,z$, respectively. Hence, $x\in YZ$, which implies that $\alpha\in YZ$. But this forces $Y=Z=X$. Then $x=x^2=yz=zy\in X^2$, making the multiplicity of $x$ strictly larger than that of $\alpha$, a contradiction.
\end{proof}

We may modify the cloning process introduced in Section \ref{sec:clone} when $x\in G$ is idempotent, specifically, we define $(x^\prime)^2=x^\prime$ and denote this semigroup \label{Iclone}$G[\hat{x}] = G\cup \{x'\}$, where $x'$ is an idempotent clone of $x$. The associativity of $G[\hat{x}]$ is clear.

\begin{Exam} The behavior of the Schur rings over $G[x]$ versus $G[\hat{x}]$ can differ. For example, consider $G=\O_n$ and $x=\theta$. Then $\O_n[\theta]\cong \O_{n+1}$ and $\Omega(\O_n[\theta]) = \B(n)$ by Theorem \ref{thm:nullschur}. Whereas in $\O_n[\hat{\theta}]$, which is commutative, $\theta'$ is the unique indivisible idempotent, which means $\theta'$ is a singleton in every Schur ring over $\O_n[\hat{\theta}]$, by Theorem \ref{thm:indivisibleidempotent}. Likewise, $\theta$ is a singleton in each Schur ring since $\O_n[\hat{\theta}]$ is a zero semigroup. Lastly, $\I(\O_n[\hat{\theta}]) = \O_n[\hat{\theta}] - \{\theta,\theta'\}$. Hence, a partition $\S$ over $\O_n[\hat{\theta}]$ is a Schur ring if and only if $\theta, \theta'\in \S$, similar to Example \ref{exam:nullindecomp}. Therefore, $\Omega(\O_n[\hat{\theta}]) = \B(n-1)$.
\end{Exam}

To demonstrate that the commutativity hypothesis is necessary in Theorem \ref{thm:indivisibleidempotent}, we present an example of a noncommutative semigroup $\OROP_n$ with two idempotent elements, $\alpha$ and $\beta$. In $\OROP_n$, $\beta$ is indivisible but $\alpha$ is not, but $\alpha+\beta$ is a primitive set in some Schur rings over $\OROP_n$.

\begin{Exam}
    For $n\ge 2$, we define a semigroup \label{orop}$\OROP_n = \{\theta,x_1,\ldots, x_{n-1}\}\cup\{\alpha,\beta\}$ by the following multiplication rules. For convenience,  let $X = \{\theta,x_1,\ldots, x_{n-1}\}$. For each $x,y\in X$, we define $xy=\theta$. Hence, $X\cong \O_n$. We define $x\alpha = \theta$, $\alpha x= x$, and $\alpha^2=\alpha$. Hence, $X\cup \alpha \cong\ORO_{n,1}$. We define $x\beta = x$, $\beta x = \theta$, and $\beta^2 = \beta$. Hence, $X\cup \beta\cong \ORO_{n,1}^\text{op}$. Lastly, $\alpha\beta =x_1$ and $\beta\alpha=\alpha$. It is routine to verify this multiplication is associative. We note that $X$ is an ideal of $\OROP_n$, and so $\OROP_n$ is an ideal extension of $\O_n$ by $K_{1,2}$. 


  \begin{Thm}\label{thm:orop} Let $\S$ be a partition over $\OROP_{n}$, for some $n\ge2$. Then $\S$ is a Schur ring if and only if $\theta$, $x_1$, $\alpha+\beta\in \QS$.\footnote{Note that $\alpha+\beta\in \QS$ does not imply that $\alpha+\beta$ is a primitive set, as we are not saying $\alpha+\beta\in \S$. Instead, if $\alpha+\beta$ is not a primitive set, then $\alpha$ and $\beta$ are. Hence, we require $\alpha+\beta\in \S$ or $\alpha,\beta\in\S$.} The number of Schur rings over $\OROP_{n}$ is $2\B(n-2)$.
\end{Thm}
    \begin{proof}
    Recall $\alpha$ and $\beta$ are indivisible idempotents. Let $\S$ be a Schur ring over $G=\OROP_{n}$, and let $A,B\in \S$ be the primitive sets containing $\alpha$ and $\beta$, respectively. Suppose that $A=B$, that is, $A = \alpha+\beta +C$ for some $C\subseteq \O_n$. Then $A^2 = \alpha + \beta + (|C|^2+2|C|+1)\theta + x_1 + 2C$. If $|C|\neq 0$, then the elements of $C$ have a multiplicity in the product $A^2$ higher than that of $\alpha$ and $\beta$, which is $1$. Thus, $C=\emptyset$. Suppose that $A=\alpha+C$ and $B=\beta +D$ for some $C,D\subseteq \O_n$, where $C\cap D=\emptyset$. Then $AB = C+D+|C||D|\theta + x_1$. If $|C|\neq 0$, then the elements of $C$ have a multiplicity in this product higher than that of $\alpha$, which is zero. Thus, $C=\emptyset$. Similarly, $D=\emptyset$. Therefore, $\{\alpha,\beta\}$ is an $\S$-subset for any Schur ring $\S$. Likewise, its complement, namely $\O_n$, is an $\S$-subsemigroup. By \thmref{thm:nullschur}, $\theta \in \S|_{\O_n}$. Finally, $(\alpha+\beta)^2 = \theta + x_1 + \alpha+\beta$. Therefore, $x_1\in\S$.

    Conversely, suppose that $\theta$, $x_1$, $\alpha+\beta\in \QS$. Since $\O_n=G-(\alpha+\beta)$ is an $\S$-ideal and $\theta\in \S|_{\O_n}$, $\S|_{\O_n}$ is a Schur ring over $\O_n$. Note $\theta\alpha = \alpha\theta=\theta\beta=\beta\theta = \theta \in \S$. Let $C\in \S|_{\O_n}$ be a nonzero primitive set. Then $\alpha C = C\beta = C$, $C\alpha = \beta C = |C|\theta\in \QS$. Independent of the primitive set containing $\alpha$ or $\beta$, this shows that the product between it and $C$ is contained in $\Q[\S|_{\O_n}]$. If $\alpha$ is a primitive set, then $\beta$ is likewise primitive and $\alpha^2=\alpha, \alpha\beta=x_1$, $\beta\alpha=\theta$, $\beta^2=\beta\in \S$. If $\alpha+\beta$ is a primitive set, then $(\alpha+\beta)^2 = \theta+x_1+(\alpha+\beta)\in \S$. Therefore, $\S$ is closed under multiplication.
    
    There are two choices for the primitive set containing $\alpha$, namely $\alpha$ or $\alpha+\beta$. There are $\B(n-2)$ Schur rings over $\O_n$ if $\theta$ and $x_1$ are necessarily singletons. By \thmref{thm:nullschur}, all these partitions are Schur rings over $\O_n$. 
    \end{proof}
\end{Exam}

By modifying the construction of $\LOO(\ell,m,n,k)$, we can construct a band where certain subsemigroups impose restriction on the rest of the semigroup, although none of the idempotents are indivisible. Let \label{loro}$\LORO(\ell,m,n,k) = \{0,1,\ldots, \ell+m+n-1\}$ be a semigroup with the following multiplication: let $x,y\in \LORO(\ell, m,n,k)$; if $x<\ell+m$ then $xy=x$; if $x\ge \ell + m$ and if $y<\ell$ or $y\ge \ell + m$ then $xy=y$; otherwise, the product $xy$, where $x\ge \ell +m$ and $\ell \le y <m$, is determined by an integer $k$ which encodes this portion of the Cayley table as a number which, when written in base $\ell$, fills in the $n\times m$ portion of the Cayley table in standard reading order. Let $L=\{0,\ldots, \ell-1\}$, $M=\{\ell, \ldots, \ell+m-1\}$, and $N=\{\ell+m,\ldots, \ell+m+n-1\}$. Then $L\cong \LO_\ell$, $M\cong \LO_m$, $L\cup M\cong \LO_{\ell+m}$, $N \cong \RO_n$, and $L\cup N\cong \LO_{\ell}\s \RO_n$. Unlike $\O(m,n,k)$ and $\LOO(\ell,m,n,k)$, a change of $k$ cannot guarantee that $NM=L$, but we still have $NM\subseteq L$. Despite neither $N$ nor $M$ containing any indivisible idempotents, the choice of $k$ still imposes restrictions on compatible partitions on $L$. 

\begin{Thm}\label{thm:loroschur} Let $\S$ be a Schur ring over $\LORO(\ell,m,n,k)$ for positive integers $\ell,m,n,k$. Then $L,M,N\in \QS$. Furthermore, $\B(m)\B(n) \le \Omega(\LORO(\ell,m,n,k))\le \B(\ell)\B(m)\B(n)$.
\end{Thm}
\begin{proof}
    If $\S$ is a Schur ring over $\LORO(\ell,m,n,k)$, consider $X\in S$ such that $X\cap N\neq \emptyset$. Then $X^2 = (X_L+X_M+X_N)^2 = (|X|+|X_N|)X_L+|X|X_M + |X_N|X_N + X_NX_M$. Hence, $X\subseteq X^2$. As $X_NX_M\subseteq L$, it must be that $|X_N|=|X|$, which implies that $X_L=X_M=\emptyset$. Therefore, $N\in \QS$. Consider next $N(L+M) = nL + NM\subseteq L$. Therefore, $L, M\in \QS$. 
\end{proof}

\begin{Exam}
\begin{multicols}{2}
    We present here the Cayley table for the semigroup $G = \LORO(2,1,2,1)$, illustrated to the right. The shaded (red) submatrix indicates the portion encoded by $k$. By Theorem \ref{thm:loroschur}, there are two to four Schur rings over $G$. The two guaranteed Schur rings are afforded by a discrete partition on $L$, namely
    \[\{ 0,1,2,3+4 \}\quad \text{and}\quad\{ 0,1,2,3,4 \}. \] The two other possible partitions are  
\begin{center}
    \begin{tikzpicture}
    \fill[fill=red!70, fill opacity=0.3] (0.75, -1.25) rectangle (1.25, -2.25); 
        \path (0, 0.5) node {$0$} (0.5, 0.5) node {$1$} (1, 0.5) node {$2$} (1.5, 0.5) node {$3$} 
            (2, 0.5) node {$4$}
        (-0.5, 0) node {$0$} (-0.5, -0.5) node {$1$} (-0.5, -1) node {$2$} (-0.5, -1.5) node {$3$} 
            (-0.5, -2) node  {$4$}
        (-1, -0.25) node {$L$} (-1, -1) node {$M$} (-1, -1.75) node {$N$}
        (0.25, 1) node {$L$} (1, 1) node {$M$} (1.75, 1) node {$N$};
    \draw (-0.25, 1.25)--(-0.25, -2.25) (-1.25, 0.25)--(2.25, 0.25);
    \draw[dashed] (0.75, 0.25) -- (0.75, -2.25) (1.25, 0.25) -- (1.25, -2.25)
        (-0.25, -0.75) -- (2.25, -0.75) (-0.25, -1.25) -- (2.25, -1.25);
    \path (0, 0) node {$0$} (0.5, 0) node {$0$} (1, 0) node {$0$} (1.5, 0) node  {$0$} 
            (2, 0) node {$0$}
        (0, -0.5) node {$1$} (0.5, -0.5) node {$1$} (1, -0.5) node {$1$} (1.5, -0.5) node {$1$} 
            (2, -0.5) node {$1$}
        (0, -1) node {$2$} (0.5, -1) node {$2$} (1, -1) node {$2$} (1.5, -1) node {$2$}
            (2, -1) node {$2$}
        (0, -1.5) node {$0$} (0.5, -1.5) node {$1$} (1, -1.5) node {$0$} (1.5, -1.5) node {$3$}
            (2, -1.5) node {$4$} 
        (0, -2) node {$0$} (0.5, -2) node {$1$} (1, -2) node {$1$} (1.5, -2) node {$3$}
            (2, -2) node {$4$};
    \end{tikzpicture}
\end{center}
\end{multicols}\vspace{-15pt}
\[\{ 0+1,2,3+4 \}\quad \text{and}\quad\{ 0+1,2,3,4\},\] but this last partition is not a Schur ring over $G$ since $2(3) = 0$. Therefore, there are only three Schur rings over $\LORO(2,1,2,1)$. This illustrates an imposition by elements which are neither indecomposable nor indivisible. But $L\cup M$ is an ideal of $G$, and $L$ is a left ideal. Further study on how the complement of an ideal imposes restrictions onto the ideal in a Schur ring is necessary. 
\end{Exam}

\section{Rostered Semigroups}\label{sec:roster}
We introduce one last extension of semigroups and its correspondence with their Schur rings, which we call \emph{rosters}. Rosters will be a generalization of the idea of cloning, but the new element may have some ``mutation'' to it. In particular, these new mutants are not necessarily indecomposable, which was discussed in Section \ref{indecomposable}. 

For a semigroup $G$, we say that an ideal $P$ is \emph{prime}\footnote{In semigroup theory, an ideal $P$ is \emph{prime} when for all ideals $A,B\subseteq G$ if $AB\subseteq P$ then $A\subseteq P$ or $B\subseteq P$. What was defined above as a prime ideal is instead called a \emph{completely prime ideal} (see \cite{Park}). For commutative semigroups, the two notions coincide, but they differ for noncommutative semigroups. As we only utilize the notion of completely prime ideals, we use the shorter label of prime instead.} when for all $x,y\in G$ if $xy\in P$ then $x\in P$ or $y\in P$. Note that an ideal $P$ of $G$ is prime if and only if $G-P$ is a subsemigroup of $G$. Let $x\in G$. We let \label{lstab}$G_x = \{y\in G\mid yx=x\}$ and \label{rstab}$_xG = \{y\in G\mid xy=x\}$, be the \emph{left-} and \emph{right-stabilizer} of $x$. These are both subsemigroups of $G$. With these concepts, we can build a roster.

Let $G$ be a semigroup with $g\in G$. Let $H\le G_x$ and $K\le\!_xG$ be subsemigroups such that their complements are prime ideals and one of the following conditions hold, called condition (*) \begin{equation}\tag{*}\label{eqn:*} \text{(a) } H\le K, \text{ (b) } K\le H, \text{ (c) } x\in H\cap K,\text{ or (d) } x\notin H\cup K.\footnote{The conditions that we are avoiding here is there exists some element $y\in H-K$ and $z\in K-H$. If $y\in H-K$, associativity requires $(x')^2 = x'(yx') = (x'y)x' = (xy)x' = x(yx') = xx'$. Similarly, if $z\in K-H$, associativity requires $(x')^2 = (x'z)x' = x'(zx') = x'(zx) = (x'z)x=x'x$. Hence, if (a) and (b) fail, then it must be that $x'x = (x')^2 = xx'$. If $x\in H$, then $xx'=x'$; otherwise, $xx'=x^2$. Hence, we need condition (c) or (d). Therefore, for conditions (c) and (d), $x'x=xx'$ and $(x')^2$ is determined solely by this condition, but for conditions (a) or (b), $x'x\neq xx'$ if and only if $x\in K-H$ or $x\in K-H$, respectively. Hence, $(x')^2$ equals $x'$ or $x$, as $x$ is necessarily idempotent, depending on this criterion.}\end{equation} Then \label{roster1}$G\r{x} (H,K) = G\cup x'$, where $x'$ is a clone of $x$. When $H=K$, we abbreviate this as \label{roster2}$G\r x H$. The set $G\r x (H,K)$ is made into a semigroup by extending the associative multiplication of $G$ in the following way. If $y\in H$, then $yx' = x'$; otherwise if $y\not\in H$, $yx'=yx$. Likewise, if $z\in K$, then $x'z=x'$; otherwise for $z\not\in K$, $x'z=xz$. Essentially, $x'$ behaves like a clone of $x$, like above, but when an element stabilizes $x$ it might also stabilize $x'$, depending on the roster $(x,H,K)$. Finally, we define $x'x'$ according to \eqref{eqn:*} from before: \[\text{(a) } x'x'=x'x, \text{ (b) } x'x'=xx', \text{ (c) } x'x'=x',\text{ or (d) } x'x'=x^2.\]

\begin{Prop}\label{prop:roster} Let $G$ be a semigroup and $x\in G$. If $H$, $K\le G$ such that $H\le G_x$, $K\le\,_xG$, $G-H$ and $G-K$ are prime ideals, and condition (*), then $G\r x(H,K)$ is a semigroup.
\end{Prop}
\begin{proof}
     The multiplication on the subset $G$ is associative by hypotheses. We need only check for associativity in products involving $x'$.
\begin{align*}
a,b\in H:\ &a(bx') = ax' = x' = (ab)x' & \text{since } ab\in H\\
a\in G-H, b\in H:\ &a(bx') = ax' = ax = a(bx) = (ab)x = (ab)x' & \text{since } ab\in G-H\\
a\in G,b\in G-H:\ &a(bx') = a(bx) = (ab)x = (ab)x' & \text{since } ab\in G-H\\
a\in H,b\in K:\ &a(x'b) = ax' = x' = x'b = (ax')b & \\ 
a\in G-H, b\in K:\ &a(x'b) = ax' = ax = a(xb) = (ax)b = (ax')b & \\
a\in H, b\in G-K:\ &a(x'b) = a(xb) = (ax)b = xb = x'b = (ax')b & \\
a\in G-H,b\in G-K:\ &a(x'b) = a(xb) = (ax)b = (ax')b & \\ 
a,b\in K:\ &x'(ab) = x' = x'a = (x'a)b & \text{since } ab\in K\\
a\in K, b\in G-K:\ &x'(ab) = x(ab) = (xa)b = xb = x'b = (x'a)b & \text{since } ab\in G-K\\
a\in G-K,b\in G:\ &x'(ab) = x(ab) = (xa)b = (x'a)b & \text{since } ab\in G-K
\end{align*}
Associativity checks involving two $x'$ are a bit more involved, but we may assume associativity when at least two elements of $G$ are involved. If (a), then $H\le K$ and $x'x'=x'x$.
 \begin{align*}
    \text{(a) }&& a\in K:\ & x'(x'a) = x'x' = x'x = x'(xa)= (x'x)a = (x'x')a& \\
    &&a\in G-K:\ & x'(x'a) = x'(xa) = (x'x)a = (x'x')a &\\
    &&a\in H:\ &x'(ax') = x'x' = (x'a)x'  & \text{since } a\in K\\
    &&a\in K-H:\ &x'(ax') = x'(ax) = (x'a)x = x'x = x'x' = (x'a)x'& \\
    &&a\in G-K:\ & x'(ax') = x'(ax) = (x'a)x = (xa)x = (xa)x' = (x'a)x' & \text{since } a, xa\in G-K\le G-H\\
    &&x, a\in H:\ & a(x'x')= ax' = x' = x'x' = (ax')x' & \text{since } (x')^2= x'x = x' = xx' \\
    &&x\in H, a\in G-H:\ & a(x'x') = a(xx') = (ax)x' =(ax')x' &\\
    &&x\in H:\ & x'(x'x') = x'x' = (x'x')x'& \\
    && x\in K-H, a\in H:\ & a(x'x')= ax' = x' = x'x' = (ax')x'  & \text{since } (x')^2 = x'x= x'\\
    &&x\in K-H, a\in G-H:\ &a(x'x') = a(x'x) = (ax')x =(ax)x = a(xx) = a(xx') & \text{but } xx' = x^2\\
    &&&=  (ax)x' =(ax')x'&\\
    &&x\in K-H:\ & x'(x'x') = x'x' = (x'x')x'&\\
    &&x\in G-K, a\in H:\ & a(x'x') =a(xx) = (ax)x = xx  = x'x'=   (ax')x' & \text{since } (x')^2 = x'x= x^2 = xx'\\
    &&x\in G-K,  a\in G-H:\ & a(x'x') = a(xx') = (ax)x' =(ax')x' &\\
    &&x\in G-K:\ & x'(x'x') = x'(xx) = (x'x)x = x^3 = x(xx') = (xx)x' &\\
    &&&= (x'x')x'
\end{align*}
The case (b) is handled similarly. If (c), then $(x')^2=x'=x'x=xx'$. Then the only associativity check from (a) that do not carry over (barring the cases that do not apply) to (c) is $x'ax'$: 
\begin{align*}
    \text{(c) }&& a\in H\cap K:\ & x'(ax') = x'x' = (x'a)x' &\\
    && a\in H- K:\ & x'(ax') = x'x' = xx' = x(ax') = (xa)x' = (x'a)x' &\\
    && a\in K- H:\ & x'(ax') = x'(ax) = (x'a)x = x'x = x'x' = (x'a)x' &\\
    && a\in G- (H\cup K):\ & x'(ax') = x'(ax) = (x'a)x = (xa)x = x(ax)= x(ax') = (xa)x' = (x'a)x' &
\end{align*} Lastly, if (d), then $(x')^2 = x^2 = x'x=xx'$. Hence, all the necessary associativity checks are already considered in (a) and (c). Therefore, the multiplication is associative.
\end{proof}

\begin{Exam}
Consider the semigroup $G=\CH_2\t \z_1 = \{0,1,2\}$. Note that $0y=0$ and $2y=2$ for all $y\in G$, but $1\cdot 0 = 1\cdot 2 = 0$ and $1^2=1$. We provide three different rosters of this semigroup, namely $(\CH_2\t\z_1) \r{0}(\z_1, \CH_2\t\z_1)$, $(\CH_2\t\z_1) \r{1}(\emptyset, \z_1)$, and $(\CH_2\t\z_1) \r{2}(\emptyset,\z_1)$ (all three Cayley tables are found below) and observe vastly different behaviors with regard to the semigroups' Schur rings. Note that $\CH_2\t \z_1$ only has a single Schur ring, which is necessarily discrete.

We will begin with $(\CH_2\t\z_1) \r{0}(\z_1,\CH_2\t\z_1)$, where $H=\z_1=\{1\} \leq \{0, 1\} = G_0$ and $K=G=\,_0G$. Note that $G - H = \{0,2\}$ and $G-K=\emptyset$ are prime ideals\footnote{In this case, it is unambiguous to denote $\langle 1\rangle=\{1\}$ simply as $\z_1$, since $H$ cannot possibly be $\langle 0 \rangle = \{0\}$, despite $\langle 0\rangle \cong\z_1$,  because $G-\langle 0\rangle$ is not prime.} of $G$. When adjoining this new element $0'$ to the Cayley table of $\CH_2\t\z_1$, we copy the rows and columns of $0$ except according to the roster $(H,K)$, where $H$ dictates where in the column of $0'$ we change $0$ into $0'$ and likewise $K$ dictates where in the row of $0'$ we change $0$ into $0'$. Also, we note $H\leq K$, hence condition (a) of \eqref{eqn:*} is satisfied. This means that adjoining the element $0^\prime$ to $G$ extends the associative multiplication as follows: $0^\prime \cdot y = 0^\prime$ for all $y\in K$; similarly, $1\cdot 0^\prime = 0^\prime$ since $1\in H$; $0\cdot 0^\prime =0$ and $2\cdot 0^\prime=2$, because $0, 2\not\in H$; finally, $\left(0^\prime\right)^2 = 0^\prime \cdot 0 = 0^\prime$ by (a). Note that $(\CH_2\t\z_1) \r{0}(\z_1,\CH_2\t\z_1)\cong\LORO(2,1,1,0)$ and has only one Schur ring, the discrete one, an example of obtaining the lower bound of Theorem \ref{thm:loroschur}.

Continuing to the second semigroup $(\CH_2\t\z_1) \r{1}(\emptyset, \z_1)$, where $H=\emptyset$, $K=\,_1G = \{1\}$, and $G-K=\{0,2\}$ is a prime ideal.\footnote{Note that for any semigroup $G$ and any $x\in G$, the stabilizer $H=\emptyset$ is always compatible for rosters since $\emptyset\le G_x$ and $G-\emptyset = G$ is a prime ideal of $G$. Likewise for $K=\emptyset$.} Again, condition (a) holds. Hence, $1^\prime \cdot 1 = 1^\prime$, since $1\in K$, but $1^\prime\cdot 0 = 1\cdot 0 = 0$ and $1^\prime\cdot 2 = 1\cdot 2 = 0$, since $0, 2\not\in K$. Further, $y\cdot 1^\prime = y\cdot 1$ for all $y\in \CH_2\t\z_1$ since $H=\emptyset$ implies nothing stabilizes $1^\prime$ on the left. By (a) we have $(1^\prime)^2 = 1^\prime \cdot 1 = 1^\prime$. Note that $(\CH_2\t\z_1) \r{1}(\emptyset, \z_1)\cong\LO^\theta_2 \t \z_1$ and has two Schur rings, corresponding to the two possible twists between the two Schur rings over $\LO_2^\theta$ and the single Schur ring over $\z_1$, the discrete one, as expected by Theorem \ref{thm:veebarschur}.

For the third example, consider $(\CH_2\t\z_1) \r{2}(\emptyset, \z_1)$, where $H=\emptyset$, $K=\langle 1\rangle\le\,_2G = \{0,1,2\}$. The roster on $\CH_2\t\z_1$ is the same as the second semigroup except that the cloned element is $2$. Note that $2^\prime \cdot1 = 2^\prime$, but $2'\cdot y=2\cdot y$ and $y\cdot 2'=y\cdot 2$ in all other cases. In particular, $2'\cdot2 = 2 = (2')^2\neq 2'$. This roster cannot be described using any of the other semigroup constructions we have introduced in this paper, and it contains three Schur rings, namely:
\[\{0+2, 1+2'\},\quad \{0,2, 1+2'\}, \quad \{0,1,2,2'\}\] 

\begin{center}
\begin{tabular}{cccc}
\begin{tikzpicture}
    \fill[fill=red!70, fill opacity=0.3] (1.25, 0.25) rectangle (1.75, -1.25);
    \fill[fill=green!70, fill opacity=0.3] (-0.25, -1.25) rectangle (1.25, -1.75);
    \fill[fill=mixed, fill opacity=0.3] (1.25, -1.25) rectangle (1.75, -1.75); 
        \path (0, 0.5) node {$0$} (0.5, 0.5) node {$1$} (1, 0.5) node {$2$} (1.55, 0.5) node {$0^\prime$} (-0.5, 0) node {$0$} (-0.5, -0.5) node {$1$} (-0.5, -1) node {$2$} (-0.45, -1.5) node {$0^\prime$};
    \draw (-0.25, 0.75)--(-0.25, -1.75) (-0.75, 0.25)--(1.75, 0.25);
    \draw[dashed]  (1.25, 0.25) -- (1.25, -1.75)
         (-0.25, -1.25) -- (1.75, -1.25);
    \path (0, 0) node {$0$} (0.5, 0) node {$0$} (1, 0) node {$0$} (1.5, 0) node  {$0$} 
        (0, -0.5) node {$0$} (0.5, -0.5) node {$1$} (1, -0.5) node {$0$} (1.55, -0.5) node {$0^\prime$}
        (0, -1) node {$2$} (0.5, -1) node {$2$} (1, -1) node {$2$} (1.5, -1) node {$2$}
        (0.05, -1.5) node {$0^\prime$} (0.55, -1.5) node {$0^\prime$} (1.05, -1.5) node {$0^\prime$} (1.55, -1.5) node {$0'$};
\end{tikzpicture}
&
\begin{tikzpicture}
    \fill[fill=red!70, fill opacity=0.3] (1.25, 0.25) rectangle (1.75, -1.25);
    \fill[fill=green!70, fill opacity=0.3] (-0.25, -1.25) rectangle (1.25, -1.75);
    \fill[fill=mixed, fill opacity=0.3] (1.25, -1.25) rectangle (1.75, -1.75); 
        \path (0, 0.5) node {$0$} (0.5, 0.5) node {$1$} (1, 0.5) node {$2$} (1.55, 0.5) node {$1^\prime$} (-0.5, 0) node {$0$} (-0.5, -0.5) node {$1$} (-0.5, -1) node {$2$} (-0.45, -1.5) node {$1^\prime$};
    \draw (-0.25, 0.75)--(-0.25, -1.75) (-0.75, 0.25)--(1.75, 0.25);
    \draw[dashed]  (1.25, 0.25) -- (1.25, -1.75)
         (-0.25, -1.25) -- (1.75, -1.25);
    \path (0, 0) node {$0$} (0.5, 0) node {$0$} (1, 0) node {$0$} (1.5, 0) node  {$0$} 
        (0, -0.5) node {$0$} (0.5, -0.5) node {$1$} (1, -0.5) node {$0$} (1.5, -0.5) node {$1$}
        (0, -1) node {$2$} (0.5, -1) node {$2$} (1, -1) node {$2$} (1.5, -1) node {$2$}
        (0, -1.5) node {$0$} (0.55, -1.5) node {$1^\prime$} (1, -1.5) node {$0$} (1.55, -1.5) node {$1^\prime$};
\end{tikzpicture}
&
\begin{tikzpicture}
    \fill[fill=red!70, fill opacity=0.3] (1.25, 0.25) rectangle (1.75, -1.25);
    \fill[fill=green!70, fill opacity=0.3] (-0.25, -1.25) rectangle (1.25, -1.75);
    \fill[fill=mixed, fill opacity=0.3] (1.25, -1.25) rectangle (1.75, -1.75); 
        \path (0, 0.5) node {$0$} (0.5, 0.5) node {$1$} (1, 0.5) node {$2$} (1.55, 0.5) node {$2^\prime$} (-0.5, 0) node {$0$} (-0.5, -0.5) node {$1$} (-0.5, -1) node {$2$} (-0.45, -1.5) node {$2^\prime$};
    \draw (-0.25, 0.75)--(-0.25, -1.75) (-0.75, 0.25)--(1.75, 0.25);
    \draw[dashed]  (1.25, 0.25) -- (1.25, -1.75)
         (-0.25, -1.25) -- (1.75, -1.25);
    \path (0, 0) node {$0$} (0.5, 0) node {$0$} (1, 0) node {$0$} (1.5, 0) node  {$0$} 
        (0, -0.5) node {$0$} (0.5, -0.5) node {$1$} (1, -0.5) node {$0$} (1.5, -0.5) node {$0$}
        (0, -1) node {$2$} (0.5, -1) node {$2$} (1, -1) node {$2$} (1.5, -1) node {$2$}
        (0, -1.5) node {$2$} (0.55, -1.5) node {$2'$} (1, -1.5) node {$2$} (1.5, -1.5) node {$2$};
\end{tikzpicture}
\\
$(\CH_2\t\z_1) \r{0}(\z_1, \CH_2\t\z_1)$ & $(\CH_2\t\z_1) \r{1}(\emptyset, \z_1)$ & $(\CH_2\t\z_1) \r{2}(\emptyset,\z_1)$\\
\end{tabular}
\end{center}
\end{Exam}

Note above that in all cases in \propref{prop:roster} we have $xx'$ is equal to $x'$ or $x^2$. In the special case where $H=K$ (hence (c) or (d) from (*)) and $x^2=x$, we may define an alternative semigroup \label{rosterhat}$G\rhat x H$ identically to $G\r x H$ except that $$x'x'=\begin{cases} x', &x'x=x^2\\ x^2, &x'x=x'.\end{cases}$$
Since $H=K$, we know that $x'x=xx'$. 

\begin{Prop} Let $G$ be a semigroup and $x\in G$. If $H\le G$ such that $H\le G_x\cap\,_xG$, $G-H$ is a prime ideal, and $x^2=x$, then $G\rhat xH$ is a semigroup.
\end{Prop}
\begin{proof}
Most of the associative checks in \propref{prop:roster} apply with no change, but we have included below those that have modified calculations.
 \begin{align*}
    x\in H:\ & x'(x'x') = x'(xx) = (x'x)x = x'x\\ =\ &x' =xx' =x(xx') = (xx)x' =(x'x')x'&&\text{since } xx'=x'x=x',\\
    x, a\in H:\ &a(x'x') = a(xx) = (ax)x = xx = x'x' = (ax')x' && (x')^2=x^2 \\
    x\in H, a\in G-H:\ &a(x'x') = a(xx) = (ax)x = (ax)x' = (ax')x'&&\text{since } ax\in G-H\\
    x, a\in H:\ &x'(x'a) = x'x' = xx = x(xa) = (xx)a = (x'x')a && \\
    x\in H, a\in G-H:\ &x'(x'a) = x'(xa) = (x'x)a = (xx)a = (x'x')a&&\\    
    x\in G-H:\ & x'(x'x') = x'x' = (x'x')x'&&\text{since } xx'=x'x=x^2,\\
    x\in G-H, a\in H:\ &a(x'x') = ax' = x' = x'x' = (ax')x' && (x')^2=x' \\
    x,a \in G-H:\ &a(x'x') = ax' =ax = a(xx) = a(xx') = (ax)x' = (ax')x'&&\text{since } ax\in G-H\\
    x\in G-H, a\in H:\ &x'(x'a) = x'x' = x' = x'a = (x'x')a && \\
    x,a\in G-H:\ &x'(x'a) = x'(xa) = (x'x)a = (xx)a = xa= x'a= (x'x')a&&
\end{align*}
Therefore, the multiplication is associative.
\end{proof}

\begin{Exam}
Consider the semigroup $G=\ORO_{2,1}=\{\theta, 1,2\}$ and $x=\theta$. As a zero, $G_\theta=\!_\theta G=G$. Let $H=K=\{2\}$. Then $G-H = \{\theta,1\}$ is a prime ideal of $G$.  We construct both $\ORO_{2,1} \r\theta \z_1$ and $\ORO_{2,1} \rhat\theta \z_1$. Both semigroups have the exact same two Schur rings, namely:
\[\{\theta, 1+2,\theta'\},\quad\{\theta, 1, 2, \theta'\}.\]

\begin{center}
\begin{tabular}{cc}
\begin{tikzpicture}
    \fill[fill=red!70, fill opacity=0.3] (1.25, 0.25) rectangle (1.75, -1.25);
    \fill[fill=green!70, fill opacity=0.3] (-0.25, -1.25) rectangle (1.25, -1.75);
    \fill[fill=mixed, fill opacity=0.3] (1.25, -1.25) rectangle (1.75, -1.75);
        \path (0, 0.5) node {$\theta$} (0.5, 0.5) node {$1$} (1, 0.5) node {$2$} (1.55, 0.5) node {$\theta'$} (-0.5, 0) node {$\theta$} (-0.5, -0.5) node {$1$} (-0.5, -1) node {$2$} (-0.45, -1.5) node {$\theta'$};
    \draw (-0.25, 0.75)--(-0.25, -1.75) (-0.75, 0.25)--(1.75, 0.25);
    \draw[dashed]  (1.25, 0.25) -- (1.25, -1.75)
         (-0.25, -1.25) -- (1.75, -1.25);
    \path (0, 0) node {$\theta$} (0.5, 0) node {$\theta$} (1, 0) node {$\theta$} (1.5, 0) node  {$\theta$} 
        (0, -0.5) node {$\theta$} (0.5, -0.5) node {$\theta$} (1, -0.5) node {$\theta$} (1.5, -0.5) node {$\theta$}
        (0, -1) node {$\theta$} (0.5, -1) node {$1$} (1, -1) node {$2$} (1.55, -1) node {$\theta^\prime$}
        (0, -1.5) node {$\theta$} (0.5, -1.5) node {$\theta$} (1.05, -1.5) node {$\theta^\prime$} (1.5, -1.5) node {$\theta$};
\end{tikzpicture}
&
\begin{tikzpicture}
    \fill[fill=red!70, fill opacity=0.3] (1.25, 0.25) rectangle (1.75, -1.25);
    \fill[fill=green!70, fill opacity=0.3] (-0.25, -1.25) rectangle (1.25, -1.75);
    \fill[fill=mixed, fill opacity=0.3] (1.25, -1.25) rectangle (1.75, -1.75);
        \path (0, 0.5) node {$\theta$} (0.5, 0.5) node {$1$} (1, 0.5) node {$2$} (1.55, 0.5) node {$\theta^\prime$} (-0.5, 0) node {$\theta$} (-0.5, -0.5) node {$1$} (-0.5, -1) node {$2$} (-0.45, -1.5) node {$\theta^\prime$};
    \draw (-0.25, 0.75)--(-0.25, -1.75) (-0.75, 0.25)--(1.75, 0.25);
    \draw[dashed]  (1.25, 0.25) -- (1.25, -1.75)
         (-0.25, -1.25) -- (1.75, -1.25);
    \path (0, 0) node {$\theta$} (0.5, 0) node {$\theta$} (1, 0) node {$\theta$} (1.5, 0) node  {$\theta$} 
        (0, -0.5) node {$\theta$} (0.5, -0.5) node {$\theta$} (1, -0.5) node {$\theta$} (1.5, -0.5) node {$\theta$}
        (0, -1) node {$\theta$} (0.5, -1) node {$1$} (1, -1) node {$2$} (1.55, -1) node {$\theta^\prime$}
        (0, -1.5) node {$\theta$} (0.5, -1.5) node {$\theta$} (1.05, -1.5) node {$\theta^\prime$} (1.55, -1.5) node {$\theta'$};
\end{tikzpicture}\\
$\ORO_{2,1} \r\theta \z_1$ & $\ORO_{2,1} \rhat\theta \z_1$
\end{tabular}
\end{center}
\end{Exam}

\begin{Exam}
\begin{multicols}{2}
Consider the semigroup $G=K_{1,2} = \{\theta, 1, 2\}$. Choosing stabilizers $H=\langle 1\rangle$ and $K=\langle 2\rangle$, observe that both $G-H=\{\theta,2\}$ and $G-K=\{\theta,1\}$ are prime ideals of $G$. This gives the rostered semigroup $K_{1,2} \r\theta (\langle 1\rangle, \langle 2\rangle)$, where $1\cdot \theta' =\theta'\cdot 2 = \theta'$, but all other products involving $\theta'$ are $\theta$. Since $\theta\not\in \{1, 2\}$, we satisfy condition (d) of \eqref{eqn:*}. As such, $\left(\theta'\right)^2 = \theta^2=\theta$. This semigroup has two Schur rings, just as $K_{1,2}$: $\{\theta, 1+2, \theta^\prime\}$, $\{\theta, 1,2, \theta'\}$.

\begin{center}
\begin{tikzpicture}
    \fill[fill=red!70, fill opacity=0.3] (1.25, 0.25) rectangle (1.75, -1.25);
    \fill[fill=green!70, fill opacity=0.3] (-0.25, -1.25) rectangle (1.25, -1.75);
    \fill[fill=mixed, fill opacity=0.3] (1.25, -1.25) rectangle (1.75, -1.75); 
        \path (0, 0.5) node {$\theta$} (0.5, 0.5) node {$1$} (1, 0.5) node {$2$} (1.55, 0.5) node {$\theta^\prime$} (-0.5, 0) node {$\theta$} (-0.5, -0.5) node {$1$} (-0.5, -1) node {$2$} (-0.45, -1.5) node {$\theta^\prime$};
    \draw (-0.25, 0.75)--(-0.25, -1.75) (-0.75, 0.25)--(1.75, 0.25);
    \draw[dashed]  (1.25, 0.25) -- (1.25, -1.75)
         (-0.25, -1.25) -- (1.75, -1.25);
    \path (0, 0) node {$\theta$} (0.5, 0) node {$\theta$} (1, 0) node {$\theta$} (1.5, 0) node  {$\theta$} 
        (0, -0.5) node {$\theta$} (0.5, -0.5) node {$1$} (1, -0.5) node {$\theta$} (1.55, -0.5) node {$\theta'$}
        (0, -1) node {$\theta$} (0.5, -1) node {$\theta$} (1, -1) node {$2$} (1.5, -1) node {$\theta$}
        (0, -1.5) node {$\theta$} (0.5, -1.5) node {$\theta$} (1.05, -1.5) node {$\theta'$} (1.5, -1.5) node {$\theta$};
\end{tikzpicture}
\\$K_{1,2}\r\theta(\langle 1\rangle, \langle 2\rangle)$
\end{center}
\end{multicols}
\end{Exam}

The previous two examples demonstrate that cloning $\theta$ via a roster can often lead to $\theta'$ being isolated in each Schur ring over $G\r\theta (H,K)$. Of course, this is not always the case as the clone in $\O_n[\theta]\cong \O_{n+1}$ can be fused with any other nonzero element. On the other hand, $\theta'$ is isolated in every Schur ring over $G[\hat{\theta}]\cong G^\theta$. Similar to $G^\theta$, if $G$ is a zero semigroup, we define \label{dualclone}$G^\mu = G\rhat\theta G$ which extends $G$ by adjoining a \emph{mutant}\footnote{As this new element is a \emph{mutant} zero and the Greek word for zero is $\mu\eta\delta\acute\varepsilon\nu$, $\mu$ seems like a fitting label.} of $\theta$. Specifically, $\mu x = x\mu = \mu$ for all $x\in G$ including $\theta$, but $\mu^2=\theta$. In $G^\mu$, this clone of $\theta$ is typically isolated in each Schur ring.

\begin{Thm}
    Let $G$ be a zero semigroup. Then $\mu\in \S$ for all Schur rings $\S$ over $G^\mu$ or $G$ has no proper divisors of zero. In the latter case, $G^\mu \cong \z_2\s H$ for some semigroup $H$.
\end{Thm}
\begin{proof}
    Let $\S$ be a Schur ring over $G^\mu$ and let $M$ be the primitive set containing $\mu$. Let $M = \mu + M'$. Then $M^2 = (2|M'|)\mu+\theta+(M')^2$. Hence, $2|M'|)M\subseteq M^2$, but this accounts for $(2|M|-2)|M| = 2|M|^2-2|M|$. As $M^2$ contains $|M|^2$ elements, counting multiplicities, this implies that $|M|\le 2$. If $M = x+\mu$, then $M^2 = 2\mu+\theta+x^2$. As the coefficient of $\mu$ is 2, it must be that $x=\theta$, that is, $M=\mu+\theta$. 
    Suppose that $x,y\in G-\theta$ such that $xy=\theta$. Let $X,Y\in\S$ be the primitive sets containing $x,y$. Then $\theta\in XY$ but $\mu\notin XY$. Therefore, $\mu\in \S$.
\end{proof}

Rosters are ubiquitous among semigroups, and many of the semigroup constructions and extensions discussed in this paper can be realized as rosters. For example, $G[x] = G\r x \emptyset$, $G[\hat{x}] = G\rhat x \emptyset$, and $G\r\theta (\emptyset,G)= G\t \z_1$. The following identities hold for $n\ge 2$, illustrating how many of the semigroup families discussed previously are formed by rosters starting with the trivial group $\z_1$: 
\begin{center}
\begin{tabular}{lll}
$\O_n = \O_{n-1} \r{\theta} \emptyset$ & $\CH_n = \CH_{n-1} \r{\theta} \CH_{n-1}$  &$\RO_{n}=\RO_{n-1}\r{\theta} (\RO_{n-1}, \emptyset)$\\
$K_{1,n} = K_{1, n-1} \widehat{\r{\theta}}\ \emptyset$ & $\OG_n = \OG_{n-1} \r{\theta} (G,\emptyset)$ &  
$\z_2 = \z_1 \widehat{\r{\theta}}\ \z_1$
\end{tabular}
\end{center}
With the exception of $\z_2$, the monogenic semigroup $\z_{m,n}$ is not a roster. Likewise, $\LOZ_n$ is not a roster for $n\ge 2$ (note that $\LOZ_1\cong \CH_2$). While $\OROP_2$ is not a roster, for $n\ge3$, $\OROP_n = \OROP_{n-1} \r{\theta} (\alpha, \beta)$. Lastly, each of $\O(m,n,k)$, $\LOO(\ell,m,n,k)$, and $\LORO(\ell,m,n,k)$ can sometimes be realized as rosters depending on the choice of $k$. For example, $\LORO(2,1,1,0) = (\CH_2\t\z_1)\r0 (\z_1,\CH_2\t\z_1)$, but $\LORO(2,1,2,1)$ is not equivalent to any roster.

We use rosters to define one final family of semigroups. We had previously defined $\OG_m$ for some semigroup $G$ and positive integer $m$. As observed above, this semigroup can recursively be constructed via rosters, since $\OG_1 = G^\theta$ and $\OG_m=\OG_{m-1}\r\theta (G,\emptyset)$ for $m\ge 2$. Given the direction of the roster determines the multiplication of the clone, we can reverse the roster's direction and extend $\OG_m$ in a different way. Define \label{ogtwo}$\OG_{m,1} = \OG_m$ and $\OG_{m,n} = \OG_{m,n-1}\r\theta (\emptyset, G)$ for $n\ge 2$. An important example is $G=\LO_n$, where we define $\OLO_{\ell,m,n} = \O(\LO_n)_{\ell,m}$.\label{olothree}\footnote{By the nature of this two-sided rostering, we observe that $\OG_{m,n}$ is equivalent to $\O(G^\text{op})_{n,m}$. In particular, we have that $\OLO_{\ell,m,n}$ is equivalent to \label{orothree}$\ORO_{m,\ell,n}$. When choosing notation to label a semigroup, the authors strive to have the notation generate the lexicographically minimal Cayley of the semigroup. As such, $\OLO_{\ell,m,n}$ is preferred over $\ORO_{m,\ell,n}$.} If $\OLO_{\ell,m,n} = \{\theta,1,\ldots,\ell+m+n-1\}$, let $L = \{1,\ldots,\ell-1\}$, $M=\{\ell,\ldots,\ell+m-1\}$, and $N=\{\ell+m,\ldots, \ell+m+n-1\}$, then $L+M+\theta \cong \O_{\ell+m-1}$, $N+\theta \cong \LO_n$, $L+N+\theta\cong\OLO_{\ell,n}$, and $M+N+\theta\cong \ORO_{m,n}^\text{op}$.

\begin{Thm}\label{thm:orolmn} Let $\ell,m,n$ be positive integers. Let $\S$ be a partition over $\OLO_{\ell,m,n}$. Then $\S$ is a Schur ring if and only if $\S$ is the common refinement of a Schur ring over $L+N+\theta$ and $L+\theta$, or $\theta\in \S$ and there exists some $X\in \S$ such that $N\subseteq X$ or $M+N\subseteq X$. Then \[\Omega(\OLO_{\ell,m,n}) = \B(\ell-1)\B(m+n-1) + \sum_{k=1}^{\ell-1} \dbinom{\ell-1}{k}(\B(\ell-k-1)\B(n-1)+1).\]
\end{Thm}
\begin{proof}
The proof is similar to the proof of Theorem \ref{thm:og}. Note that if $gh=a$ for $a\in L$, then $g=a$ and $h\in N$. Let $X\in \S$ such that $a,x\in X$ where $a\in L$ and $x\in M+N$. Suppose that $y\in N-X$ and $Y\in \S$ is the primitive set containing $y$. Then $yx=x\in YX$ but $a\notin YX$ since $a\notin Y$, a contradiction. Hence, if $(M+N)\cap X\neq \emptyset$, then $N\subseteq X$. If $x\in M$ and $y\in M-X$, then $yx=x\in YX$ but $a\notin YX$. Hence, if $M\cap X\neq\emptyset$, then $M\subseteq X$.      
\end{proof}

\begin{Exam} 
\begin{multicols}{2}The semigroup $\OLO_{2,2,1}$ is illustrated in the Cayley table to the right. It has four Schur rings: \[\{0,1+2+3\},\quad \{0,1,2+3\},\quad \{0,2,1+3\},\quad\text{and}\quad\{0,1,2,3\},\] as predicted by Theorem \ref{thm:orolmn}.
\begin{center}
\begin{tikzpicture}
    \fill[fill=red!70, fill opacity=0.3] (1.25, 0.25) rectangle (1.75, -1.25);
    \fill[fill=green!70, fill opacity=0.3] (-0.25, -1.25) rectangle (1.25, -1.75);
    \fill[fill=mixed, fill opacity=0.3] (1.25, -1.25) rectangle (1.75, -1.75);
        \path (0, 0.5) node {0} (0.5, 0.5) node {1} (1, 0.5) node {2} (1.5, 0.5) node {3} (-0.5, 0) node {0} (-0.5, -0.5) node {1} (-0.5, -1) node {2} (-0.5, -1.5) node {3};
    \draw (-0.25, 0.75)--(-0.25, -1.75) (-0.75, 0.25)--(1.75, 0.25);
    \draw[dashed]  (1.25, 0.25) -- (1.25, -1.75)
         (-0.25, -1.25) -- (1.75, -1.25);
    \path (0, 0) node {0} (0.5, 0) node {0} (1, 0) node {0} (1.5, 0) node  {0} 
        (0, -0.5) node {0} (0.5, -0.5) node {0} (1, -0.5) node {0} (1.5, -0.5) node {0}
        (0, -1) node {0} (0.5, -1) node {0} (1, -1) node {0} (1.5, -1) node {2}
        (0, -1.5) node {0} (0.5, -1.5) node {1} (1, -1.5) node {0} (1.5, -1.5) node {3};
\end{tikzpicture}
\\$\ORO_{2,2,1}$
\end{center}
\end{multicols}
\end{Exam}

Like cloning, rosters can be iterated. Let $x,y\in G$, let $H\le G_x\cap G_y$ such that $G-H$ is a prime ideal and $x,y\notin H$ and let $K\le\, _xG\cap\, _yG$ such that $G-K$ is a prime ideal.\footnote{When rostering a single element $x'$ to $G$, we are allowed to have $x\in H$, but when considering a sequential rostering, such as $G\r{x,y} (H,K)$ we do not allow for the clones $x'$ and $y'$ to stabilize $x$ and $y$ to stabilize themselves, otherwise we might break associativity with products involving $x'y'$. In particular, we require (d) in this case.} Then $G\r{x} (H,K) - H$ and $G\r{x} (H,K)-K$ are prime ideals in $G\r{x} (H,K)$. Hence, $(G\r x (H,K))\r y(H,K) = (G\r y (H,K))\r x(H,K)$ is a roster semigroup. We define \label{rosterdual}$G\r{x,y} (H,K) = (G\r x (H,K))\r y(H,K)$.  Let $X\subseteq G$, let $H\le \bigcap_{x\in X} G_x$ such that $G-H$ is a prime ideal and $X\subseteq G-H$ and let $K\le \bigcap_{x\in X}\,_xG$ such that $G-K$ is a prime ideal and $X\subseteq G-K$. By induction, we define \label{rinduction}$G\r X (H,K) = G\r{X-x,x} (H,K)$ for any $x\in X$. Note that we allow $X$ here to be a multiset, that is, elements of $X$ may repeat. For example, $\OG_n = G^\theta\r\Theta (G,\emptyset)$ where $\Theta$ is $(n-1)$-copies of $\theta$.

\begin{Thm}\label{thm:rosterrings}
Let $G$ be a semigroup, let $X\subseteq G$, let $H\le \bigcap_{x\in X} G_x$ such that $G-H$ is a prime ideal and $X\subseteq G-H$, and let $K\le \bigcap_{x\in X}\, _xG$ such that $G-K$ is a prime ideal  and $X\subseteq G-K$. Let $\S$ be a Schur ring over $G$ such that $X$ is an $\S$-subset and $H,K$ are $\S$-subsemigroups. If $X^\prime$ is the set of clones of elements of $X$ in $G\r X (H,K)$, then $\S \r X (H,K) =\S\vee \{X^\prime\}$ is a Schur ring over $G\r X (H,K)$.
\end{Thm}
\begin{proof}
Note that $\left(X'\right)^2 = X'X = XX' = X^2 \in \QS$. If $Y\in \S|_H$, then $YX' = |Y|X'\in\Q[S\r X(H,K)]$. If $Y\in \S|_{G-H}$, the $YX'=YX\in \QS \le \Q[\S \r X (H,K)]$. Similarly, $X'Y \in \Q[S\r X(H,K)]$.
\end{proof}

\section{Schur Rings over Semigroups of Small Order}\label{gap}
Algebraists have worked to enumerate all semigroups, up to equivalence, for small orders. This sequence of integers grows rapidly. At the time of this writing, all semigroups of order up to 10 have been enumerated \cite[A001423]{semicount}. While orders 0, 1, 2, and 3 are simple enough to enumerate by hand, semigroups of order $4$, for which there are 126 equivalence types, proved more challenging. Notably, in 1955 Forsythe \cite{Forsythe} used the early computer system SWAC to enumerate all semigroups of order up to 4, correcting a previous attempt. Presently, all semigroups of order up to 8 can be found within GAP \cite{GAP}, specifically the package \emph{smallsemi}. The indexing of the semigroups in \emph{smallsemi} disagrees with the indexing found in \cite{Forsythe}, the latter of which lists semigroups by their Cayley table in lexicographical order, where each semigroup is represented by its lexicographically-minimum Cayley table. As such there is disagreement in how to number them, we list both indexing methods for the convenience of the reader in Table \ref{apx:A}. 

Using \emph{smallsemi}, we enumerate all Schur rings over each semigroup up to order $7$. This was accomplished simply by an exhaustive search over all possible partitions for each of the semigroups. We will often refer to the index of semigroups in the \emph{smallsemi} package as a semigroups GAP identification number.

\subsection{Labeling of Semigroups of Small Order}
Throughout this paper, we have introduced many families of semigroups, including $\CH$, $\K$, $\LO$, $\LOO$, $\LORO$, $\LOZ$, $\O$, $\OLO$, $\ORO$, $\OROP$, $\RO$, $\z$, and constructions of semigroups, including $G[x]$, $G[\hat{x}]$, $\mathfrak{r}$, $\mathfrak{s}$, $\mathfrak{t}$, $\mathfrak{u}$, $G^e$, $G^\mu$, $G^\theta$, and $\times.$ We have then discussed classification and enumeration of Schur rings over these families and constructions of semigroups. 

These families and constructions are sufficient to classify all semigroups up to order $4$ (see Tables \ref{fig:sgsr2}, \ref{fig:sgsr3}, and \ref{fig:sgsr4}). In many instances, multiple classifications are available for an equivalence class of semigroups. In those instances, the authors followed two principles. First, the authors elected to use classifiers which generated the minimal Cayley table of a semigroup, e.g. preferring $\LO_n^\theta$ over $\RO_n^\theta$ but defining $\ORO_{m,n}$ via $\RO_n$ as left-identities instead of $\LO_n$ as right-identities. Second, the authors elected to use the perceived simplest classifier. When common patterns emerged in classifying small semigroups, families with parameters were introduced, e.g. $\LORO$, $\ORO$, or $\OROP$. In these examples, ideal extensions were present and codes were chosen to denote this ideal and Rees quotient. These ideal extensions were often restricted to specific semigroup families, e.g. the ideal is restricted to $\O_n$ or $\LO_n$. When no or few restrictions were required, we instead defined operations on semigroups, e.g. $\mathfrak{s}$, $\mathfrak{t}$, or $\mathfrak{u}$. Despite its wide versatility, the authors elected to use rosters only when no other classification in this scheme was available. But even among this classification scheme of small semigroups, some patterns emerged with such high frequency that abbreviations were adopted, such as $G^\theta$, $G^\mu$, and $G[x]$. The goal in this scheme is to encode as compactly as possible the original multiplication and Cayley table with few parameters. For the convenience of the reader, the complete almanac of this classification scheme is provided in the Appendix. 
This scheme certainly does not classify all semigroups of small order but describes a majority of the semigroups within the \emph{smallsemi} database.

It is also noteworthy that this classification scheme emerged from the study of Schur rings over these semigroups of small order. As the Schur rings were classified, enumerated, and organized, this organized their respective semigroups into families, thus illustrating the potential of Schur rings in an algebraic category other than groups. For example, we will see that many semigroups only have 1 or 2 Schur rings. Also, the number of Schur rings generally cluster toward Bell numbers as well as toward their sums or products. It is rare for a semigroup to have a very large number of Schur rings. We say that a semigroup with a relatively high number of Schur rings compared to its order is \emph{Schur-dense}, meaning more than $\B(n-2)$ Schur rings.

\subsection{Semigroups of Orders $0$, $1$, and $2$}
For orders $0$ and $1$, there is a unique semigroup, namely $\emptyset$ and $\z_1$, respectively. In each case, there is a unique partition on these semigroups, which is a Schur ring in both cases, namely the discrete Schur ring. 

For semigroups of order $2$, there are $4$ equivalence types. There are only two possible partitions on a set of cardinality $2$, namely the discrete and indiscrete partitions. As mentioned above, the discrete partition always affords a Schur ring. Hence, the question of enumeration is equivalent to determining whether the indiscrete partition affords a Schur ring over specific order $2$ semigroups or not. In Table \ref{fig:sgsr2}, we list the number of distinct Schur rings over each semigroup of order $2$, using Forsythe's lexicographical ordering of Cayley tables. The equivalence type of each semigroup is also provided, using notation introduced earlier, for convenience and clarity. The GAP index of each semigroup is also included for convenience, but was not used in the calculation of these. Thus our labeling of the elements may disagree with that of GAP.

\begin{table}[ht]
\begin{center}
\begin{tabular}{ccc|c}
\multicolumn{2}{c}{G} & GAP ID  & $\Omega$\\
\hline
00 & $\O_2$ & 1 & 1\\
01 & $\CH_2$ & 3 & 1 \\
02 & $\LO_2$ & 4 & 2 \\
03 & $\z_2$ & 2 & 2
\end{tabular}
\end{center}
\caption{Schur rings of Semigroups of Order $2$}
\label{fig:sgsr2}
\end{table}

We observe that exactly half the semigroups only have a single Schur ring, namely $\text{O}_2$ and $\CH_2$, which agrees with Theorems \ref{thm:nullschur} and \ref{thm:chainschur}. Whereas the other two semigroups, namely $\LO_2$ and $\z_2$, have the maximum number of Schur rings, which agrees with Theorems \ref{thm:leftnullschur} and \ref{thm:monogenic}. 

For semigroups of small order, the Schur rings can be found using GAP. The package \textit{smallsemi} contains the semigroups of orders $1$ through $8$. Using this package, we provide the number of Schur rings each semigroup has, for orders $0$ -- $4$. For these orders, we also use the Cayley tables given in \cite{Forsythe} to give the specific Schur rings. For order $5$ -- $7$ we list how many semigroups have a fixed number of Schur rings, and identify some.  

This was done by writing one function which took two lists of elements in a given semigroup and multiplied them as if it were a semigroup ring. Another function was defined which checked if a given semigroup ring element was a linear combination of primitive sets from a partition. We then performed an exhaustive search over all possible products of all partitions for each the semigroups, and recorded how many semigroups attain each number of Schur rings. It is for this reason that order $8$ proved computationally prohibitive at this time.

\subsection{Semigroups of Order $3$}
As there are five partitions on a set of cardinality 3, we will need to be more concise in our reporting of Schur rings over semigroups of order $3$ than we have been so far. In Figure \ref{fig:3-partition} we illustrate the five possible partitions, where elements of the semigroup are denoted $0,1,2$. Note that partition 1 illustrates the discrete partition, while partition 5 illustrates the indiscrete partition. 

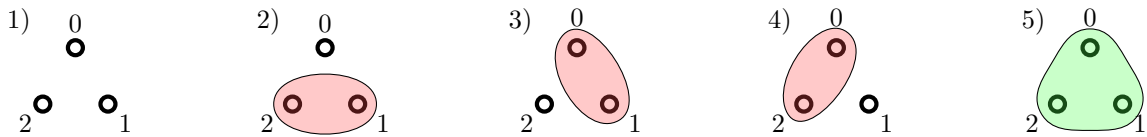
\begin{figure}[ht]
\begin{multicols}{5}
\begin{center}
\begin{tikzpicture}
\path (-0.75, 0.85) node {1)};
\path[ultra thick] (90:0.5) \vertex (0)  node[above, yshift=2] {$0$};
\path[ultra thick] (-30:0.5) \vertex (1)  node[below right] {$1$};
\path[ultra thick] (210:0.5) \vertex (2)  node[below left] {$2$};
\end{tikzpicture}
\end{center}

\begin{center}
\begin{tikzpicture}
\path (-0.75, 0.85) node {2)};
\path[ultra thick] (90:0.5) \vertex (0)  node[above, yshift=2] {$0$};
\path[ultra thick] (-30:0.5) \vertex (1)  node[below right, xshift=3] {$1$};
\path[ultra thick] (210:0.5) \vertex (2)  node[below left, xshift=-3] {$2$};
\filldraw[fill=red!70, fill opacity=0.3] ($(2)+(180:0.25)$)
    to[out=90,in=90] ($(1)+(0:0.25)$)
    to[out=270,in=270] ($(2)+(180:0.25)$);
\end{tikzpicture}
\end{center}

\begin{center}
\begin{tikzpicture}
\path (-0.75, 0.85) node {3)};
\path[ultra thick] (90:0.5) \vertex (0)  node[above, yshift=4] {$0$};
\path[ultra thick] (-30:0.5) \vertex (1)  node[below right, xshift=3] {$1$};
\path[ultra thick] (210:0.5) \vertex (2)  node[below left] {$2$};
\filldraw[fill=red!70, fill opacity=0.3] ($(0)+(135:0.25)$)
    to[out=45,in=30] ($(1)+(-60:0.25)$)
    to[out=210,in=225] ($(0)+(135:0.25)$);
\end{tikzpicture}
\end{center}

\begin{center}
\begin{tikzpicture}
\path (-0.75, 0.85) node {4)};
\path[ultra thick] (90:0.5) \vertex (0)  node[above, yshift=4] {$0$};
\path[ultra thick] (-30:0.5) \vertex (1)  node[below right] {$1$};
\path[ultra thick] (210:0.5) \vertex (2)  node[below left, xshift=-3] {$2$};
\filldraw[fill=red!70, fill opacity=0.3] ($(0)+(45:0.25)$)
    to[out=135,in=120] ($(2)+(210:0.25)$)
    to[out=300,in=315] ($(0)+(45:0.25)$);
\end{tikzpicture}
\end{center}

\begin{center}
\begin{tikzpicture}
\path (-0.75, 0.85) node {5)};
\path[ultra thick] (90:0.5) \vertex (0)  node[above, yshift=4] {$0$};
\path[ultra thick] (-30:0.5) \vertex (1)  node[below right] {$1$};
\path[ultra thick] (210:0.5) \vertex (2)  node[below left, xshift=-3] {$2$};
\filldraw[fill=green!70, fill opacity=0.3] ($(0)+(90:0.25)$)
    to[out=0,in=120] ($0.5*(0)+0.5*(1)+(45:0.35)$)
    to[out=300,in=60] ($(1)+(-30:0.25)$)
    to[out=240,in=0] ($0.5*(2)+0.5*(1)+(270:0.35)$)   
    to[out=180,in=-60] ($(2)+(210:0.25)$)
    to[out=120,in=240] ($0.5*(0)+0.5*(2)+(135:0.35)$)
    to[out=60, in=180] ($(0)+(90:0.25)$);
\end{tikzpicture}
\end{center}
\end{multicols}
\caption{The Five Partitions of Order 3}
\label{fig:3-partition}
\end{figure}

In Table \ref{fig:sgsr3}, we list the number of distinct Schur rings over each semigroup of order $3$, using Forsythe's lexicographical ordering of Cayley tables. The equivalence type of each semigroup is also provided, using the classification notion introduced earlier, for convenience and clarity. Additionally, a list of Schur rings for each semigroup is provided, using the indexing of Figure \ref{fig:3-partition}. The GAP index of each semigroup is again included for convenience. 

\begin{table}[ht]
\begin{center}
\begin{tabular}{ccc|c|c}
\multicolumn{2}{c}{$G$} & GAP ID & Schur Ring Indices & $\Omega(G)$\\
\hline
00 & $\O_3$ & 1 & 1, 2 & 2\\
01 & $\z_{3,1}$ & 4 & 1 & 1\\
02 & $\O_2 \u \CH_2$ & 5 & 1 & 1\\
03 & $\ORO_{2,1}$ & 6 & 1, 2 & 2\\
04 & $\O_2 \t \z_1$ & 7 & 1 & 1\\
05 & $\O_2^e$ & 8 & 1 & 1\\
06 & $\K_{1,2}$ & 12 & 1, 2 & 2\\
07 & $\CH_2\t\z_1$ & 13 & 1 & 1\\
08 & $\O_2^\theta$ & 9 & 1 & 1\\
09 & $\CH_3$ & 14 & 1 & 1\\
10 & $\LO_2^\theta$ & 15 & 1, 2 & 2\\ 
11 & $\z_2^\theta$ & 11 & 1, 2 & 2\\
12 & $\LO_2^e$ & 16 & 1, 4 & 2\\
13 & $\LO_3$ & 17 & 1,2,3,4,5 & 5\\
14 & $\z_2[e]$& 2 & 1 & 1\\
15 & $\z_2^e$ & 10 & 1, 4 & 2\\
16 & $\z_{2,2}$ & 3 & 1 & 1\\
17 & $\z_3$ & 18 & 1, 2, 5 & 3\\
\end{tabular}
\end{center}
\caption{Schur rings of Semigroups of Order $3$}
\label{fig:sgsr3}
\end{table}

An inspection of Table \ref{fig:sgsr3} shows that the mean number of Schur rings for semigroups of order $3$ is $\mu = 1.72$, and the standard deviation is $\sigma=1.02$. Additionally, 88\% of the semigroups have only 1 or 2 Schur rings, where the second one is never the indiscrete partition. Thus, having three or more Schur rings is a rare feature for a semigroup of order $3$. The two exceptions are the left-null semigroup $\LO_3$ and the only group of order $3$, $\z_3$. For $\LO_3$, the maximum number of Schur rings is obtained, as observed by Theorem \ref{thm:leftnullschur}. The three Schur rings over the cyclic group $\z_3$ are the discrete Schur ring, the indiscrete Schur ring, and the trivial Schur ring. 

\subsection{Semigroups of Order $4$}
In similar manner for semigroups of order $4$, there are fifteen partitions on a set of cardinality $4$, as illustrated in Figure \ref{fig:4-partition}. In Table \ref{fig:sgsr4}, we list the number of distinct Schur rings over each semigroup of order $4$, using Forsythe's order \cite{Forsythe}. We also provide a correspondence between Forsythe's indexing of semigroups and GAP's, given in Table \ref{apx:A} at the end of this paper, which appears to not be available elsewhere in the literature. 

Since there are $15$ ways to partition the elements $0, 1, 2, 3$ to potentially create Schur rings, we provide and label these. These are the labels which are being refered to in Table \ref{fig:sgsr4} and the elements $0-3$ are understood to refer to the elements of the semigroup as identified in \cite{Forsythe}.

\begin{figure}[!ht]
\begin{multicols}{5}
\begin{center}
\begin{tikzpicture}
\path (-0.75, 1.25) node {1)};
\path[ultra thick] (0,1) \vertex (0)  node[above left, yshift=2] {$0$};
\path[ultra thick] (1,1) \vertex (1)  node[above right] {$1$};
\path[ultra thick] (1,0) \vertex (2)  node[below right] {$2$};
\path[ultra thick] (0,0) \vertex (3)  node[below left] {$3$};
\end{tikzpicture}
\end{center}

\begin{center}
\begin{tikzpicture}
\path (-0.75, 1.25) node {2)};
\path[ultra thick] (0,1) \vertex (0)  node[above left, yshift=2] {$0$};
\path[ultra thick] (1,1) \vertex (1)  node[above right] {$1$};
\path[ultra thick] (1,0) \vertex (2)  node[below right] {$2$};
\path[ultra thick] (0,0) \vertex (3)  node[below left] {$3$};
\filldraw[fill=red!70, fill opacity=0.3] ($(0)+(180:0.25)$)
    to[out=90,in=180] ($0.5*(0)+0.5*(1)+(0,0.3)$)
    to[out=0,in=90] ($(1)+(0:0.25)$) 
    to[out=270,in=0] ($0.5*(0)+0.5*(1)-(0,0.3)$)
    to[out=180,in=270] ($(0)+(180:0.25)$);
\end{tikzpicture}
\end{center}

\begin{center}
\begin{tikzpicture}
\path (-0.75, 1.25) node {3)};
\path[ultra thick] (0,1) \vertex (0)  node[above left, yshift=2] {$0$};
\path[ultra thick] (1,1) \vertex (1)  node[above right] {$1$};
\path[ultra thick] (1,0) \vertex (2)  node[below right] {$2$};
\path[ultra thick] (0,0) \vertex (3)  node[below left] {$3$};
\filldraw[fill=red!70, fill opacity=0.3] ($(0)+(135:0.25)$)
    to[out=45,in=135] ($0.5*(0)+0.5*(2)+(45:0.35)$)
    to[out=315,in=45] ($(2)+(315:0.25)$) 
    to[out=225,in=315] ($0.5*(0)+0.5*(2)-(45:0.35)$)
    to[out=135,in=225] ($(0)+(135:0.25)$);
\end{tikzpicture}
\end{center}

\begin{center}
\begin{tikzpicture}
\path (-0.75, 1.25) node {4)};
\path[ultra thick] (0,1) \vertex (0)  node[above left, yshift=2] {$0$};
\path[ultra thick] (1,1) \vertex (1)  node[above right] {$1$};
\path[ultra thick] (1,0) \vertex (2)  node[below right] {$2$};
\path[ultra thick] (0,0) \vertex (3)  node[below left] {$3$};
\filldraw[fill=red!70, fill opacity=0.3] ($(0)+(90:0.25)$)
    to[out=0,in=90] ($0.5*(0)+0.5*(3)+(0.3,0)$)
    to[out=270,in=0] ($(3)+(270:0.25)$) 
    to[out=180,in=270] ($0.5*(0)+0.5*(3)-(0.3,0)$)
    to[out=90,in=180] ($(0)+(90:0.25)$);
\end{tikzpicture}
\end{center}

\begin{center}
\begin{tikzpicture}
\path (-0.75, 1.25) node {5)};
\path[ultra thick] (0,1) \vertex (0)  node[above left, yshift=2] {$0$};
\path[ultra thick] (1,1) \vertex (1)  node[above right] {$1$};
\path[ultra thick] (1,0) \vertex (2)  node[below right] {$2$};
\path[ultra thick] (0,0) \vertex (3)  node[below left] {$3$};
\filldraw[fill=red!70, fill opacity=0.3] ($(0)+(180:0.25)$)
    to[out=90,in=180] ($0.5*(0)+0.5*(1)+(0,0.3)$)
    to[out=0,in=90] ($(1)+(0:0.25)$) 
    to[out=270,in=0] ($0.5*(0)+0.5*(1)-(0,0.3)$)
    to[out=180,in=270] ($(0)+(180:0.25)$);
\filldraw[fill=red!70, fill opacity=0.3] ($(3)+(180:0.25)$)
    to[out=90,in=180] ($0.5*(3)+0.5*(2)+(0,0.3)$)
    to[out=0,in=90] ($(2)+(0:0.25)$) 
    to[out=270,in=0] ($0.5*(3)+0.5*(2)-(0,0.3)$)
    to[out=180,in=270] ($(3)+(180:0.25)$);
\end{tikzpicture}
\end{center}
\end{multicols}
\begin{multicols}{5}
\begin{center}
\begin{tikzpicture}
\path (-0.75, 1.25) node {6)};
\path[ultra thick] (0,1) \vertex (0)  node[above left, yshift=2] {$0$};
\path[ultra thick] (1,1) \vertex (1)  node[above right] {$1$};
\path[ultra thick] (1,0) \vertex (2)  node[below right] {$2$};
\path[ultra thick] (0,0) \vertex (3)  node[below left] {$3$};
\filldraw[fill=red!70, fill opacity=0.3] ($(0)+(135:0.25)$)
    to[out=45,in=135] ($0.5*(0)+0.5*(2)+(45:0.35)$)
    to[out=315,in=45] ($(2)+(315:0.25)$) 
    to[out=225,in=315] ($0.5*(0)+0.5*(2)-(45:0.35)$)
    to[out=135,in=225] ($(0)+(135:0.25)$);
\filldraw[fill=red!70, fill opacity=0.3] ($(1)+(45:0.25)$)
    to[out=315,in=45] ($0.5*(1)+0.5*(3)+(-45:0.35)$)
    to[out=225,in=315] ($(3)+(225:0.25)$) 
    to[out=135,in=225] ($0.5*(1)+0.5*(3)-(-45:0.35)$)
    to[out=45,in=135] ($(1)+(45:0.25)$);
\end{tikzpicture}
\end{center}

\begin{center}
\begin{tikzpicture}
\path (-0.75, 1.25) node {7)};
\path[ultra thick] (0,1) \vertex (0)  node[above left, yshift=2] {$0$};
\path[ultra thick] (1,1) \vertex (1)  node[above right] {$1$};
\path[ultra thick] (1,0) \vertex (2)  node[below right] {$2$};
\path[ultra thick] (0,0) \vertex (3)  node[below left] {$3$};
\filldraw[fill=red!70, fill opacity=0.3] ($(0)+(90:0.25)$)
    to[out=0,in=90] ($0.5*(0)+0.5*(3)+(0.3,0)$)
    to[out=270,in=0] ($(3)+(270:0.25)$) 
    to[out=180,in=270] ($0.5*(0)+0.5*(3)-(0.3,0)$)
    to[out=90,in=180] ($(0)+(90:0.25)$);
\filldraw[fill=red!70, fill opacity=0.3] ($(1)+(90:0.25)$)
    to[out=0,in=90] ($0.5*(1)+0.5*(2)+(0.3,0)$)
    to[out=270,in=0] ($(2)+(270:0.25)$) 
    to[out=180,in=270] ($0.5*(1)+0.5*(2)-(0.3,0)$)
    to[out=90,in=180] ($(1)+(90:0.25)$);    
\end{tikzpicture}
\end{center}

\begin{center}
\begin{tikzpicture}
\path (-0.75, 1.25) node {8)};
\path[ultra thick] (0,1) \vertex (0)  node[above left, yshift=2] {$0$};
\path[ultra thick] (1,1) \vertex (1)  node[above right] {$1$};
\path[ultra thick] (1,0) \vertex (2)  node[below right] {$2$};
\path[ultra thick] (0,0) \vertex (3)  node[below left] {$3$};
\filldraw[fill=red!70, fill opacity=0.3] ($(1)+(90:0.25)$)
    to[out=0,in=90] ($0.5*(1)+0.5*(2)+(0.3,0)$)
    to[out=270,in=0] ($(2)+(270:0.25)$) 
    to[out=180,in=270] ($0.5*(1)+0.5*(2)-(0.3,0)$)
    to[out=90,in=180] ($(1)+(90:0.25)$);    
\end{tikzpicture}
\end{center}

\begin{center}
\begin{tikzpicture}
\path (-0.75, 1.25) node {9)};
\path[ultra thick] (0,1) \vertex (0)  node[above left, yshift=2] {$0$};
\path[ultra thick] (1,1) \vertex (1)  node[above right] {$1$};
\path[ultra thick] (1,0) \vertex (2)  node[below right] {$2$};
\path[ultra thick] (0,0) \vertex (3)  node[below left] {$3$};
\filldraw[fill=red!70, fill opacity=0.3] ($(1)+(45:0.25)$)
    to[out=315,in=45] ($0.5*(1)+0.5*(3)+(-45:0.35)$)
    to[out=225,in=315] ($(3)+(225:0.25)$) 
    to[out=135,in=225] ($0.5*(1)+0.5*(3)-(-45:0.35)$)
    to[out=45,in=135] ($(1)+(45:0.25)$);
\end{tikzpicture}
\end{center}

\begin{center}
\begin{tikzpicture}
\path (-0.75, 1.25) node {10)};
\path[ultra thick] (0,1) \vertex (0)  node[above left, yshift=2] {$0$};
\path[ultra thick] (1,1) \vertex (1)  node[above right] {$1$};
\path[ultra thick] (1,0) \vertex (2)  node[below right] {$2$};
\path[ultra thick] (0,0) \vertex (3)  node[below left] {$3$};
\filldraw[fill=red!70, fill opacity=0.3] ($(3)+(180:0.25)$)
    to[out=90,in=180] ($0.5*(3)+0.5*(2)+(0,0.3)$)
    to[out=0,in=90] ($(2)+(0:0.25)$) 
    to[out=270,in=0] ($0.5*(3)+0.5*(2)-(0,0.3)$)
    to[out=180,in=270] ($(3)+(180:0.25)$);
\end{tikzpicture}
\end{center}
\end{multicols}
\begin{multicols}{5}
\begin{center}
\begin{tikzpicture}
\path (-0.75, 1.25) node {11)};
\path[ultra thick] (0,1) \vertex (0)  node[above left, yshift=2] {$0$};
\path[ultra thick] (1,1) \vertex (1)  node[above right] {$1$};
\path[ultra thick] (1,0) \vertex (2)  node[below right] {$2$};
\path[ultra thick] (0,0) \vertex (3)  node[below left] {$3$};
\filldraw[fill=green!70, fill opacity=0.3] ($(1)+(45:0.25)$)
    to[out=315,in=90] ($0.5*(1)+0.5*(2)+(0:0.35)$)
    to[out=270,in=45] ($(2)+(315:0.25)$)
    to[out=225,in=315] ($(2)+(225:0.25)$)
    to[out=135,in=315] ($(0)+(225:0.25)$)

    to[out=135,in=225] ($(0)+(135:0.25)$)
    to[out=45,in=180] ($0.5*(0)+0.5*(1)+(90:0.35)$)
    to[out=0,in=135] ($(1)+(45:0.25)$);
\end{tikzpicture}
\end{center}

\begin{center}
\begin{tikzpicture}
\path (-0.75, 1.25) node {12)};
\path[ultra thick] (0,1) \vertex (0)  node[above left, yshift=2] {$0$};
\path[ultra thick] (1,1) \vertex (1)  node[above right] {$1$};
\path[ultra thick] (1,0) \vertex (2)  node[below right] {$2$};
\path[ultra thick] (0,0) \vertex (3)  node[below left] {$3$};
\filldraw[fill=green!70, fill opacity=0.3] ($(0)+(135:0.25)$)
    to[out=45,in=180] ($0.5*(0)+0.5*(1)+(90:0.35)$)
    to[out=0,in=135] ($(1)+(45:0.25)$) 
    to[out=315,in=45] ($(1)+(315:0.25)$)
    to[out=225,in=45] ($(3)+(315:0.25)$)
    to[out=225,in=315] ($(3)+(225:0.25)$)    
    to[out=135,in=270] ($0.5*(3)+0.5*(0)+(180:0.35)$)
    to[out=90,in=225] ($(0)+(135:0.25)$);
\end{tikzpicture}
\end{center}

\begin{center}
\begin{tikzpicture}
\path (-0.75, 1.25) node {13)};
\path[ultra thick] (0,1) \vertex (0)  node[above left, yshift=2] {$0$};
\path[ultra thick] (1,1) \vertex (1)  node[above right] {$1$};
\path[ultra thick] (1,0) \vertex (2)  node[below right] {$2$};
\path[ultra thick] (0,0) \vertex (3)  node[below left] {$3$};
\filldraw[fill=green!70, fill opacity=0.3] ($(3)+(225:0.25)$)
    to[out=315,in=180] ($0.5*(3)+0.5*(2)+(270:0.35)$)
    to[out=0,in=225] ($(2)+(315:0.25)$)
    to[out=45,in=315] ($(2)+(45:0.25)$)
    to[out=135,in=315] ($(0)+(45:0.25)$)
    to[out=135,in=45] ($(0)+(135:0.25)$)
    to[out=225,in=90] ($0.5*(0)+0.5*(3)+(180:0.35)$)
    to[out=270,in=135] ($(3)+(225:0.25)$);   
\end{tikzpicture}
\end{center}

\begin{center}
\begin{tikzpicture}
\path (-0.75, 1.25) node {14)};
\path[ultra thick] (0,1) \vertex (0)  node[above left, yshift=2] {$0$};
\path[ultra thick] (1,1) \vertex (1)  node[above right] {$1$};
\path[ultra thick] (1,0) \vertex (2)  node[below right] {$2$};
\path[ultra thick] (0,0) \vertex (3)  node[below left] {$3$};
\filldraw[fill=green!70, fill opacity=0.3] ($(2)+(315:0.25)$)
    to[out=45,in=270] ($0.5*(2)+0.5*(1)+(0:0.35)$)
    to[out=90,in=315] ($(1)+(45:0.25)$) 
    to[out=135,in=45] ($(1)+(135:0.25)$)
    to[out=225,in=45] ($(3)+(135:0.25)$)
    to[out=225,in=135] ($(3)+(225:0.25)$)    
    
    to[out=315,in=180] ($0.5*(3)+0.5*(2)+(270:0.35)$)
    to[out=0,in=225] ($(2)+(315:0.25)$);
\end{tikzpicture}
\end{center}

\begin{center}
\begin{tikzpicture}
\path (-0.75, 1.25) node {15)};
\path[ultra thick] (0,1) \vertex (0)  node[above left, yshift=2] {$0$};
\path[ultra thick] (1,1) \vertex (1)  node[above right] {$1$};
\path[ultra thick] (1,0) \vertex (2)  node[below right] {$2$};
\path[ultra thick] (0,0) \vertex (3)  node[below left] {$3$};
\filldraw[fill=blue!70, fill opacity=0.3] ($(0)+(135:0.25)$)
    to[out=45,in=180] ($0.5*(0)+0.5*(1)+(90:0.35)$)
    to[out=0,in=135] ($(1)+(45:0.25)$) 
    to[out=315,in=90] ($0.5*(2)+0.5*(1)+(0:0.35)$)
    to[out=270,in=45] ($(2)+(315:0.25)$)
    to[out=225,in=0] ($0.5*(2)+0.5*(3)+(270:0.35)$)
    to[out=180,in=315] ($(3)+(225:0.25)$)
    to[out=135,in=270] ($0.5*(0)+0.5*(3)+(180:0.35)$)    
    to[out=90,in=225] ($(0)+(135:0.25)$);
\end{tikzpicture}
\end{center}
\end{multicols}
\caption{The Fifteen Partitions of Order 4}
\label{fig:4-partition}
\end{figure}
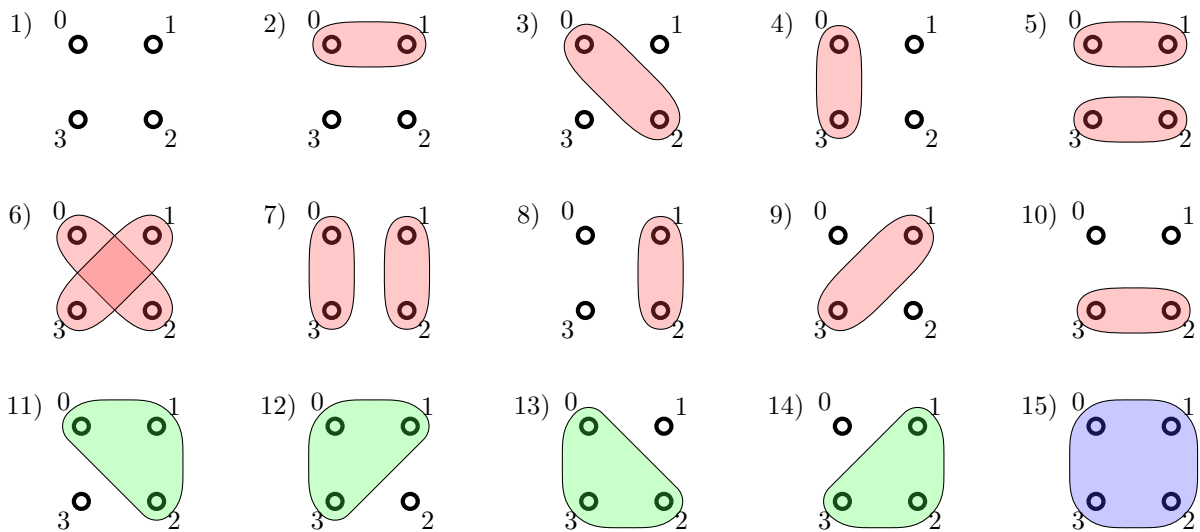

\begin{table}[p]
\hspace{-35pt}
\begin{tabular}{cccc|cccc|cccc}
\multicolumn{2}{l}{$G$} & Schur Rings  & $\Omega$ & \multicolumn{2}{l}{$G$} & Schur Rings & $\Omega$ & \multicolumn{2}{l}{$G$} & Schur Rings & $\Omega$\\
\hline
000 & $\O_4$ & \resizebox{45pt}{\height}{1, 8, 9, 10, 14} & 5 & 042 &$\z_{3,1}^e$& 1 & 1 & 084 &$\CH_2\s\O_2$& 1 & 1\\
001 &$\O_2\u\z_{3,1}$& 1, 10 & 2 & 043 &$\O_2^e\u\CH_2$& 1, 10 & 2 & 085 &$\CH_4$ & 1 & 1\\
002 &$\O_3\u\CH_2$& 1, 8 & 2 & 044 &\resizebox{45pt}{\height}{$(\O_2\u\CH_2)^e$}& 1 & 1 & 086 &\resizebox{45pt}{\height}{$\CH_2\s\LO_2$}& 1, 10 & 2\\
003 &$\O(2,2,2)$& 1, 10 & 2 & 045 &\resizebox{45pt}{\height}{$\K_{1,2}\r\theta (\{1\},\{2\})$}& 1, 10 & 2 & 087 &\resizebox{45pt}{\height}{$\CH_2\s\z_2$}& 1, 10 & 2\\
004 &$\O(2,2,3)$& 1, 10 & 2 & 046 &$\OROP_2$& 1, 10 & 2 & 088 &$\LO_2^{\theta e}$& 1, 9 & 2\\
005 &\resizebox{45pt}{\height}{$\O_2\u\ORO_{2,1}$}& 1, 10 & 2 & 047 &$\ORO_{2,1}^e$& 1, 8 & 2 & 089 &$\LO_3^\theta$& \resizebox{45pt}{\height}{1, 8, 9, 10, 14} & 5\\
006 &$\ORO_{2,1}[1]$& 1 & 1 & 048 &$\ORO_{2,1}^\text{op} \t \z_1$ & 1, 9 & 2 & 090 &$\z_2[e]^\theta$& 1 & 1\\
007 &$\ORO_{3,1}$& \resizebox{45pt}{\height}{1, 8, 9, 10, 14} & 5 & 049 & \resizebox{76pt}{\height}{$\ORO_{2,1}\r\theta (\ORO_{2,1},\z_1)$} & 1, 9 & 2 & 091 &$\LO_2\s \O_2$& 1, 4 & 2\\
008 &$\O_3\t\z_1$& 1, 8 & 2 & 050 &$\O_2^e\t\z_1$& 1 & 1 & 092 &$\z_2^{\theta e}$& 1, 9 & 2\\
009 &$\O(2,2,6)$& 1, 10 & 2 & 051 &\resizebox{45pt}{\height}{$(\O_2\t\z_1)^e$}& 1 & 1 & 093 &\resizebox{45pt}{\height}{$\LO_2\s\CH_2$}& 1, 4 & 2\\
010 &$\O(2,2,7)$& 1, 10 & 2 & 052 & \resizebox{70pt}{\height}{$(\CH_2\t\z_1) \r2 (\emptyset,\z_1)$} & 1, 6, 9 & 3 & 094 &\resizebox{45pt}{\height}{$\LO_2\s\LO_2$}& 1, 4, 7, 8 & 4\\
011 &$\z_{4,1}$& 1 & 1 & 053 &\resizebox{45pt}{\height}{$\LOO(2,1,1,1)$}& 1 & 1 & 095 &$\z_{2,2}^\theta$& 1 & 1\\
012 &$\O_2\u\O_2^e$& 1 & 1 & 054 &$\CH_2\s\O_2$& 1 & 1 & 096 &\resizebox{45pt}{\height}{$\LO_2\s\RO_2$}& 1, 4, 7, 8 & 4\\
013 &$\OLO_{2,2,1}$& \resizebox{45pt}{\height}{1, 9, 10, 14} & 4 & 055 &$\O_2\s\CH_2$& 1 & 1 & 097 &$\LO_2\s\z_2$& 1, 4, 7, 8 & 4\\
014 &\resizebox{45pt}{\height}{$\ORO_{2,1}\r\theta \z_1$}& 1, 9 & 2 & 056 &$\O_2\s\LO_2$& 1, 10 & 2 & 098 &$\z_3^\theta$& 1, 10, 14 & 3\\
015 &$\O(2,2,9)$& 1, 10 & 2 & 057 &$\O_2\s\z_2$& 1, 10 & 2 & 099 &$\LOZ_2$& 1, 3, 6, 9 & 4\\
016 &\resizebox{45pt}{\height}{$\z_{3,1}\u\CH_2$}& 1 & 1 & 058 &$\K_{1,3}$& \resizebox{45pt}{\height}{1, 8, 9, 10, 14} & 5 & 100 &$\LO_3^e$& \resizebox{45pt}{\height}{1, 3, 4, 10, 12} & 5\\
017 &$\O(2,2,11)$& 1, 10 & 2 & 059 &$\K_{1,2}\t\z_1$& 1, 8 & 2 & 101 & $\LO_4$ & 1--15 & 15\\
018 &$\z_{3,1}\t\z_1$& 1 & 1 & 060 &$\O_2^\theta\u\CH_2$& 1 & 1 & 102 &$\z_2[e,e]$& 1, 8 & 2\\
019 &$\z_{3,1}[z]$& 1, 10 & 2 & 061 &\resizebox{45pt}{\height}{$\CH_2\u\CH_3$}& 1, 9 & 2 & 103 &
$\z_{3,1}^\mu$& 1 & 1\\
020 &$\O_2\u \K_{1,2}$& 1, 10 & 2 & 062 &\resizebox{45pt}{\height}{$\CH_2\u\LO_3^\theta$}& 1, 10 & 2 & 104 &$\z_2[e,\hat{e}]$& 1 & 1\\
021 &\resizebox{45pt}{\height}{$\ORO_{2,1}\u\CH_2$}& \resizebox{45pt}{\height}{1, 9, 10, 14} & 4 & 063 &$\CH_2\u\z_2^\theta$& 1, 10 & 2 & 105 &
$\ORO_{2,1}^\mu$ & 1, 8 & 2\\
022 &\resizebox{45pt}{\height}{$(\O_2\u\CH_2)\t\z_1$}& 1 & 1 & 064 &$\LO_2^e[\hat{0}]$& 1, 8 & 2 & 106 &$\z_2[e]^e$& 1 & 1\\
023 &$\O_2\times\CH_2$& 1, 9 & 2 & 065 &\resizebox{45pt}{\height}{$\CH_2\t\CH_2$}& 1, 6, 9 & 3 & 107 &$\z_2[\hat{e},\hat{e}]$& 1, 8 & 2\\
024 &$\O_2\u\CH_3$& 1 & 1 & 066 &\resizebox{45pt}{\height}{$\LO_2\t\CH_2$}& 1, 10 & 2 & 108 &$\z_2\s\O_2$& 1, 4 & 2\\
025 &$\O_2\u\LO_2^\theta$& 1, 10 & 2 & 067 &\resizebox{45pt}{\height}{$\CH_2\times\CH_2$}& 1, 8 & 2 & 109 &$\z_2\s\CH_2$& 1, 4 & 2\\
026 &\resizebox{45pt}{\height}{$\ORO_{2,1}\rhat\theta \z_1$}& 1, 9 & 2 & 068 &$\O_2^\theta\t\z_1$& 1 & 1 & 110 &$\z_2\s\LO_2$& 1, 4, 7, 8 & 4\\
027 &$\O_2\u\z_2^\theta$& 1, 10 & 2 & 069 &$\CH_3\t\z_1$& 1 & 1 & 111 &$\z_2\s\z_2$& 1, 4, 7, 8 & 4\\
028 &$\LO_2^e[0]$& 1 & 1 & 070 &\resizebox{45pt}{\height}{$(\CH_2\t\z_1)^e$}& 1 & 1 & 112 &$\O_2\times\z_2$& 1, 6, 9 & 3\\
029 &\resizebox{45pt}{\height}{$\ORO_{2,1}\t\z_1$}& 1, 8 & 2 & 071 &$\LO_2^\theta\t\z_1$& 1, 9 & 2 & 113 &$\z_{3,2}$& 1 & 1\\
030 &$\OO_{2,2}$
& 1 & 1 & 072 &\resizebox{45pt}{\height}{$\CH_2\times\LO_2$}& 1, 6, 9 & 3 & 114 &$\z_2^e[z]$& 1 & 1\\
031 &$\OCH_{2,2}$
& 1 & 1 & 073 &
$\RO_2^\theta\t\z_1$& 1,9 & 2 & 115 &\resizebox{45pt}{\height}{$\z_2^e \r z (\emptyset,\z_1)$}& 1, 6, 9 & 3\\
032 &$\OLO_{2,2}$& 1, 10, 14 & 3 & 074 &$\z_2^\theta\t\z_1$& 1, 9 & 2 & 116 &$\z_{2,2}^e$& 1 & 1\\
033 &$\ORO_{2,2}$& \resizebox{45pt}{\height}{1, 8, 9, 10, 14} & 5 & 075 & $\LORO(2,1,1,0)$ 
&1 & 1 & 117 &$\CH_2\times\z_2$& 1, 3, 6, 9 & 4\\
034 &$\OZ_{2,2}$
& 1, 10, 14 & 3 & 076 &$\O_3^\theta$& 1, 10 & 2 & 118 &\resizebox{45pt}{\height}{$\LO_2\times\RO_2$}& \resizebox{45pt}{\height}{1, 5, 6, 7, 15} & 5\\
035 &\multicolumn{2}{l}{\resizebox{76pt}{\height}{$\ORO_{2,1}\r\theta (\z_1,\ORO_{2,1})$} 1, 8} & 2 & 077 &$\z_{3,1}^\theta$& 1 & 1 & 119 &$\LO_2\times\z_2$& \resizebox{45pt}{\height}{1, 3, 5, 6, 7, 9, 15} & 7\\
036 &$\O_2\times\LO_2$& 1, 6, 9 & 3 & 078 &\resizebox{45pt}{\height}{$(\O_2\u\CH_2)^\theta$}& 1 & 1 & 120 &$\z_3[e]$& 1, 10 & 2\\
037 &$\O_2\t\CH_2$& 1 & 1 & 079 &$\ORO_{2,1}^\theta$& 1, 10 & 2 & 121 &$\z_3^e$& 1, 10, 13 & 3\\
038 &\resizebox{45pt}{\height}{$\LO_2\t\O_2$}& 1, 10 & 2 & 080 &$(\O_2\t\z_1)^\theta$& 1 & 1 & 122 &$\z_{2,2}[z]$& 1, 10 & 2\\
039 &$\O_2^e[1]$& 1 & 1 & 081 &$\O_2^{\theta e}$& 1 & 1 & 123 &$\z_{2,3}$& 1 & 1\\
040 &\resizebox{45pt}{\height}{$\O_2^e\r1 (\emptyset,\z_1)$}& 1, 10 & 2 & 082 &$\K_{1,2}^\theta$& 1, 10 & 2 & 124 &$\text{V}_4$ & \resizebox{45pt}{\height}{1, 5-10, 14, 15} & 9\\
041 &$\O_3^e$& 1, 8 & 2 & 083 &\resizebox{45pt}{\height}{$(\CH_2\t\z_1)^\theta$}& 1 & 1 & 125 &$\z_4$ & \resizebox{45pt}{\height}{1, 5, 10, 14, 15} & 5
\end{tabular}
\caption{Schur rings of semigroups of order $4$}
\label{fig:sgsr4}
\end{table}

We see a similar statistical distribution as the semigroups of order $3$ in Table \ref{fig:sgsr4}. The mean number of Schur rings over semigroups of order $4$ is $\mu=2.29$, with a standard deviation of $\sigma=1.76$. Additionally, $76.2\%$ of the semigroups have only 1 or 2 Schur rings (meaning 23.8\% are Schur-dense), and $91.2\%$ have a number of Schur rings less than $\B(3)=5$. 
The only semigroups having $5$ Schur rings are 000, 007, 033, 058, 089, 100, 118, and 125, which are isomorphic to $\O_4$, $\ORO_{3,1}$, $\ORO_{2,2}$, $\K_{1,3}$, $\LO_3^\theta$, $\LO_3^e$, $\LO_2\times \RO_2$, and $\z_4$, respectively.

All but 118 and 125 are predicted by \thmref{bnminus}. The semigroup $\LO_2\times \RO_2$ is only purely rectangular band\footnote{Recall that a \emph{band} is a semigroup where all elements are idempotent. A \emph{rectangular band} $G$ is a band with the property that $xyz = xz$ for $x,y,z\in G$. A \emph{purely rectangular band} is a rectangular band not equivalent to $\LO_n$ for any $n$.} of order $4$. Since $\Aut(\LO_n\times \RO_m) \cong S_n\times S_m$ we know the number of automorphic Schur rings over $\LO_n\times \RO_m$ is at least $\B(n)\B(m)$. There are also diagonal subgroups of $S_n\times S_m$ which afford automorphic Schur rings too. In the case of $\LO_2\times\RO_2$, all five Schur rings are automorphic and correspond to the five subgroups of $S_2\times S_2\cong V_4$. The only semigroups obtaining more than $\B(3)$ Schur rings are 101, 119, and 124, which are isomorphic to $\LO_4$, $\LO_2\times \z_2$, and $V_4$, respectively.

\subsection{Semigroups of Order $5$}
We observe that there are $1160$ semigroups of order $5$ and $\B(5)=52$. Therefore, we cannot reasonably provide a comprehensive list of all the Schur rings over these semigroups. Instead, we provide a summary of the GAP calculations. Of the total, 65.8\% have only 1 or 2 Schur rings, and only 4.7\% are Schur-dense (having 6 or more Schur rings). Observe that $924$  semigroups (79.7\%) have a Bell number of Schur rings. Many of these semigroups are predictable; in particular, $\B(5)=52$ will only be attained by $\LO_5$ and the $7$ semigroups which have $\B(4)=15$ Schur rings are given by Theorem \ref{bnminus}. In Table \ref{fig:sgsr5}, we use $\Omega$ to denote the number of Schur rings and $G$ for the number of semigroups which obtain $\Omega$ Schur rings.

We will refer to semigroups by their GAP identification while making some comments. Several semigroups obtain a product of Bell numbers, which is not a surprise. Semigroups 1007, 1008, 1149, and 1152 are $\z_2 \s \z_3$, $\z_3 \s \z_2$, $\LO_2\s\z_3$, and $\z_3\s\LO_2$, respectively, where Theorem \ref{thm:stackschur} predicts the amount of Schur rings. Of interest are semigroups 314 and 315, being $\OLO_{3,2}$ and $\ORO_{3,2}$, which despite their similarity have different numbers of Schur rings ($7$ and $15$, respectively) as indicated in \thmref{thm:oro} and \corref{cor:olo}. Semigroups 738 and 745 ($\OLO_{2,3}$ and $\ORO_{2,3}$) also illustrate this point ($6$ and $15$ Schur rings). Other Schur-dense semigroups ($148$, $149$, $1125$, $1140$, $671$, and $672$) come from one element extensions of order 4 Schur-dense semigroups ($V_4^e$, $V_4^\theta$, $\LO_4^\theta$, $\LO_4^e$, $(\LO_2\times \z_2)^e$, and $(\LO_2\times \z_2)^\theta$, respectively). Furthermore, unites of Schur-dense semigroups often produce Schur-dense semigroups (\thmref{thm:thetastackrings}); e.g. $297$, $331$, and $672$, which are $\ORO_{2, 1}\u\ORO_{2,1}$, $\OLO_{2,1}\u\ORO_{2,1}$, and $\z_2^\theta\u \z_2^\theta$, respectively. 

\begin{table}[!ht]
\begin{center}
\begin{tabular}{cc|cc|cc|cc}
$\Omega$ & $G$ & $\Omega$ & $G$ & $\Omega$ & $G$ & $\Omega$ & $G$ \\
\hline
$1$ & $246$ & $4$ & $117$ & $7$ & $12$ & $10$ & $19$\\
 $2$ &$517$ & $5$ & $153$ & $8$ & $2$  & $15$ & $7$\\
 $3$ &$84$  & $6$ & $9$   & $9$ & $3$  & $52$ & $1$
\end{tabular}
\caption{Summary of results for Schur rings of semigroups of order $5$}
\label{fig:sgsr5}
\end{center}
\end{table}

Considering the distribution, the mean number of Schur rings over semigroups of order 5 is $\mu=2.75$, 
with a standard deviation of $\sigma=2.35.$ 
The most common number of Schur rings afforded a semigroup of this order is $2$, with $517$ semigroups admitting a single non-discrete Schur ring.

%

\subsection{Semigroups of Order $6$}
We observe that there are 15973 semigroups of order $6$ and $\B(6)=203$. We must again offer a brief summary of the GAP calculations, given in the table. Of the total, 46.0\% have only 1 or 2 Schur rings, and only 0.4\% are Schur-dense (having 16 or more Schur rings). Observe that 11636 semigroups (72.8\%) have a Bell number of Schur rings. There are two groups of order $6$, namely $S_3$ and $\z_6$. The symmetric group $S_3$ (4337) is the only semigroup of this order to attain $45$ semigroup Schur rings. The cyclic group $\z_6$ (14996) has $11$ semigroup Schur rings, a property shared by $11$ other semigroups. We again mention some of the Schur-dense semigroups, in addition to $S_3$.  The semigroup $\LO_3\times \z_2$ (9978) is the only semigroup with $31$ Schur rings, and $\ORO_{2,1}\u \ORO_{2,1} \u \CH_2$ (7066) is the only semigroup to have $22$. Rosters of Schur-dense semigroups are often Schur-dense themselves. The only two semigroups with $16$ Schur rings (3353 and 12431) are rosters, including $\OLO_{2,4}$. The only two semigroups with $17$ Schur rings (1002 and 7714) are rosters.

While not considered Schur-dense, the two semigroups with $13$ Schur rings (7273 and 15874) are $\OLO_{3,3}$ and $\LO_3\times \RO_2$, the latter being the unique purely rectangular band of order $6$. 



\begin{table}[!ht]
\begin{center}
\begin{tabular}{cc|cc|cc|cc|cc|cc|cc}
$\Omega$   & $G$ & $\Omega$   & $G$ & $\Omega$  & $G$ & $\Omega$   & $G$ & $\Omega$ & $G$ & $\Omega$ & $G$ & $\Omega$ & $G$\\
\hline
1  &2093 &  5  &2443 &  9  &46  &  13 &2    &  17 &2  &  25 &5  &  52 & 8 \\
2  &5259 &  6  &400  &  10 &292 &  14 &13   &  18 &18 &  30 &9  &  203 & 1 \\
3  &916  &  7  &215  &  11 &12  &  15 &1832 &  20 &10 &  31  & 1  \\
4  &2202 &  8  &173  &  12 &17  &  16 &2   &  22 &1  &  45  & 1  \\
\end{tabular}
\caption{Summary of results for Schur rings of semigroups of order $6$}
\label{fig:sr6}
\end{center}
\end{table}

Considering the distribution, the mean number of Schur rings over semigroups of order 6 is $\mu=4.67$, 
with a standard deviation of $\sigma=4.70$. 
The most common number of Schur rings afforded a semigroup of this order is again $2$, with $5259$ semigroups admitting a single non-discrete Schur ring. See Table \ref{fig:sr6} for more details.

\subsection{Semigroups of Order $7$}
We observe that there are 836021 semigroups of order $7$ and $\B(6)=877$. Of the total, 10.8\% have only 1 or 2 Schur rings, and only 0.006\% are Schur-dense (having 53 or more Schur rings). Observe that 694279 semigroups (83.0\%) have a Bell number of Schur rings. There is one group of order $7$, namely $\z_7$ (836017), which has five Schur rings. The only semigroup with $55$ Schur rings (724285) is $\ORO_{3,2}\u \RO_2^\theta$. The semigroup $(\O_3^e[\hat{\theta}, \theta])^e$ (546872) is the lone semigroup with $57$ Schur rings, and $\ORO_{2,1}\u\ORO_{2,1}\u K_{1,2}$ (723804) is the only semigroup to have $65$ Schur rings. Among semigroups which produce an uncommon number of Schur rings, rosters are again common. For example, one of the three semigroups with $36$ Schur rings (546862) is $\OLO_{3,2,3}$, as well as the only semigroup with $54$ Schur rings (730772), $\OLO_{2,2,4}$.

\begin{table}[!ht]
\begin{center}
\begin{tabular}{cc|cc|cc|cc|cc}
$\Omega$    & $G$ & $\Omega$     & $G$ & $\Omega$  & $G$ & $\Omega$     & $G$ & $\Omega$ & $G$  \\
\hline
1  &22667 &  12 &680    &  23 &7   &  34 &2      &  55  & 1  \\
2  &67360 &  13 &73     &  24 &10  &  35 &11     &  57  & 1  \\
3  &12394 &  14 &471    & 25 & 265 &  36 &3      &  60  & 9  \\
4  &40933 &  15 &423748 &  26 &8   &  40 &1      &  65  & 1  \\
5  & 30789 & 16 &45356  &  27 &6   &  41 &3      &  67  & 6  \\
6  &8812  &  17 &9503   &  28 &6   &  44 &2      &  75  & 7  \\
7  &5338  &  18 &836    &  29 &8   &  45 &9      &  77  & 1  \\
8  &5551  &  19 &89     &  30 &328 &  50 &8      &  82  & 1  \\
9  &1077  &  20 &823    &  31 &8   &  52 &149705 &  104  & 9 \\
10 &8846  &  21 &25     &  32 &5   &  53 &2      &  203  & 9 \\
11 &125   &  22 &72     &  33 &9   & 54 & 1      &  877  & 1
\end{tabular}
\caption{Summary of results for Schur rings of semigroups of order $7$}
\label{fig:sr7}
\end{center}
\end{table}

Considering the distribution, the mean number of Schur rings over semigroups of order 7 is $\mu = 18.96$, 
with a standard deviation of $\sigma = 16.24$. 
Unlike the previous orders, the mode is at $\B(4) = 15$, with 423748 semigroups.  See Table \ref{fig:sr7} for more details.

\section{Conclusion}\label{sec:questions}
This paper has studied Schur rings over semigroups, which are those partitions of the semigroup which afford combinatorial subrings of the associated semigroup algebra. More broadly speaking, Schur rings are those partitions of the semigroup which behave like subsemigroups. Subsemigroups are those subsets of a semigroup which are closed under multiplication, that is, they are the algebraic subsets. Likewise, homomorphisms are those functions between semigroups which are algebraic, that is, functions which preserve multiplication. Our position here is that Schur rings are the algebraic partitions of a semigroup in similar manner. Two specific partitions on semigroups of note that were not considered in this paper are kernels of homomorphisms and the Green relations. Of future interest will be how these partitions relate to the Schur rings of semigroups. Some natural questions are immediate, such as the fact that all five\footnote{Really four, since $\mathcal{D}=\mathcal{J}$ for finite semigroups.} Green relations on a finite nilpotent semigroup are discrete, and hence a Schur ring.

As Bell numbers count partitions of sets, we have seen that Bell numbers are significant in the counting of algebraic partitions of semigroups, with many families demonstrating a Bell number of Schur rings, such as $\LO_n$, $\O_n$, $\K_{1,n}$, $\ORO_{n,m}$, and $\OCH_{n,m}$. Examining all Schur rings up to order 7, we see $82.2\%$ of semigroups have a Bell number of Schur rings. In the distribution of Schur rings number, we clearly see a clustering around Bell numbers. For these semigroups, we recognize broad flexibility of Schur rings, an algebraic property of the semigroup. Conversely, we have also seen great rigidity where specific elements or subsets must necessarily be isolated in every Schur ring, such as $G^e$, $G^\mu$, $G^\theta$, or $G\t \z_1$. In extreme cases, the semigroup might be so rigid to only exhibit $1$ or $2$ Schur rings, such as $\CH_n$ and $\z_{n,m}$. Combining these observations, we see that specific semigroups have algebraic zones of rigidity and zones of flexibility, which the Schur rings measure. Given that for some families of semigroups the number of Schur rings grows multiplicatively, such as $G\s H$, $\LO_n\t G$, $\OO_{n,m}$, or $\OROP_n$, it is clear why products of Bell numbers likewise see clustering of semigroups, that is $\Omega(G) $ being $10=2\cdot 5$, $25 = 5\cdot 5$, $30=2\cdot 15$ is a common behavior for semigroups of order up to $7$. We have also seen many almost-multiplicative constructions on semigroups that influence their number of Schur rings, meaning that most of the Schur rings are accounted for by multiplicative growth but a few others may be present (or missing) for other reasons, such as $G\t H$, $G\u H$, $\LOZ_n$, $\OLO_{n,m}$, or $\OZ_{n,m}$.

Because of computational efficiency and the sheer volume of semigroups of order 8, this paper only statistically analyzes semigroups up to order 7, despite the entire catalog of order 8 semigroups being within the \emph{smallsemi} package. Of future interest will be an attempt to analyze these Schur rings over semigroups of order 8. While not contained within \emph{smallsemi}, semigroups of orders 9 and 10 have been enumerated \cite{Distler}. Of future interest would be any attempt to analyze the Schur ring distribution of any order beyond 8. Dwarfing even the number of semigroups of order 8, a census approach may be unlikely, but utilizing sampling methods and other statistical analyses may be able to still provide sufficient insights.

From the data we have been able to analyze, asymptotically $\mu\sim \sigma \sim \B(n-3)$. This distribution could be explained by the clustering around Bell numbers as mentioned above, especially around 1, 2, $\B(n-3)$, and $\B(n-2)$ many Schur rings. As mentioned above, as the order increases the percentage of nilpotent semigroups, especially a majority of semigroups being $3$-nilpotent, likewise increases. The indecomposable imposition puts an enormous restriction on Schur rings of nilpotent semigroups, there is hope that statistical methods can provide an asymptotic approximation of this Schur ring distribution.

As observed previously, it is quite common for semigroups to have a single Schur ring, necessarily the discrete Schur ring. As the authors have explored the Schur rings over various families of semigroups, we have observed that classifying the coarsest Schur ring over a semigroup is not always obvious nor straightforward. For example, in preparation of this paper, the authors have observed that the coarsest Schur ring over a random, finite semilattice can vary dramatically from another random semilattice. Since attempting to classify just the coarsest Schur ring over a semigroup $G$ led to so many insights on the classification of all the Schur rings over $G$, identifying the coarsest Schur ring is the first question to ask when studying Schur rings over a specific semigroup. Therefore, the classification of the coarsest Schur ring over specific semigroup families will be a fruitful tree for future research on this subject (see, for example, Theorem \ref{thm:nilriisr}). 

The converse question, classifying semigroups which obtain a specific coarsest Schur ring will also be of interest, e.g. semigroups whose coarsest Schur ring is discrete, indiscrete, trivial, or which only have two Schur rings, the discrete and coarsest Schur rings. Additionally, we propose the question of which semigroups have a Schur ring for every refinement of its coarsest Schur ring, which will necessarily result in the product of Bell number of Schur rings seen in \corref{cor:omegabound}, similar to Theorems \ref{leftnullconverse} and \ref{bnminus}. When all but one class is a singleton, the number of Schur rings will be a Bell number. As such, we ask which semigroups have as the coarsest Schur ring a partition having exactly two or three singletons and all refinements are Schur rings, the classification of which would enlighten the clustering of semigroups at $\B(n-2)$ and $\B(n-3)$.

We have seen many examples of ideal extensions of semigroups and their connections to Schur rings. Consider, for example, stacks of semigroups as an ideal extension. Just like $G$ is an ideal of the semigroup $G\s H$, the Schur ring $\S$ is an ideal in the ring $\S\s\T$. In fact, every Schur ring over $G\s H$ is an ideal extension of $\S$ by $\T$, where $\S$ and $\T$ Schur rings over $G$ and $H^\theta$, respectively. By \thmref{cor:thetaschur}, we may identify $\T$ with a Schur ring over $H$. Furthermore, the compatibility criterion for such ideal extensions of Schur rings is always trivial, that is, $\S$ can be any Schur ring over $G$, not just the discrete partition. By contrast, while the two Schur rings in Example \ref{exam:numerical} are both ideal extensions, they both require the partition on the ideal be discrete, even though this ideal has non-discrete Schur rings. These examples compared side-by-side illustrate the wide gap of possibilities for the partition of the ideal in these ideal extension Schur rings, further suggesting the need for deeper consideration.

We saw that recursive indecomposable imposition gives a sequence of ideal extensions which imposes bounds on the coarsest Schur ring. The ideal $G^k$ inherited restrictions recursively from its complement. Can these restrictions be predicted systematically from the Cayley table of the semigroup? Which non-discrete partitions on $G^\infty$ are compatible with the selected partition on its complement?

Ideal extensions and Rees quotients are a special case of semigroup extensions. Returning to the question about kernels of homomorphisms, wedge products of group Schur rings lift Schur rings from the quotient group via cosets. The partition on the kernel may be selected so long as it is ``compatible'' with the partition of the quotient. The articulation of this compatibility criterion for wedge products was the key piece regarding the classification of Schur rings over finite cyclic groups \cite{LeungII, LeungI}. We anticipate that a similar formalization of compatibility criteria for the semigroup ``wedge'' products discussed here will likewise prove potent and deserves further investigation.

We present one more question. This paper began with taking a group and relaxing its structure into a semigroup and observing an increase of Schur rings. If we view Schur rings as the algebraic partitions of a group, it is expected that there would be more algebraic partitions when we weaken the axioms of the algebra. We would expect a group to have fewer Schur rings as a group as it would as a monoid or inverse semigroup, as it would also as a semigroup. But we observed that zero semigroups have the exact same Schur rings as the weaker semigroup structure. Where else does this switch in algebraic categories result in the same number of Schur rings? As another example, as a semigroup $G^e$ has the same number of Schur rings as it's monoid structure. For small abelian groups, we observed that the number of monoid Schur rings is identical to its number of group Schur rings. We conjecture that this holds for all finite abelian groups.

We hope these results and questions will excite future study of Schur rings over semigroups, which may elucidate algebraic structure of semigroups not presently examined. Wielandt's ``Method of Schur,'' which has proven helpful in the study of groups, now appears equally useful in the study of semigroups. Perhaps this ``Method of Schur'' could be applied in other algebraic categories. 

\begin{table}[h!]
\begin{center}
\begin{tabular}{cc|cc|cc|cc|cc|cc}
Fors. & GAP & Fors. & GAP & Fors. & GAP & Fors. & GAP & Fors. & GAP & Fors. & GAP\\
\hline
000 & 1  & 021 & 54 & 042 & 40  & 063 & 92  & 084  & 86  & 105 & 27  \\
001 & 8  & 022 & 55 & 043 & 67  & 064 & 102 & 085  & 114 & 106 & 28  \\
002 & 14 & 023 & 42 & 044 & 68  & 065 & 103 & 086  & 115 & 107 & 88  \\
003 & 2  & 024 & 56 & 045 & 69  & 066 & 104 & 087  & 94  & 108 & 49  \\
004 & 9  & 025 & 57 & 046 & 70  & 067 & 105 & 088  & 116 & 109 & 89  \\
005 & 15 & 026 & 58 & 047 & 71  & 068 & 85  & 089  & 117 & 110 & 90  \\
006 & 16 & 027 & 43 & 048 & 72  & 069 & 106 & 090  & 29  & 111 & 50  \\
007 & 17 & 028 & 59 & 049 & 73  & 070 & 107 & 091  & 87  & 112 & 5   \\
008 & 18 & 029 & 60 & 050 & 74  & 071 & 108 & 092  & 91  & 113 & 13  \\
009 & 3  & 030 & 44 & 051 & 75  & 072 & 109 & 093  & 118 & 114 & 30  \\
010 & 10 & 031 & 61 & 052 & 76  & 073 & 110 & 094  & 119 & 115 & 31  \\
011 & 11 & 032 & 62 & 053 & 77  & 074 & 93  & 095  & 33  & 116 & 32  \\
012 & 19 & 033 & 63 & 054 & 47  & 075 & 111 & 096  & 120 & 117 & 51  \\
013 & 20 & 034 & 45 & 055 & 78  & 076 & 25  & 097  & 95  & 118 & 123 \\
014 & 21 & 035 & 64 & 056 & 79  & 077 & 41  & 098  & 124 & 119 & 52  \\
015 & 34 & 036 & 46 & 057 & 48  & 078 & 80  & 099  & 96  & 120 & 97  \\
016 & 38 & 037 & 65 & 058 & 98  & 079 & 81  & 100 & 121 & 121 & 125 \\
017 & 35 & 038 & 66 & 059 & 99  & 080 & 82  & 101 & 122 & 122 & 6   \\
018 & 39 & 039 & 22 & 060 & 84  & 081 & 83  & 102 & 4   & 123 & 126 \\
019 & 36 & 040 & 23 & 061 & 100 & 082 & 112 & 103 & 12  & 124 & 7   \\
020 & 53 & 041 & 24 & 062 & 101 & 083 & 113 & 104 & 26  & 125 & 37 \\
\\
\\
\\
GAP & Fors. & GAP & Fors.  & GAP & Fors.  & GAP & Fors. & GAP & Fors.  & GAP & Fors.  \\
\hline
1   & 000      & 22  & 039  & 43  & 027  & 64  & 035 & 85  & 068  & 106 & 069  \\
2   & 003      & 23  & 040  & 44  & 030  & 65  & 037 & 86  & 084  & 107 & 070  \\
3   & 009      & 24  & 041  & 45  & 034  & 66  & 038 & 87  & 091  & 108 & 071  \\
4   & 102      & 25  & 076  & 46  & 036  & 67  & 043 & 88  & 107  & 109 & 072  \\
5   & 112      & 26  & 104  & 47  & 054  & 68  & 044 & 89  & 109  & 110 & 073  \\
6   & 122      & 27  & 105  & 48  & 057  & 69  & 045 & 90  & 110  & 111 & 075  \\
7   & 124      & 28  & 106  & 49  & 108  & 70  & 046 & 91  & 092  & 112 & 082  \\
8   & 001      & 29  & 090  & 50  & 111  & 71  & 047 & 92  & 063  & 113 & 083  \\
9   & 004      & 30  & 114  & 51  & 117  & 72  & 048 & 93  & 074  & 114 & 085  \\
10  & 010      & 31  & 115  & 52  & 119  & 73  & 049 & 94  & 087  & 115 & 086  \\
11  & 011      & 32  & 116  & 53  & 020  & 74  & 050 & 95  & 097  & 116 & 088  \\
12  & 103      & 33  & 095  & 54  & 021  & 75  & 051 & 96  & 099  & 117 & 089  \\
13  & 113      & 34  & 015  & 55  & 022  & 76  & 052 & 97  & 120  & 118 & 093  \\
14  & 002      & 35  & 017  & 56  & 024  & 77  & 053 & 98  & 058  & 119 & 094  \\
15  & 005      & 36  & 019  & 57  & 025  & 78  & 055 & 99  & 059  & 120 & 096  \\
16  & 006      & 37  & 125  & 58  & 026  & 79  & 056 & 100 & 061  & 121 & 100 \\
17  & 007      & 38  & 016  & 59  & 028  & 80  & 078 & 101 & 062  & 122 & 101 \\
18  & 008      & 39  & 018  & 60  & 029  & 81  & 079 & 102 & 064  & 123 & 118 \\
19  & 012      & 40  & 042  & 61  & 031  & 82  & 080 & 103 & 065  & 124 & 098  \\
20  & 013      & 41  & 077  & 62  & 032  & 83  & 081 & 104 & 066  & 125 & 121 \\
21  & 014      & 42  & 023  & 63  & 033  & 84  & 060 & 105 & 067  & 126 & 123
\end{tabular}
\end{center}
\caption{Correspondence between indexing of order 4 semigroups in Forsythe \cite{Forsythe} and GAP}
\label{apx:A}
\end{table}

\subsection*{Appendix A: \quad Notation}\label{appB}
All unions are assumed to be disjoint unions. If $G$ and $H$ are semigroups, multiplication on them is preserved, i.e. when defining a multiplication on $G\cup H$ we need only consider the definitions of $gh$ and $hg$, for $g\in G$ and $h\in H$.\\ 

\begin{hangparas}{15pt}{1}

$\CH_n$, the chain semigroup, is the set $\{1, \ldots, n\}$ under the operation $\max(\cdot, \cdot)$ and is introduced on page \pageref{ch}. It has $1$ Schur ring by \thmref{thm:chainschur}.

$G \times H$, the direct product of $G$ and $H$, is the set $\{(g, h) \mid g\in G, h\in H\}$ where $(g_1, h_1)(g_2, h_2)=(g_1g_2, h_1h_2)$ and is introduced on page \pageref{direct}.

$G^e$ is the set $G\cup\{e\}$ where $ge=eg=g$ for all $g\in G$ and is introduced on page \pageref{monoid}. Its Schur rings are described by \thmref{thm:monoidSchur}.


$G^\text{op}$, the opposite semigroup of $G$, requires knowledge of the operation of $G$. Let juxtaposition denote the operation of $G$ and $\cdot$ the operation of $G^\text{op}$. Then if $xy=z$ we have $y\cdot x=z$. It is introduced on page \pageref{Pageop}.

$G \r x (H,K)$ is the set $G\cup\{x^\prime\}$. We require that $H\leq G_x$ and $K\leq \, _x G$ where both $G-H$ and $G-K$ are prime ideals. Further, one of 
$$\begin{tikzpicture}
    \path (0, 0) node {(a)} (2, 0) node {(b)} (4, 0) node {(c)} (6, 0) node  {(d)} 
        (0, -0.5) node {$H\leq K$} (2, -0.5) node  {$K\leq H$}
        (4, -0.5) node {$x\in H\cap K$} (6, -0.5) node {$x\not\in H\cap K$}
        (0, -1) node {$x^\prime x$} (2, -1) node {$xx^\prime$} (4, -1) node {$x^\prime$} (6, -1) node {$x^2$};
\end{tikzpicture}$$
must be true. For $y\in H$ we have $yx^\prime=x^\prime$ and if $y\not\in H$ then $yx^\prime=yx$. If $z\in K$ then $x^\prime z= x^\prime$ but if $z\not\in K$ then $x^\prime z = xz$, and $(x^\prime)^2$ is determined by which of (a), (b), (c), or (d) are satisfied. This is called a roster on $G$, where $(x, H, K)$ is the roster of $G \r{x}(H, K)$, and is introduced on page \pageref{roster1}.

$G \r x H$ is the set $G\cup\{x^\prime\}$, and is a simplified notation for $G \r x (H, H)$. For $y\in H$ we have $yx^\prime=x^\prime$ and $x^\prime y = x^\prime$, but if $y\not\in H$ then $yx^\prime=yx$ and $x^\prime y = xy$, additionally $(x^\prime)^2=xx^\prime=x^\prime x$. We require $H\leq G_x$, $H\leq\ _x G$, and that $G-H$ is a prime ideal. It is introduced on page \pageref{roster2}.

$G \rhat x H$ is the set $G\cup\{x^\prime\}$. For $y\in H$ we have $yx^\prime=x^\prime$ and $x^\prime y = x^\prime$, but if $y\not\in H$ then $yx^\prime=yx$ and $x^\prime y = xy$, additionally $(x^\prime)^2$ is $x^\prime$ if $x^\prime x=x^2$ but $x^2$ if $x^\prime x = x^\prime$. We require $H\leq G_x$, $H\leq\ _x G$, and that $G-H$ is a prime ideal, $x^2=x$. It is introduced on page \pageref{rosterhat}.

$G\r{x,y} (H,K)$ is the set $(G\r x (H,K))\r y(H,K)$. This is introduced on page \pageref{rosterdual}.

$G\r X (H,K)$ is equal to $G\r{X-x,x} (H,K)$ for any $x\in X$. It is introduced on page \pageref{rinduction}.

$G \s H$ is the set $G\cup H$ where $gh=hg=g$ for all $g\in G$ and $h\in H$ and is introduced on page \pageref{stack}. It is called the stack of $G$ and $H$. The Schur rings of $G \s H$ are given by \thmref{thm:stackschur}.

$G \t H$ is the set $G\cup H$ where $gh=\theta_G$ and $hg=\theta_H$ for all $g\in G$, $h\in H$. Both $G$ and $H$ must be zero-semigroups. It is introduced on page \pageref{twist}. This is called the twist of $G$ and $H$. If $G=\LO_n$ see $\LO_n \t G$ below. Some of its Schur rings are described by \thmref{thm:veebarschur}.

$G^\theta$ is the set $G\cup\{\theta\}$ where $g\theta=\theta g = \theta$ for all $g\in G$ and is introduced on page \pageref{theta}. The Schur rings are given by \corref{cor:thetaschur}.

$G^{e\theta}=G^{\theta e} = (G^e)^\theta = (G^\theta)^e$.

$G \u H$ is given by $(G\t H)/\{\theta_G, \theta_H\}$. Multiplication is defined by $gh=hg=\theta$. Both $G$ and $H$ must be zero-semigroups. It is introduced on page \pageref{unite}. This is called the unite or unification of $G$ and $H$. If $G=\LO_n$ see $\LO_n \u G$ below. Some of the Schur rings are described by \corref{thm:thetastackrings}.

$G[x]$ is the set $G\cup \{x^\prime\}$ where $gx^\prime = gx$ and $x^\prime g = xg$, for $x\in G$ and is introduced on page \pageref{clone}. This process is called cloning, where $x^\prime$ is the clone of $x$. We can iterate this process, adjoining multiple clones, including multiple clones of the same elements. If $X=\{x_1, x_2, \ldots x_n\}$ is a multiset of elements of $G$, then $G[X]$ is the semigroup $G[x_1][x_2]\ldots[x_n]$. Schur rings can be constructed using clones, \thmref{thm:clonerings}.

$G[\hat{x}]$ is the set $G\cup \{x^\prime\}$ where $gx^\prime = gx$, $x^\prime g = xg$, and $(x^\prime)^2=x^\prime$, for $x\in G$, $x^2=x$ and is introduced on page \pageref{Iclone}. It is the same as cloning except that $x^\prime$, being idempotent, is called an idempotent clone.

$G^\mu$ is the set $G \cup \{\mu\}$ where $g\mu = \mu g = \mu$ for all $g\in G$, and $\mu^2 =\theta$. This is the roster $G\rhat{\theta} G$. It is introduced on page \pageref{dualclone}.



$\K_{1,n}$, the bipartite semilattice, is the set $\{\theta, 1, \ldots, n\}$ where $x^2=x$ and all other products are $\theta$ and is introduced on page \pageref{bipart}. It has $\B(n-1)$ Schur rings by \thmref{thm:bipartiteschur}.


$\LO_n$, the left-null semigroup, is the set $\{1, \ldots, n\}$ where $ab=a$ for all $a$, $b\in \LO_n$ and is introduced on page \pageref{lnull}. It is the semigroup $\RO_n^\text{op}$. It has $\B(n)$ Schur rings by \thmref{thm:leftnullschur}, and up to equivalence is defined by that fact, see \thmref{leftnullconverse}.

$\LOO(\ell, m, n, k)$ is the set $\{0, 1, \ldots, \ell+m+n-1\}$. We require $0\le k<\ell^{nm}$. Let $x$, $y\in\LOO(\ell, m, n, k)$, if $x<\ell+m$ then $xy=x$, otherwise if $x\geq \ell+m$ and if $y<\ell$ or $y\geq\ell+m$ then $xy=0$. If $x\geq \ell+m$ and $\ell\leq y < m$ then $k$ reports the behavior of the product by encoding how the $n\times m$ submatix of the Cayley table is to be filled in.  When $k$ is converted into a base $\ell$ number, it gives the entries of the $n\times m$ submatrix in standard reading order, that is left-to-right, top-to-bottom. Let $L=\{0,\ldots, \ell-1\}$, $M=\{\ell, \ldots, \ell+m-1\}$, and $N=\{\ell+m,\ldots, \ell+m+n-1\}$. Then $L\cong \LO_\ell$, $M\cong \LO_m$, $L\cup M\cong \LO_{\ell+m}$, $N\cup 0 \cong \O_{n+1}$, and $L\cup N\cong \LO_{\ell-1}\t \O_{n+1}$. By a change of $k$ if necessary, we may require that $NM=L$, ignoring multiplicities. It is introduced on page \pageref{loo}. We know $\B(m)\B(n) \le \Omega(\LOO(\ell,m,n,k))\le \B(\ell-1)\B(m)\B(n)$ by \thmref{thm:looschur}.

$\LORO(\ell,m,n,k)$ is the set $\{0,1,\ldots, \ell+m+n-1\}$. We define multiplication of $x$, $y\in \LORO(\ell, m, n, k)$ as $xy=x$ if $x<\ell+m$, otherwise if $x\geq \ell+m$ and one of $y< \ell$ or $y\geq \ell+m$, we have $xy=y$. The remaining $n\times m$ submatrix of the Cayley table is given by a number $k$ which, when converted into a base $\ell$ number, gives the entries of the submatrix in standard reading order. It is introduced on pages \pageref{semistack}, \pageref{loro}. By \thmref{thm:loroschur} $\B(m)\B(n) \le \Omega(\LORO(\ell,m,n,k))\le \B(\ell)\B(m)\B(n)$.

$\LO_n \shat G$ is the set $\LO_n \cup G$. We require $\varphi: G\to S_n$ to be a homomorphism. Then for $x\in \LO_n$ and $g\in G$ define $xg=x$ and $gx=\varphi_g(x)$, where $\varphi_g\colon \LO_n\rightarrow \LO_n$ is the automorphism associated to $g$. It is introduced on page \pageref{semistack}. It is called a semi-stack.

$\LO_n \t G$ is the set $\LO_n\cup G$ where for $g\in G$ and $x\in \LO_n$ $gx=\theta$ and $xg=x$. We require $\theta\in G$. It is introduced on page \pageref{semitwist}. This is called a semi-twist.

$\LO_n \u G$ is the set $(\LO_n \t G)/\{0, \theta_G\}$, where $G$ is a zero-semigroup, but as this equals $\LO_{n-1}\t G$ we do not use this ``semi-unite.'' It is mentioned in a footnote on page \pageref{semiunite}.

$\LOZ_n$ is the set $\LO_n\cup\z_n$ (see $\LOZ_{m,n}$, $\LO_n \shat G$) and is introduced on page \pageref{lozn}.

$\LOZ_{m,n}$ is the set $\LO_m \cup \z_n$ (see $\LO_n \shat G$).

$\LOZ_{m,n,k}$ is the set $\LO_m \cup \z_{n,k}$ (see $\LO_n \shat G$).

$\O_n$, the null semigroup, is the set $\{\theta, 1, \ldots, n-1\}$ where $ab=\theta$ for all $a$, $b$ and is introduced on page \pageref{null}. Every partition containing $\theta$ is a Schur ring, giving $\B(n-1)$ Schur rings by \thmref{thm:nullschur}.

$\O(m,n,k)$ is the set $\{0,1,\ldots, m+n-1\}$. Multiplication is defined for $x,y\in \O(m,n,k)$ by cases. If $x<m$ then $xy=yx=0$, if $x$, $y\geq m$ then $k$ reports the behavior of the product by encoding how the Cayley table is to be filled it. There is an $n\times n$ submatrix of the Cayley table which has not been determined which will consist of integers from $\{0, 1, \ldots, m-1\}$. If the $n^2$ entries of this matrix are, reading left-to-right and top-to-bottom, $k_{n^2-1},\ldots, k_2, k_1, k_0$, then $k = k_{n^2-1}(m^{n^2-1}) + \ldots + k_2(m^2)+ k_1(m) + k_0$. It is introduced on page \pageref{bigO}. By \thmref{thm:bigOschur} we know $\B(n) \le \Omega(\O(m,n,k))\le \B(m-1)\B(n)$.

$\OCH_{m,n}$ is the set $\O_m\cup\CH_n$ (see $\O G_n$) and is introduced on page \pageref{och}. It has $\B(m-1)$ Schur rings by \corref{cor:och}.

$\O G_n$ is the set $\O_n\cup G$ where for $x\in \O_n$, $g\in G$ we have $xg=\theta$, $gx=x$ and is introduced on page \pageref{og}. If $G$ is a finite zero-semigroup, then by \thmref{thm:och} we have $\Omega(OG_n)=\Omega(G)\B(n-1)$, otherwise see \thmref{thm:og}.

$\O G_{m, n}$ is introduced on page \pageref{ogtwo}. It is defined recursively by $\OG_{m,1} = \OG_m$ and $\OG_{m,n} = \OG_{m,n-1}\r\theta (\emptyset, G)$ for $n\ge 2$.

$\O \K_{m, n}$ is the set $\O_m \cup \K_{1,n}$, see $\O G_n$. It has $\B(m-1)\B(n-1)$ Schur rings.

$\OO_{m,n}$ is the set $O_m\cup\O_n$ (see $\O G_n$) and is introduced on page \pageref{oo}. By \corref{oo} it has $\B(n-1)\B(m-1)$ Schur rings.

$\OLO_{m,n}$ is the set $O_m\cup\LO_n$ (see $\O G_n$) and is introduced on page \pageref{olo}. By \corref{cor:olo} its Schur rings are given by the sum $\B(m-1)\B(n) + \sum_{k=1}^{m-1} \binom{m-1}{k}\B(m-k-1)$.

$\OLO_{\ell, m, n}$ is the set $\{\theta, 1, \ldots, \ell+m+n-1\}$. It is defined recursively on page \pageref{olothree} by $\ORO_{\ell, m, n} = \OLO_{\ell, m-1, n} \r{\theta} (\emptyset, \LO_n)$  where $\OLO_{\ell, 1, n}=\OLO{\ell, n}$. This semigroup has $\Omega(\OLO_{\ell,m,n}) = \B(\ell-1)\B(m+n-1) + \sum_{k=1}^{\ell-1} \dbinom{\ell-1}{k}(\B(\ell-k-1)\B(n-1)+1)$ Schur rings. We remark that $\OLO_{\ell, m, n} = \ORO_{m, \ell, n}$. 


$\ORO_{m,n}$ is the set $O_m\cup\RO_n$ (see $\O G_n$) and is introduced on page \pageref{oro}. By \thmref{thm:oro} it has $B(m+n-1)$ Schur rings.

$\ORO_{\ell, m,n}$ is the set $\{\theta, 1, \ldots, \ell+m+n-1\}$. It is first mentioned on page \pageref{orothree}. It is defined recursively by $\ORO_{\ell, m, n} = \ORO_{\ell, n} \r{\theta} (\ORO_{\ell, m-1,n}, \RO_n)$ where $\ORO_{\ell, 1, n}=\ORO{\ell, n}$. This semigroup has $\Omega(\ORO_{\ell,m,n}) = \B(\ell+n-1)\B(m-1) + \sum_{k=1}^{m-1} \binom{m-1}{k}\B(\ell-1)(\B(m-k-1)+1)$ Schur rings.

$\OROP_n$ is the set $O_n\cup \{\alpha,\beta\}$ where for $x$, $y\in \O_n$ we have $xy=\theta$, $\alpha x=\alpha$, $\alpha^2=\alpha$, $x\beta=x$, $\beta x =\theta$, $\beta^2=\beta$, $\alpha\beta=x_1$, $\beta\alpha=\alpha$ and is introduced on page \pageref{orop}. By \thmref{thm:orop} the number of Schur rings is $2\B(n-2)$.

$\OZ_{m,n}$ is the set $\O_m\cup\z_n$ (see $\O G_n$) and is introduced on page \pageref{oz2}.

$\OZ_{m, n, k}$ is the set $\O_m\cup \z_{n,k}$ (see $\O G_n$).

$\RO_m$, the right-null semigroup, is the set $\{1, \ldots, n\}$ where $ab=b$ for all $a$, $b\in\RO_n$ and is introduced on page \pageref{rnull}. It has the same number of Schur rings as $\RO_n^\text{op} = \LO_n$, being $\B(n)$.





$\z_n$, the cyclic group, is generated by $\langle z \mid z^n=e\rangle$ and is introduced on page \pageref{Cg}.


$\z_{m,n}$, the monogenic semigroup, is generated by $\langle z \mid z^m=z^n\rangle$ and is introduced on page \pageref{zmn}. For $m>1$ it has a single Schur ring by \thmref{thm:monogenic}. Observe that $\z_{1,m}=\z_n$.
\end{hangparas}

\bibliographystyle{plain}
\bibliography{Srings}
\end{document}